\documentclass[11pt]{amsart}
\usepackage{latexsym, mathrsfs, color, tikz, multirow,bbm,mathtools}
\usepackage{graphics}
\usepackage{subcaption}
\usepackage{xparse}
\input{xy}
\xyoption{all}

\usepackage[colorlinks=true, pdfstartview=FitV,
 linkcolor=black,citecolor=black,urlcolor=black]{hyperref}
 
\usepackage{tikz}

\numberwithin{equation}{section}
\theoremstyle{plain}
\newtheorem{Thm}[equation]{Theorem}
\newtheorem{Prop}[equation]{Proposition}
\newtheorem{Cor}[equation]{Corollary}
\newtheorem{Lem}[equation]{Lemma}
\newtheorem{Conj}[equation]{Conjecture}

\theoremstyle{definition}
\newtheorem{Def}[equation]{Definition}
\newtheorem{Exa}[equation]{Example}

\newenvironment{red}{\relax\color{red}}{\relax}
\newenvironment{blue}{\relax\color{blue}}{\hspace*{.5ex}\relax}

\newcommand{\ber}{\begin{red}}
\newcommand{\er}{\end{red}}
\newcommand{\beb}{\begin{blue}}
\newcommand{\eb}{\end{blue}}

\newcommand{\bx}{{\boldsymbol{x}}}
\newcommand{\bp}{{\boldsymbol{p}}}
\newcommand{\bw}{{\boldsymbol{w}}}
\newcommand{\bv}{{\boldsymbol{v}}}
\newcommand{\bu}{{\boldsymbol{u}}}
\newcommand{\up}{{\underline{\bp}}}

\DeclareMathOperator{\sgn}{sgn}

\begin{document}
\title[Mutations of reflections and existence of pseudo-acyclic orderings for type $A_n$]{Mutations of reflections  and \\ existence of pseudo-acyclic orderings for type $A_n$}
\author[T. J. Ervin]{Tucker J. Ervin}
\address{Department of Mathematics, University of Alabama,
Tuscaloosa, AL 35487, U.S.A.}
\email{tjervin@crimson.ua.edu}

\author[B. Jackson]{Blake Jackson}
\address{Department of Mathematics, University of Alabama,
Tuscaloosa, AL 35487, U.S.A.}
\email{bajackson9@crimson.ua.edu}

\author[K.-H. Lee]{Kyu-Hwan Lee$^{\star}$}
\thanks{$^{\star}$This work was partially supported by a grant from the Simons Foundation (\#712100).}
\address{Department of
Mathematics, University of Connecticut, Storrs, CT 06269, U.S.A.}
\email{khlee@math.uconn.edu}

\author[K. Lee]{Kyungyong Lee$^{\dagger}$}
\thanks{$^{\dagger}$This work was partially supported by NSF grant DMS 2042786, the University of Alabama, and Korea Institute for Advanced Study.}
\address{Department of Mathematics, University of Alabama,
Tuscaloosa, AL 35487, U.S.A. 
and Korea Institute for Advanced Study, Seoul 02455, Republic of Korea}
\email{klee94@ua.edu; klee1@kias.re.kr}


\begin{abstract}
In a recent paper by K.-H. Lee, K. Lee and M. Mills, a mutation of reflections in the universal Coxeter group is defined in association with a mutation of a quiver. A matrix representation of these reflections is determined by a linear ordering on the set of vertices of the quiver. It was conjectured that  there exists an ordering (called a {\em pseudo-acyclic ordering} in this paper) such that  whenever two mutation sequences of a quiver lead to the same labeled seed, the representations of the associated reflections also coincide. In this paper, we prove this conjecture for every quiver mutation-equivalent to an orientation of a type $A_n$ Dynkin diagram by decomposing a mutation sequence into a product of elementary swaps and checking relations studied by Barot and Marsh. 
\end{abstract}

\maketitle

\section{Introduction}

The mutation of a quiver $Q$ was defined by S. Fomin and A. Zelevinsky in their seminal paper \cite{FZ1} where they introduced cluster algebras. It also appeared in the context of Seiberg duality \cite{FHHU}. The $c$-vectors (and $C$-matrices) of $Q$ were defined through mutations in further developments of the theory of cluster algebras \cite{FZ4}, and together with their companions, $g$-vectors (and $G$-matrices), played  fundamental roles in the study of cluster algebras (for instance, see \cite{DWZ, GHKK, MG, NZ, Pl}). When $Q$ is acyclic, it turned out that positive $c$-vectors are actually real Schur roots, that is, the dimension vectors of indecomposable rigid modules over $Q$ \cite{HK, NC,ST}. Moreover, they appear as the denominator vectors of non-initial cluster variables of the cluster algebra associated to $Q$ \cite{CK}.

Due to the multifaceted appearance of $c$-vectors in important constructions, there have been various results related to description of $c$-vectors (or real Schur roots) of an acyclic quiver (\cite{BDSW,HK,IS,S,Se,ST}).
In \cite{LL}, K.-H. Lee and K. Lee conjectured  a correspondence   between real Schur roots of an acyclic quiver  and  non-self-crossing curves on a Riemann surface and hence proposed a new combinatorial/geometric description.
The conjecture is now proven by A. Felikson and P. Tumarkin \cite{FT} for acyclic quivers with multiple edges between every pair of vertices. Recently, S. D. Nguyen \cite{Ngu} informed us that he proved the conjecture for an arbitrary acyclic (valued) quiver.

When $Q$ is a non-acyclic quiver with the set of vertices $\mathcal I$, the situation is vastly different and little is known about description of $c$-vectors. In an attempt to understand the general case, K.-H. Lee, K. Lee and M. Mills \cite[Conjecture 1.9]{LLM} conjectured existence of a linear ordering $\prec$ on $\mathcal I$ that enables us to understand behaviors of $c$-vectors in comparison with acyclicity. For this purpose, they introduced a mutation of reflections in the universal Coxeter group for any quiver by adopting formulas from acyclic quivers and defined a  matrix representation of these reflections to be determined by a linear ordering $\prec$ on $\mathcal I$. The corresponding vectors to these reflections are called the {\em $l$-vectors}, and the matrix consisting of $l$-vectors the {\em $L$-matrix}. 

The conjecture of \cite{LLM} (Conjecture \ref{vague_conj-3} below) says that  there exists a linear ordering $\prec$ on $\mathcal I$ such that  whenever two mutation sequences from an initial labelled seed lead to the same labelled seed, the representations of the associated reflections also coincide, or equivalently, the $L$-matrices coincide up to signs of row vectors, where the representations and the $L$-matrices are determined by the ordering $\prec$. Such an ordering will be called a {\em pseudo-acyclic ordering} in this paper. If the conjecture is established, we would be able to study $c$-vectors by comparing them with $l$-vectors obtained from a pseudo-acyclic ordering. The deviation of $c$-vectors from $l$-vectors would serve as a measurement of non-acyclicity.

\medskip

In this paper, we prove Conjecture \ref{vague_conj-3} for type $A_n$ quivers. 
{\em By a quiver of type $A_n$, we mean any quiver that is mutation-equivalent to an orientation of a type $A_n$ Dynkin diagram.} 
Below is an example of the quivers considered in this paper:
\medskip
	\begin{center}
	\includegraphics[scale=0.5]{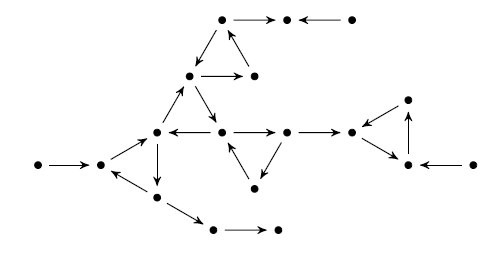}
	\end{center}
The key ingredients of our proof are decomposition of a mutation sequence into a product of elementary swaps and generalization of Coxeter relations studied by Barot and Marsh  \cite{BMa} (see also \cite{FT2,Se3}). 

For a transposition $(i,j)$ in the symmetric group $\mathfrak S_n$, an {\em elementary swap} for a labelled seed is a mutation sequence which switches labels $i$ and $j$ in the labelled seed. We first show, generalizing the result of unlabelled case in \cite{FST}, that an arbitrary mutation sequence yielding back the initial seed can be decomposed into a product of elementary swaps. Therefore, it is enough to consider elementary swaps only. Next, when an elementary swap is applied to reflections, we have the corresponding indices switched in the reflections, however,  with some extra factors emerging. Interestingly, the extra factors turn out to be nothing but the relations studied by Barot and Marsh in \cite{BM}. Hence the final step of the proof is to prove the existence of an ordering with which these relations hold for the reflections, making the extra factors equal the identity. 

Our proof involves dividing cases and applying inductions along with explicit computations. Some of the computations hold for a broader class of quivers. We expect that our approach would work to prove Conjecture \ref{vague_conj-3} for the quivers associated to unpunctured surfaces and the type $D_n$ quivers. In another direction, it will be interesting to investigate properties of $l$-vectors and mutation of reflections in comparison with behaviors of $c$-vectors as indicated above.

\subsection{Organization of the paper} 
In Section \ref{sec-prelim}, we fix notations to state the main conjecture and the main theorem. In the next section, we define elementary swaps and establish their properties. In Section \ref{sec-results}, we prove the main theorem deferring proofs of three lemmas to later sections. In Sections 5-8, the three lemmas are proven.


\section{Preliminaries} \label{sec-prelim}

In this section, we fix our notations and state the conjecture in \cite{LLM} which we will prove for type $A_n$ quivers.

\medskip

For a nonzero vector $c=( c_1 , \dots , c_n) \in \mathbb Z^n$, we define $c >0$ if all $c_i$ are non-negative, and  $c <0$ if all $c_i$ are non-positive. This induces a partial ordering $<$ on $\mathbb Z^n$. Define $|c|=( |c_1|, \dots,  |c_n| )$.
Assume that $M=[m_{ij}]$ is an $n \times 2n$ matrix of integers. Let $\mathcal I:= \{ 1, 2, \dots, n \}$ be the set of indices. For $\bw=[i_i, i_2, \dots , i_\ell]$, $i_j \in \mathcal I$, we define the matrix $M^\bw=[m_{ij}^\bw]$ inductively: the initial matrix is $M$ for $\bw=[\,]$, and assuming we have $M^\bw$, define the matrix $M^{\bw[k]}=[m_{ij}^{\bw[k]}]$ for $k \in \mathcal I$ with $\bw[k]:=[i_i, i_2, \dots , i_\ell, k]$
by  
\begin{equation} \label{eqn-mmuu} m_{ij}^{\bw[k]} = \begin{cases} -m_{ij}^\bw & \text{if  $i=k$ or $j=k$}, \\ m_{ij}^\bw + \mathrm{sgn}(m_{ik}^\bw) \, \max(m_{ik}^\bw m_{kj}^\bw,0) & \text{otherwise}, \end{cases}
\end{equation}
where  $\mathrm{sgn}(a) \in \{1, 0, -1\}$ is the signature of $a$. The matrix $M^{\bw[k]}$ is called the {\em mutation of $M^\bw$} at index (or label) $k$.

Let $B=[b_{ij}]$ be an $n\times n$ skew-symmetrizable matrix  and $D=\mathrm{diag}(d_1, \dots, d_n)$ be its symmetrizer such that $BD$ is symmetric, $d_i \in \mathbb Z_{>0}$ and $\gcd(d_1, \dots, d_n)=1$.  Consider the $n\times 2n$ matrix $\begin{bmatrix}B&I\end{bmatrix}$ and
 a mutation sequence $\bw=[i_1, \dots ,i_k]$, where $I$ is the $n\times n$ identity matrix. After the mutations at the indices $i_1, \dots , i_k$ consecutively, we obtain $\begin{bmatrix}B^\bw&C^\bw\end{bmatrix}$. The matrix $C^\bw$ is called the {\em $C$-matrix} and its row vectors the {\em $c$-vectors}. 
Write their entries as 
\begin{equation} \label{eqn-bcbc} B^\bw= \begin{bmatrix} b_{ij}^\bw \end{bmatrix}, \qquad
C^\bw= \begin{bmatrix} c_{ij}^\bw \end{bmatrix} = \begin{bmatrix} c_1^\bw \\ \vdots \\ c_n^\bw \end{bmatrix},\end{equation}
where $c_i^\bw$ are the $c$-vectors.
It is well-known that the $c$-vector $c_i^\bw$ is non-zero for each $i$, and either $c_i^\bw >0$ or $c_i^\bw <0$  due to the sign coherence of $c$-vectors  \cite{DWZ-1, GHKK}.

A {\em generalized intersection matrix} (GIM) is a square matrix ${A}=[a_{ij}]$ with integral entries such that
(1) for diagonal entries, $a_{ii}=2$;
(2) $a_{ij}>0$ if and only if $a_{ji}>0$;
(3) $a_{ij}<0$ if and only if $a_{ji}<0$.
Let $\mathcal A$ be the (unital) $\mathbb Z$-algebra generated by $s_i, e_i$, $i=1,2, \dots, n$, subject to the following relations:
$$ s_i^2=1, \quad \sum_{i=1}^n e_i =1, \quad s_ie_i = -e_i, \quad e_is_j=  \begin{cases} s_i+e_i-1 &\text{if } i =j, \\ e_i &\text{if } i \neq j, \end{cases}  \quad e_ie_j=  \begin{cases} e_i &\text{if } i =j, \\ 0 & \text{if } i \neq j. \end{cases}$$
Let $\mathcal W$ be the subgroup of the units of $\mathcal{A}$ generated by $s_i$, $i=1, \dots, n$. Note that $\mathcal W$ is (isomorphic to) the universal Coxeter group. Thus the algebra $\mathcal A$ can be considered as the algebra generated by the reflections and projections of the universal Coxeter group. 

Let ${A}=[a_{ij}]$ be an $n\times n$ symmetrizable GIM, and $D=\mathrm{diag}(d_1, \dots , d_n)$ be the symmetrizer, i.e. the diagonal matrix such that $d_i \in \mathbb Z_{>0}$, $\gcd(d_1, \dots, d_n)=1$ and ${A}D$ is symmetric. Let $\Gamma = \sum_{i=1}^n \mathbb{Z}\alpha_i$ be the lattice generated by the formal symbols $\alpha_1,\cdots,\alpha_n$. Define a representation $\pi:\mathcal A \rightarrow \mathrm{End}(\Gamma)$ by
\[ \pi(s_i)(\alpha_j) = \alpha_j - a_{ji} \alpha_i \quad \text{ and } \quad \pi(e_i)(\alpha_j) =\delta_{ij} \alpha_i , \quad\text{ for }i,j\in\{1, \dots ,n\}. \] We suppress $\pi$ when we write the action of an element of $\mathcal A$ on $\Gamma$.

Given a skew symmetrizable matrix $B$, we choose a linear ordering $\prec$ on $\{ 1,2, \dots , n\}$ and define the associated GIM ${A}=[a_{ij}]$ by
\begin{equation} 
a_{ij}= \begin{cases}  b_{ij}  & \text{ if } i \prec j , \\
2 & \text{ if } i =j, \\ -b_{ij} & \text{ if } i \succ j . \end{cases}\ 
\label{eqn-gim}
\end{equation}
An ordering $\prec$ provides a certain way for us to regard the skew-symmetrizable matrix $B$ as acyclic even when it is not.

\begin{Def} \label{def-r}
For each  mutation sequence $\bw$, define $r_i^\bw \in \mathcal W \subset \mathcal A$ inductively  with the initial elements $r_i = s_i$,  $i=1, \dots, n$, as follows: 
\begin{equation} \label{def-sx_i-1} r_i^{\bw [k]}= \begin{cases} r_k^\bw r_i^\bw  r_k^\bw & \text{ if } \ b_{ik}^\bw c_{k}^\bw>0, \\ r_i^\bw & \text{ otherwise.} \end{cases} \end{equation}
Clearly, each $r_i^\bw$ is written in the form
\[ r_i^\bw = g_i^\bw s_i {(g_i^\bw)}^{-1}, \quad g_i^\bw \in \mathcal W, \quad i=1, \dots , n.\]
\end{Def}

\begin{Def}\label{def-ell}
Choose an ordering $\prec$ on $\mathcal I$ to fix a GIM ${A}$, and define
\[   l_i^\bw = g_i^\bw (\alpha_i), \qquad i=1, \dots , n, \]
where we identify $\Gamma$ with $\mathbb Z^n$ and set $\alpha_1=(1,0,\dots , 0), \ \alpha_2=(0,1,0, \dots, 0), \dots, \ \alpha_n=(0, \dots, 0, 1)$. Then the {\em $L$-matrix} $L^{\bw}$ associated to ${A}$ is defined to be the $n \times n$ matrix whose $i^\text{th}$ row is $l_i^\bw$ for $i= 1, \dots, n$, i.e., \[ L^{\bw} =  \begin{bmatrix} l_1^{\bw}  \\ \vdots \\ l_n^{\bw} \end{bmatrix},\] and the vectors $l_i^{\bw}$ are called the {\em $l$-vectors of ${A}$}. Note that the $L$-matrix and $l$-vectors associated to a GIM ${A}$ implicitly depend the representation $\pi$ which is suppressed from the notation. 
\end{Def}

This paper is concerned with the following conjecture.
\begin{Conj}[\cite{LLM}] \label{vague_conj-3}
For any skew-symmetrizable matrix $B$, there exists a linear ordering $\prec$ on $\mathcal I$ such that if $\bw$ and $\bv$ are two mutation sequences with $C^\bw=C^\bv$ (hence $B^\bw=B^\bv$ as well) then
$\pi(r_i^\bw) =\pi(r_i^\bv), \, i=1,\dots,n,$ where $\pi$ is determined by $\prec$ and its associated GIM.
\end{Conj}

It is easy to see that Conjecture~\ref{vague_conj-3} is equivalent to saying that there exists a linear ordering $\prec$ on $\mathcal I$ such that if $\bw$ and $\bv$ are two mutation sequences with $C^\bw=C^\bv$ then $L^\bw$ and $L^\bv$ coincide up to signs of row vectors. 

\begin{Def}
A linear ordering $\prec$ on $\mathcal I$ with the property in Conjecture \ref{vague_conj-3} is called a {\em pseudo-acyclic ordering} for $B$.
\end{Def}

The elements $\pi(r_i^\bw)$ can be viewed as elements of $\pi(\mathcal W)$, and Conjecture~\ref{vague_conj-3} can be interpreted as a statement about relations in $\pi(\mathcal W).$ Relations for these groups have been explored for particular skew-symmetrizable matrices and  a restricted class of GIMs in \cite{BMa,FT2,Se3}. 

\medskip

A skew-symmetric matrix $B=[b_{ij}]$ is identified with a quiver $Q$ with vertices $\mathcal I$: when $b_{ij} >0$, the vertices $i$ and $j$ are connected with $b_{ij}$ arrows from $i$ to $j$. Recall our definition of type $A_n$ quivers from the introduction. Now we state the main theorem of this paper.
 
\begin{Thm} \label{recomposition-1}
	For any type $A_n$ quiver $Q$, there exists a pseudo-acyclic ordering $\prec$ on $\mathcal I$ for $Q$.
\end{Thm}

The rest of this paper is devoted to a proof of Theorem \ref{recomposition-1}.

\section{Elementary Swaps} \label{swap-results}

Throughout this section, we work with the class of quivers associated with bordered surfaces with marked points on the boundary and no punctures. This class includes the type $A_n$ quivers and is much broader. Details about quivers associated with bordered surfaces with marked points can be found in \cite{FST}. The purpose of this section is to show that it is essentially enough to consider elementary swaps. Our result extends the result on unlabelled case in \cite{FST} to our setting.
  
\medskip

Denote by $(i,j) \in \mathfrak{S}_n$ the transposition of $i$ and $j$ in the symmetric group $\mathfrak{S}_n$. 
When $\bp=[k_1, k_2, \dots , k_s]$, define  $\bp^{-1} = [k_s, k_{s-1}, \dots, k_2, k_1]$, \[ \bp^{(i,j)}=[ (i,j)k_1, (i,j) k_2, \dots , (i,j)k_s] \ \text{ and } \ \up^{(i,j)}=[ (i,j)k_s, (i,j) k_{s-1}, \dots , (i,j)k_2, (i,j)k_1].\] 

\begin{Exa}
When $(3,5) \in \mathfrak S_5$ and $\bp=[4,3,2,3,5]$, we have
\[  \bp^{-1}=[5,3,2,3,4], \quad \bp^{(3,5)}=[4,5,2,5,3] \quad \text{ and } \quad \up^{(3,5)}=[3,5,2,5,4] . \]
\end{Exa}

Let $(\mathbb{S}, \mathsf M)$ be a bordered surface with marked points and no punctures. 
The {\em flip graph} $\mathcal F$ of $(\mathbb S, \mathsf M)$ has a vertex for every triangulation of $(\mathbb S, \mathsf M)$ and an edge when two triangulations differ by a flip. For a triangulation $T$ of $(\mathbb S, \mathsf M)$, a {\em labelled triangulation} $\widetilde T$ of $T$ is obtained by labelling the arcs of $T$ by $1$ through $n$.   
The {\em labelled flip graph} $\widetilde{\mathcal F}$ of $(\mathbb S, \mathsf M)$ has a vertex for every labelled triangulation of $(\mathbb S, \mathsf M)$ and an edge when two labelled triangulations differ by a flip. We regard $\widetilde{\mathcal F}$ as a covering space of $\mathcal F$.

As in \cite{FST}, let $\mathcal A = \mathcal A(\mathbb S, \mathsf M)$ be the cluster algebra associated with $(\mathbb S, \mathsf M)$. 
The {\em exchange graph} $\mathcal E$ of $\mathcal A$ has a vertex for every seed and an edge when two seeds differ by a mutation. The {\em labelled exchange graph} $\widetilde{\mathcal E}$ of $\mathcal A$ has a vertex for every labelled seed of $\mathcal A$ and an edge when two labelled seeds differ by a mutation.

\begin{Lem}[\cite{FST}] \label{lem-iso}
Let $(\mathbb{S}, \mathsf M)$ be a bordered surface with marked points and no punctures. Then we have $\mathcal F \cong \mathcal E$ as graphs.
\end{Lem}

For any transposition $(i, j) \in \mathfrak S_n$, 
we write $(\bx,B)^{(i,j)}$ for the labelled seed obtained from a labelled seed $(\bx,B)$ after swapping the $i^{\mathrm{th}}$ and $j^{\mathrm{th}}$ components of $\bx$ and rows and columns of $B$. Similarly, we write $\widetilde T^{(i,j)}$ for the labelled triangulation which swaps the labels $i$ and $j$ of a labelled triangulation $\widetilde T$. More generally, for any elements $\sigma, \gamma \in \mathfrak S_n$, we define $(\bx, B)^{\sigma \gamma} =( (\bx, B)^\sigma)^\gamma$ and $\widetilde T^{\sigma\gamma} = (\widetilde T^\sigma)^\gamma$.

\begin{Def} \label{elem-walk} \hfill

(1) An {\em elementary walk} in the flip graph $\mathcal F$ is a closed walk of the form $\bp \bv \bp^{-1}$, where $\bp$ is a walk starting at a triangulation $T$ and $\bv$ is a 4-cycle or 5-cycle. It is known \cite{FST} that $\bv=[i,j,i,j]$ or $[i,j,i,j,i]$ for some $i,j \in \mathcal I$.

(2) Let $\bp$ be a walk from a triangulation $T$ which corresponds to the seed $(\bx, B)$. Then we call walks of the form $\bp\bp^{-1}$ {\em spurs}. Also, a {\em stable walk} is a walk of the form $\bp[i,j,i,j]\bp^{-1}$ where $b_{ij}^\bp = 0$ and $B^\bp = [b_{ij}^\bp]$. Note that a stable walk in $\mathcal F$ lifts to a closed walk in $\widetilde{\mathcal F}$.

(3) For any elementary walk $\bp\bv\bp^{-1}$, if $\bv=[i,j,i,j,i]$ is a 5-cycle, then $\bp\bv\bp^{-1}$ lifts to the walk $\bp \bv \up^{(i,j)}$ in the labelled flip graph $\widetilde{\mathcal F}$. In this case, we write  $$ [i,j]_\bp := \bp \bv \up^{(i,j)},$$ and say that $[i,j]_\bp$ is an {\em elementary swap}.
\end{Def}

The following lemma explains why $[i,j]_\bp$ is so called. This terminology also appears in \cite{LMW} from a different context. 

\begin{Lem} \label{swap-as-transposition}
	An elementary swap $[i,j]_\bp$ acts as a transposition of the labels of a labelled triangulation, i.e., the elementary swap $[i,j]_\bp$ takes $\widetilde{T}$ to $\widetilde{T}^{(i,j)}$.
	Similarly, mutating from $(\bx,B)$ along $[i,j]_{\bp}$ produces $(\bx,B)^{(i,j)}$.\end{Lem}

\begin{proof}
	We argue by induction on the length $m$ of $\bp$.
	If $m = 0$, then $[i,j]_\bp = [i,j,i,j,i]$.
	As $[i,j,i,j,i]$ corresponds to a pentagon in the flip graph, we have that $[i,j]_\bp$ transposes two labels of the triangulation, completing the base case.
	
	If we assume the inductive hypothesis holds for $m$, we look at walks of length $m+1$.
	Then ${\bp} = [k]\bw$ and $\up^{(i,j)} =\underline{\bw}^{(i,j)}[k]^{(i,j)}$ for some label $k$ and some sequence $\bw$ of length $m$.
	By the inductive hypothesis, $[i,j]_\bw$ permutes the labelled triangulation formed by flipping at label $k$ in $\widetilde{T}$.
	Flipping at $k$ and $(i,j)k$ at the respective ends of $[i,j]_\bw$ clearly produces $[i,j]_\bp$.
	Hence $[i,j]_\bp$ permutes the labels of $\widetilde{T}$ by the permutation $(i,j)$, demonstrating our assertion.
	
A similar inductive argument proves that $[i,j]_\bp$ yields $(\bx,B)^{(i,j)}$. 
	Indeed, if $[i,j]_\bp = [i,j,i,j,i]$ is a 5-cycle in $\mathcal F$, then applying $[i,j]_\bp$ to the initial seed $(\bx, B)$ comes back to $(\bx, B)$ as an unlabelled seed by Lemma \ref{lem-iso}. One can check that the labels are swapped to produce $(\bx, B)^{(i,j)}$, establishing the base case. (See, e.g., \cite{LS}.)  The remaining argument is straightforward as with $\widetilde{T}$.
\end{proof}

We now prove a workhorse lemma, generalizing Lemma \ref{lem-iso}, that makes our analysis possible.

\begin{Lem} \label{labelled-seed-bijection}
Let $(\mathbb{S}, \mathsf M)$ be a bordered surface with marked points and no punctures. Then we have $\widetilde{\mathcal F} \cong \widetilde{\mathcal E}$ as graphs.
\end{Lem}

\begin{proof}
	Fix a triangulation $T_0$ of $(\mathbb S, \mathsf M)$ and choose a labelled triangulation $\widetilde{T}_0$ of $T_0$. Denote the corresponding labelled seed by $(\bx, B)$. For any vertex $\widetilde T$ in $\widetilde{\mathcal F}$, there exists $\sigma \in \mathfrak{S}_n$ and a walk $\bp = [p_1, p_2, \dots, p_m]$ from $\widetilde T_0^\sigma$ to $\widetilde T$.
	Such a walk exists because the flip graph of $(\mathbb{S},\mathsf M)$ is connected and the labelled flip graph (not necessarily connected) is a covering space of the flip graph.
	Let $\mu_\bp((\bx,B)^{\sigma})$ be the labelled seed produced by consecutively mutating $(\bx,B)^{\sigma}$ at labels in $\bp$. 

Define a map from $\widetilde{\mathcal F}$ to $\widetilde{\mathcal E}$ by $\widetilde T \mapsto \mu_\bp((\bx,B)^{\sigma})$.
	To see that this map is well-defined, let $\gamma$ be another permutation such that there exists a walk $\bu$ from $\widetilde{T}_0^\gamma$ to $\widetilde T$.
	We need to show $\mu_\bp((\bx,B)^\sigma) = \mu_\bu((\bx,B)^\gamma)$, which is equivalent to $\mu_{\bp \bu^{-1}}((\bx,B)^\sigma) = (\bx, B)^\gamma$.
	Thus it is sufficient to show that
\begin{equation*} \tag{$*$}\label{tas} \text{every walk $\bw$ from $\widetilde T_0^\sigma$ to $\widetilde T_0^\gamma$ has $\mu_\bw((\bx,B)^\sigma) = (\bx,B)^\gamma$.}
\end{equation*}
	
	Every walk $\bw$ in $\widetilde{\mathcal F}$ from $\widetilde T_0^\sigma$ to $\widetilde T_0^\gamma$ corresponds under the covering map to a closed walk at $T_0$ in $\mathcal F$.
	By Theorem 3.10 in \cite{FST}, we may write $\bw$ as a concatenation of finitely many elementary walks in $\mathcal F$, up to a finite number of spur insertions or deletions. Write this decomposition of $\bw$ in $\mathcal F$ as $(\bw_1\bv_1\bw_1^{-1})(\bw_2\bv_2\bw_2^{-1})\dots(\bw_k\bv_k\bw_k^{-1})$.
Since the lifts of stable walks are closed in $\widetilde{\mathcal F}$ and the corresponding walks in $\widetilde{\mathcal E}$ are also closed,  
we may remove all stable walks from the decomposition without affecting the flips of $\widetilde{T}_0^\sigma$ or the mutations of $(\bx,B)^\sigma$.  
Thus we may assume that $\bw$ is decomposed into the product $(\bw_1\bv_1\bw_1^{-1})(\bw_2\bv_2\bw_2^{-1})\dots(\bw_k\bv_k\bw_k^{-1})$ in $\mathcal F$ with $\bw_l\bv_l\bw_l^{-1}$ lifting to an elementary swap in $\widetilde{\mathcal F}$ for each $l=1,2, \dots, k$.
If we let $(i_l, j_l)$ denote the transposition given by $\bv_l$,
	then $\bw$ in $\widetilde{\mathcal F}$ is written as 
\begin{equation} \label{dele} \bw = [i_1,j_1]_{\bw_1}[i_2,j_2]_{\bw_2}\dots[i_k,j_k]_{\bw_k}. \end{equation}

Now let us prove the assertion \eqref{tas}.	We argue by induction on $k$.  
	If $k = 1$, it follows from Lemma \ref{swap-as-transposition} that $\widetilde T_0^{\sigma (i_1, j_1)}= \widetilde T_0^\gamma$ and $\sigma (i_1, j_1) = \gamma$. Thus, again by Lemma  \ref{swap-as-transposition}, we have 
$\mu_\bw((\bx,B)^\sigma)= (\bx, B)^{\sigma(i_1,j_1)}= (\bx,B)^\gamma$, completing the base case.
	Assume that the assertion \eqref{tas} holds for $k$.
	Now, if 
	$$\bw = [i_1,j_1]_{\bw_1}[i_2,j_2]_{\bw_2}\cdots[i_{k+1},j_{k+1}]_{\bw_{k+1}},$$
	where each $[i_l,j_l]_{\bw_l}$ is an elementary swap in $\widetilde F$, let 
	$$\bw' = [i_1,j_1]_{\bw_1}[i_2,j_2]_{\bw_2}\cdots[i_k,j_k]_{\bw_k} \quad \text{ and } \quad \tau=(i_1,j_1)(i_2, j_2) \cdots (i_k, j_k) \in \mathfrak S_n.$$
Since $\bw'$ takes $\widetilde T_0^\sigma$ to $\widetilde T_0^{\sigma\tau}$, we have $\mu_{\bw'}(\bx, B)^\sigma = (\bx, B)^{\sigma \tau}$ by induction. Now $\bw$ takes $\widetilde T_0$ to $\widetilde T_0^{\sigma \tau (i_{k+1}, j_{k+1})}$ with $ \gamma=\sigma \tau (i_{k+1}, j_{k+1})$. By Lemma \ref{swap-as-transposition}, we have
\[ \mu_\bw((\bx, B)^\sigma) = \mu_{[i_{k+1}, j_{k+1}]_{\bw_{k+1}}}((\bx, B)^{\sigma \tau}) = (\bx, B)^{\sigma \tau (i_{k+1}, j_{k+1})} =(\bx, B)^\gamma . \]  
Thus the assertion \eqref{tas} is proven for all $k \ge 1$, and the map  $\widetilde T \mapsto \mu_\bp((\bx,B)^{\sigma})$ is well-defined. From the construction of the map, the edges in $\widetilde{\mathcal F}$ are preserved under this map.
	
In a similar way, we can define the inverse map and show that every sequence $\bw$ of mutations such that $\mu_\bw((\bx,B)^\sigma)= (\bx,B)^\gamma$ is a sequence of flips from $\widetilde T_0^\sigma$ to $\widetilde T_0^\gamma$. Thus each labelled triangulation $\widetilde T$ is uniquely associated with a labelled seed, preserving the graph structure.
\end{proof}

With Lemma \ref{labelled-seed-bijection} established, the terminology defined for $\widetilde{\mathcal F}$ can be used for $\widetilde{\mathcal E}$ as well. We immediately obtain the following corollary.

\begin{Cor} \label{decomposition-walks}
	In the cluster algebra associated to $(\mathbb S, \mathsf M)$, every walk from $(\bx,B)$ to a labelled seed $(\bx,B)^{\sigma}$ for a permutation $\sigma \in \mathfrak{S}_n$ may be decomposed into a sequence of elementary swaps, up to spur insertion or deletion or the possible removal of stable walks.
\end{Cor}

\begin{proof}
Since $\widetilde{\mathcal F} \cong \widetilde{\mathcal E}$, the assertion follows from \eqref{dele}. 
\end{proof}

The following proposition enables us to focus on only elementary swaps in the rest of the paper.

\begin{Prop} \label{decomposition}
Let a skew-symmetrizable matrix $B$ be associated with a bordered surface $(\mathbb S, \mathsf M)$ with marked points. Assume that $\bv$ is a mutation sequence such that $C^\bv=I$. Then $\bv$ can be written as a product of elementary swaps,
	\[ [i_1,j_1]_{\bp_1} [i_2,j_2]_{\bp_2} \cdots [i_s,j_s]_{\bp_s},\]
	up to a finite number of spur insertions and deletions as well as the removal of stable walks.  
	Moreover,
	\[ (i_1,j_1) (i_{2},j_{2}) \cdots (i_s,j_s) =e \in \mathfrak{S}_n .\]
\end{Prop}

\begin{proof}
	We write $(\bx,B)$ for the seed associated with the matrix $[B \, I]$.  
It is well-known that the $C$-matrix determines the labelled seed. In particular, the assumption $C^\bv=I$ implies that $\bv$ is a closed walk in $\widetilde{\mathcal E}$.
It follows from Corollary \ref{decomposition-walks} that we may decompose $\bv$ into \[ [i_1,j_1]_{\bp_1} [i_2,j_2]_{\bp_2} \cdots [i_s,j_s]_{\bp_s}.\]
Write $\sigma := (i_1, j_1)(i_2, j_2) \cdots (i_s, j_s) \in \mathfrak S_n$. Then, as in the proof of Lemma \ref{labelled-seed-bijection}, we have
\[  (\bx, B)=\mu_\bv((\bx, B))= (\bx, B)^\sigma. \] Thus we obtain $\sigma=e \in \mathfrak S_n$.
\end{proof}

Proposition \ref{decomposition} allows us to break any closed walk beginning at the starting seed into a series of elementary swaps, and we can concentrate on the effects of a single elementary swap.
When we understand how they work individually, we will be able to combine them together to understand closed walks in general. 

\section{Main Theorem} \label{sec-results}

From now on, we restrict ourselves to type $A_n$ quivers. We consider {\em all type $A_n$ quivers}, even non-acyclic ones, as clarified in the introduction. In this section, we prove the main theorem of this paper, deferring proofs of the first three lemmas to later sections. 

\medskip

To begin, we have Lemmas \ref{swap-indices-i,j} and \ref{swap-indices-k} whose proofs will be given in Sections \ref{proof-swap-indices-i,j}, \ref{discussion-swap-indices-k} and \ref{proof-swap-indices-k}.

\begin{Lem} \label{swap-indices-i,j}
	If $\bv = [i,j]_{\bp}$ is an elementary swap for some $i,j \in \mathcal{I}$ and for a mutation sequence $\bp$, then $\pi(r_i^{\bp [i,j,i,j,i]}) = \pi(r_i^{\bp} r_j^{\bp})^3 \pi(r_j^{\bp})$ or $\pi(r_j^{\bp} r_i^{\bp})^3 \pi(r_j^{\bp})$ and $\pi(r_j^{\bp [i,j,i,j,i]}) = \pi(r_i^{\bp} r_j^{\bp})^3 \pi(r_i^{\bp})$ or $\pi(r_j^{\bp} r_i^{\bp})^3 \pi(r_i^{\bp})$.
\end{Lem}

\begin{Lem} \label{swap-indices-k}
	Suppose that $\bv = [i,j]_{\bp}$ is an elementary swap for some $i,j \in \mathcal{I}$ and for a mutation sequence $\bp$.
	Let $k \neq i,j \in \mathcal{I}$.
	Then the following hold:
	\begin{enumerate}
		
		\item[(A)] If $b_{ki}^{\bp} = 0 = b_{kj}^{\bp}$, then 
		$$\pi(r_k^{\bp[i,j,i,j,i]}) = \pi(r_k^{\bp}).$$
		
		\item[(B)] If $b_{ki}^{\bp} = 0$ and $b_{kj}^{\bp} \neq 0$, then
		\begin{align*} \pi(r_k^{\bp [i,j,i,j,i]}) = &\ \pi(r_k^{\bp}), \quad \pi(r_i^\bp r_j^\bp)^3\pi(r_k^{\bp}) \pi(r_j^\bp r_i^\bp)^3, \quad  \pi(r_j^\bp r_i^\bp)^3\pi(r_k^{\bp}) \pi(r_i^\bp r_j^\bp)^3, \\ &\  \pi(r_i^{\bp}r_k^{\bp})^2 \pi(r_k^{\bp}), \
		\text{ or } \  \pi(r_i^\bp r_j^\bp)^3\pi(r_i^\bp r_k^{\bp})^2 \pi(r_k^\bp)\pi(r_j^\bp r_i^\bp)^3.\end{align*}
		
		\item[(C)] If $b_{ki}^{\bp} \neq 0$ and $b_{kj}^{\bp} = 0$, then
		\begin{align*} \pi(r_k^{\bp [i,j,i,j,i]}) =& \ \pi(r_k^{\bp}), \quad \pi(r_i^\bp r_j^\bp)^3\pi(r_k^{\bp}) \pi(r_j^\bp r_i^\bp)^3, \quad \pi(r_j^\bp r_i^\bp)^3\pi(r_k^{\bp}) \pi(r_i^\bp r_j^\bp)^3, \\ & \ \pi(r_j^{\bp}r_k^{\bp})^2 \pi(r_k^{\bp}), \ \text{ or } \ \pi(r_i^\bp r_j^\bp)^3\pi(r_j^\bp r_k^{\bp})^2 \pi(r_k^\bp)\pi(r_j^\bp r_i^\bp)^3. \end{align*}
		
		\item[(D)] If $b_{ki}^{\bp} \neq 0$ and $b_{kj}^{\bp} \neq 0$, then
		$$\pi(r_k^{\bp[i,j,i,j,i]}) = \pi(r_k^{\bp}) \ \text{ or } \ \pi(r_j^\bp r_i^\bp r_j^\bp r_k^{\bp})^2 \pi(r_k^\bp).$$
	\end{enumerate}
\end{Lem}

If we can show that the powered elements in Lemmas \ref{swap-indices-i,j} and \ref{swap-indices-k} reduce to the identity, then we will have shown that Corollary \ref{single-swap-result} holds for all elementary swaps.
To do so, we first look at elementary swaps $[i,j]_\bp$ where $\bp =[\,]$.
Note that $|b_{ij}^\bp| = 1$ if and only if there is an elementary swap $[i,j]_\bp$.

\begin{Lem} \label{base-case}
	If $b_{ij} = 0$ for some $i,j \in \mathcal{I}$, then 
	$$\pi(r_ir_j)^2 = \mathrm{id} = \pi(r_jr_i)^2.$$
	If $|b_{ij}| = 1$, then 
	$$\pi(r_ir_j)^3 = \mathrm{id} = \pi(r_jr_i)^3.$$
	Additionally, if $k \neq i,j \in \mathcal{I}$ such that $b_{ki} = b_{ij} = b_{jk} = 1$ and the linear order satisfies
	$$i \prec k \prec j, \quad j \prec i \prec k, \ \text{ or } \  k \prec j \prec i,$$
	then
	$$\pi(r_j r_i r_j r_k)^2 = \mathrm{id} = \pi(r_k r_j r_i r_j)^2.$$
	A similar identity holds for $b_{ki} = b_{ij} = b_{jk} = -1$ with 
	$$i \prec j \prec k, \quad j \prec k \prec i, \ \text{ or } \ k \prec i \prec j.$$
\end{Lem}

Section \ref{proof-base-case} is devoted to a proof of Lemma \ref{base-case}.
\begin{Lem} \label{order-choice}
	If $Q$ is a quiver of type $A_n$ containing $m$ triangles (cycles of length 3) and $q$ vertices not in cycles, then we may choose a linear order $\prec$ on $\mathcal I$ such that
	$$\pi(r_j r_i r_j r_k)^2 = \mathrm{id} = \pi(r_k r_j r_i r_j)^2$$
	whenever $b_{ki} = b_{ij} = b_{jk} = \pm 1$ for some $i,j,k \in \mathcal{I}$.
\end{Lem}

\begin{proof}
	We first note that a type $A_n$ quiver $Q$ only contains simple cycles that are of length $3$ and are oriented, say $m$ of them. We know from Lemma \ref{base-case} that we need to choose a linear ordering that respects cycles. An important note is that the maximum degree of any vertex is $4$ in $Q$. If a vertex has degree $4$, then it must be a common vertex for two triangles (see vertex $2$ in our example below). Proposition 3.1 in \cite{NS} gives a complete description of type $A_n$ quivers. For example, the quiver below is mutation equivalent to $A_{18}$ with $m = 5$ and $q = 6$.
	\begin{center}
	\includegraphics[scale=0.5]{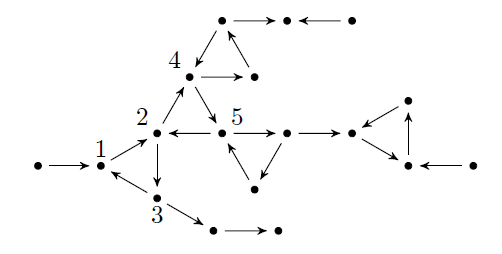}
	\end{center}
	Our argument exploits the quasi-tree nature of $Q$. We labeled only certain vertices in the above example to make the following process of choosing a linear ordering clear. We can choose a linear ordering like so:
	\begin{enumerate}
		\item Number the cycles $1, \ldots , m$. Choose one of the $3$ acceptable linear orderings for the first cycle, say  $3 \prec 2 \prec 1$.
		\item If the next cycle shares a vertex with any previous cycle, replace the shared vertex with one of the $3$ accepted linear orderings for the next cycle:
		$$3 \prec (5 \prec 2 \prec 4) \prec 1 = 3 \prec 5 \prec 2 \prec 4 \prec 1$$
		If the next cycle does not share a vertex with a previous one, form the following ordering:
		$$3 \prec 2 \prec 1 \prec i \prec j \prec k$$
		If $b_{ik} = b_{ij} = b_{jk} = -1$ is the next cycle. There are $3$ choices for this cycle as well.
		\item Continue in this way until we finish the last cycle:
		$$3 \prec 5 \prec 2 \prec \cdots \prec 4 \prec \cdots \prec 1 \prec \cdots \prec x \prec \cdots \prec y \prec z, \quad \text{ or}$$
		$$3 \prec 2 \prec 1 \prec \cdots \prec i \prec \cdots \prec j \prec \cdots \prec k \prec \cdots \prec x \prec \cdots \prec y \prec z$$
		If $b_{zx} = b_{xy} = b_{yz} = -1$ is the $m^{th}$ cycle.
		\item Finally, place the remaining $q$ vertices not in a cycle anywhere in the sequence.
	\end{enumerate}
	Clearly, for any $i,j,k \in \mathcal{I}$, if $b_{ik} = b_{ij} = b_{jk} = -1$ then $i \prec j \prec k$. This along with Lemma \ref{base-case} proves Lemma \ref{order-choice}.
\end{proof}

Note that step 2 is always feasible (with a minor adjustment) even if a cycle shares an edge with another; the only requirement we have is that cycles must be oriented and there are no ``cycles of cycles''. Therefore, this proof may be extended to the case of a type $D_n$ quiver, making adjustments as needed. 

If a type $A_n$ quiver $Q$ is acyclic, then Lemma \ref{base-case} completes the base case of the induction in Lemma \ref{induction} for arbitrary orders.
If $Q$ possesses cycles, then we can use Lemma \ref{order-choice} to complete this same induction but for a specific linear ordering.

\begin{Lem} \label{induction}
	Assume that $Q$ is a quiver of type $A_n$. 
	Then there exists a linear ordering $\prec$ such that if $|b_{ij}^\bv|=0$, then $\pi(r_i^\bv r_j^\bv)^2 =\mathrm{id}$; if $|b_{ij}^\bv| =1$, then $\pi(r_i^\bv r_j^\bv)^3 =\mathrm{id}$; if $b_{ki}^{\bv} = b_{ij}^{\bv} = b_{jk}^{\bv} = \pm 1$, then $\pi(r_j^{\bv} r_i^{\bv} r_j^{\bv} r_k^{\bv})^2 = \mathrm{id}$ for $i,j,k \in \mathcal I$ and all mutation sequences $\bv$.
\end{Lem}

\begin{proof}
	
	If $Q$ is acyclic, we may pick any linear ordering. 
	If $Q$ is not acyclic, pick a linear ordering by Lemma \ref{order-choice}.
	From Lemma \ref{base-case}, we have that the result holds for $\bv = [\,]$ and our chosen linear ordering, completing the base case for both acyclic and non-acyclic $Q$.
	Let us continue with induction on the length $m$ of $\bv$.
	Assume that $\pi(r_i^\bv r_j^\bv)^2 =\mathrm{id}$ whenever $|b_{ij}^{\bv}| = 0$, $\pi(r_i^\bv r_j^\bv)^3 =\mathrm{id}$ whenever $|b_{ij}^{\bv}| = 1$, and $\pi(r_j^{\bv} r_i^{\bv} r_j^{\bv} r_k^{\bv})^2 = \mathrm{id}$ whenever $b_{ki}^{\bv} = b_{ij}^{\bv} = b_{jk}^{\bv} = \pm 1$ for $i,j,k \in \mathcal{I}$.
	We wish to show that the same identities hold for $\bv[l]$, where $l \in \mathcal{I}$.
	We begin with the first two identities.
	
	Let us mutate in the direction $l \in \mathcal{I}$.
	If $l = i$ or $j$, then $|b_{ij}^{\bv[l]}| = |b_{ij}^{\bv}|$ and $r_i^{\bv[l]}r_j^{\bv[l]} = r_i^{\bv}r_j^{\bv}$ or $r_j^{\bv}r_i^{\bv}$. 
	The inductive hypothesis then gives us the desired result. 
	If $l \neq i, j$, then we would like to know how $|b_{ij}^{\bv[l]}|$ differs from $|b_{ij}^{\bv}|$.
	As $b_{ij}^{\bv[l]} = b_{ij}^{\bv} + \sgn(b_{il}^{\bv})\max(b_{il}^{\bv}b_{lj}^{\bv}, 0)$, it only changes when $b_{il}^{\bv}b_{lj}^{\bv} > 0$.
	Thus we may split into four cases:
	\begin{align*} r_i^{\bv[l]} r_j^{\bv[l]} &= r_i^{\bv} r_j^{\bv}, &
	r_i^{\bv[l]} r_j^{\bv[l]}& = r_i^{\bv} r_l^{\bv}r_j^{\bv} r_l^{\bv},\\
	r_i^{\bv[l]} r_j^{\bv[l]}& =  r_l^{\bv}r_i^{\bv} r_l^{\bv}r_j^{\bv},&
	r_i^{\bv[l]} r_j^{\bv[l]}&=  r_l^{\bv}r_i^{\bv}r_j^{\bv} r_l^{\bv}.\end{align*}
	In the first and fourth cases, we have $b_{il}^{\bv}b_{lj}^{\bv} \leq 0$, forcing $b_{ij}^{\bv[l]} = b_{ij}^{\bv}$.
	In the second and third cases, we have $b_{il}^{\bv}b_{lj}^{\bv} \geq 0$.
	Clearly, the inductive hypothesis would lead to the desired result in the first and last cases.
	The third case is the second case with $i$ and $j$ swapped.
	Hence, we may restrict solely to the second case to complete our induction.
	
	If
	$$r_i^{\bv[l]} r_j^{\bv[l]} = r_i^{\bv} r_l^{\bv}r_j^{\bv} r_l^{\bv},$$
	then there are four further possible situations:
	\begin{align*} (1)& \,b_{il}^{\bv} = 0, \, b_{ij}^{\bv} = 0, \, b_{ij}^{\bv[l]} = 0;&
	(2)& \,b_{il}^{\bv} = 0, \, b_{ij}^{\bv} \neq 0, \, b_{ij}^{\bv[l]} = b_{ij}^{\bv};\\
	(3)& \,b_{il}^{\bv} \neq 0, \, b_{ij}^{\bv} = 0, \, b_{ij}^{\bv[l]} \neq 0;&
	(4)& \,b_{il}^{\bv} \neq 0, \, b_{ij}^{\bv} \neq 0, \, b_{ij}^{\bv[l]} = b_{ij}^{\bv} + \sgn(b_{il}^{\bv})b_{il}^{\bv}b_{lj}^{\bv}.\end{align*}
	Note that, when working with type $A_n$ quivers, case (4) requires that $b_{ij}^{\bv[l]} = 0$.
	
	\underline{Case (1):}
	The inductive hypothesis gives that $\pi(r_i^{\bv}r_l^{\bv})^2 = \mathrm{id}$ and $\pi(r_i^{\bv}r_j^{\bv})^2 = \mathrm{id}$.
	Hence
	$$\pi(r_i^{\bv[l]} r_j^{\bv[l]})^2 = \pi(r_i^{\bv} r_l^{\bv}r_j^{\bv} r_l^{\bv}r_i^{\bv} r_l^{\bv}r_j^{\bv} r_l^{\bv}) = \pi(r_i^{\bv} r_l^{\bv}r_j^{\bv} r_i^{\bv} r_j^{\bv} r_l^{\bv}) =  \pi(r_i^{\bv} r_l^{\bv} r_i^{\bv}  r_l^{\bv}) = \mathrm{id},$$
	satisfying the result.
	
	\underline{Case (2):}
	The inductive hypothesis gives that $\pi(r_i^{\bv}r_l^{\bv})^2 = \mathrm{id}$ and $\pi(r_i^{\bv} r_j^{\bv})^3 = \mathrm{id}$.
	Hence
	$$\pi(r_i^{\bv[l]} r_j^{\bv[l]})^3 = \pi(r_i^{\bv} r_l^{\bv}r_j^{\bv} r_l^{\bv}r_i^{\bv} r_l^{\bv}r_j^{\bv} r_l^{\bv}r_i^{\bv} r_l^{\bv}r_j^{\bv} r_l^{\bv}) = \pi(r_i^{\bv} r_l^{\bv}r_j^{\bv} r_i^{\bv} r_j^{\bv} r_i^{\bv}r_j^{\bv} r_l^{\bv}) = \pi(r_i^{\bv} r_l^{\bv}r_i^{\bv} r_l^{\bv}) = \mathrm{id},$$
	satisfying the result.
	
	\underline{Case (3):}
	The inductive hypothesis gives that $\pi(r_i^{\bv}r_l^{\bv})^3 = \mathrm{id}$ and $\pi(r_i^{\bv} r_j^{\bv})^2 = \mathrm{id}$.
	Hence
	$$\pi(r_i^{\bv[l]} r_j^{\bv[l]})^3 = \pi(r_i^{\bv} r_l^{\bv}r_j^{\bv} r_l^{\bv}r_i^{\bv} r_l^{\bv}r_j^{\bv} r_l^{\bv}r_i^{\bv} r_l^{\bv}r_j^{\bv} r_l^{\bv}) =\pi(r_i^{\bv} r_j^{\bv}r_l^{\bv} r_j^{\bv}r_i^{\bv} r_j^{\bv}r_l^{\bv} r_j^{\bv}r_i^{\bv} r_j^{\bv}r_l^{\bv} r_j^{\bv}) $$
	$$= \pi(r_i^{\bv} r_j^{\bv}r_l^{\bv} r_i^{\bv}r_l^{\bv}r_i^{\bv}r_l^{\bv} r_j^{\bv}) = \pi(r_i^{\bv} r_j^{\bv}r_i^{\bv} r_j^{\bv}) = \mathrm{id},$$
	satisfying the result.
	
	\underline{Case (4):}
	Note that $b_{il}^{\bv}b_{lj}^{\bv} > 0$.
	If $b_{il}^{\bv} < 0$, then $b_{ij}^{\bv} > 0$.
	If $b_{il}^{\bv} > 0$, then $b_{ij}^{\bv} < 0$.
	Hence $b_{il}^{\bv}, b_{lj}^{\bv},$ and $b_{ji}^{\bv}$ all have the same sign.
	The inductive hypothesis gives that $\pi(r_i^{\bv}r_l^{\bv})^3 = \mathrm{id}$, $\pi(r_i^{\bv} r_j^{\bv})^3 = \mathrm{id}$, and $\pi(r_j^{\bv} r_l^{\bv}r_j^{\bv} r_i^{\bv})^2 = \mathrm{id}$.
	Hence
	$$\pi(r_i^{\bv[l]} r_j^{\bv[l]})^2 = \pi(r_i^{\bv} r_l^{\bv}r_j^{\bv} r_l^{\bv}r_i^{\bv} r_l^{\bv}r_j^{\bv} r_l^{\bv}) = \pi(r_i^{\bv} r_j^{\bv}r_l^{\bv} r_j^{\bv}r_i^{\bv} r_j^{\bv}r_l^{\bv} r_j^{\bv}) =\pi(r_i^{\bv} r_j^{\bv}r_l^{\bv} r_j^{\bv})^2 = \mathrm{id}. $$
	Case (4) is the reason why we require the $\pi(r_j^{\bv} r_l^{\bv}r_j^{\bv} r_i^{\bv})^2$ identity, which causes certain orders in non-acyclic $Q$ to fail.
	See Counterexample \ref{exa-counter}.
	
	We now work to show that $\pi(r_j^{\bv[l]} r_i^{\bv[l]} r_j^{\bv[l]} r_k^{\bv[l]})^2 = \mathrm{id}$ whenever $b_{ki}^{\bv[l]} = b_{ij}^{\bv[l]}= b_{jk}^{\bv[l]} \neq 0$ for vertices $i,j,k,l \in \mathcal{I}$.
	Assume that $i,j,k \neq l$.
	Then at least two of $b_{il}, b_{jl},$ and $b_{kl}$ are zero due to our restriction to type $A_n$ quivers.
	It is then easy to show from our inductive hypothesis that 
	$$\pi(r_j^{\bv[l]} r_i^{\bv[l]} r_j^{\bv[l]} r_k^{\bv[l]})^2 = \pi(r_j^{\bv} r_i^{\bv} r_j^{\bv} r_k^{\bv})^2  \ \text{ or } \ \pi(r_l^{\bv})\pi(r_j^{\bv} r_i^{\bv} r_j^{\bv} r_k^{\bv})^2\pi(r_l^{\bv}).$$
	The inductive hypothesis then gives us again that 
	$$\pi(r_j^{\bv[l]} r_i^{\bv[l]} r_j^{\bv[l]} r_k^{\bv[l]})^2 = \mathrm{id},$$
	as $b_{ki}^{\bv} = b_{ij}^{\bv}= b_{jk}^{\bv} \neq 0$.
Hence, we need only prove our result for the case $l \in \{i,j,k\}$, which we will split into three cases.	
	
	\underline{Case $l = i$:}
	Then $b_{ki}^{\bv} = b_{ij}^{\bv}$ and $b_{jk}^{\bv} = 0$.
	The inductive hypothesis gives us $\pi(r_i^{\bv}r_k^{\bv})^3 = \mathrm{id}$, $\pi(r_i^{\bv} r_j^{\bv})^3 = \mathrm{id}$, and $\pi(r_j^{\bv}r_k^{\bv})^2 = \mathrm{id}$.
	Additionally, either $b_{ki}^{\bv}c_i^{\bv}$ or $b_{ji}^{\bv}c_i^{\bv}$ must be positive but not both.
	Thus $r_j^{\bv[l]} r_i^{\bv[l]} r_j^{\bv[l]} r_k^{\bv[l]} = r_i^{\bv} r_j^{\bv} r_i^{\bv} r_j^{\bv}  r_i^{\bv}r_k^{\bv}$ or $ r_j^{\bv} r_i^{\bv} r_j^{\bv}  r_i^{\bv}r_k^{\bv}r_i^{\bv}$.
	
	If $r_j^{\bv[l]} r_i^{\bv[l]} r_j^{\bv[l]} r_k^{\bv[l]} = r_i^{\bv} r_j^{\bv} r_i^{\bv} r_j^{\bv}  r_i^{\bv}r_k^{\bv}$, then
	$$\pi(r_j^{\bv[l]} r_i^{\bv[l]} r_j^{\bv[l]} r_k^{\bv[l]} )^2 = \pi(r_i^{\bv} r_j^{\bv} r_i^{\bv} r_j^{\bv}  r_i^{\bv}r_k^{\bv}r_i^{\bv} r_j^{\bv} r_i^{\bv} r_j^{\bv}  r_i^{\bv}r_k^{\bv}) $$
	$$ = \pi(r_j^{\bv} r_k^{\bv}r_j^{\bv} r_k^{\bv}) = \pi(r_j^{\bv} r_k^{\bv})^2 = \mathrm{id}.$$
	
	If $r_j^{\bv[l]} r_i^{\bv[l]} r_j^{\bv[l]} r_k^{\bv[l]} =  r_j^{\bv} r_i^{\bv} r_j^{\bv}  r_i^{\bv}r_k^{\bv}r_i^{\bv}$, then
	\begin{align*} \pi(r_j^{\bv[l]} r_i^{\bv[l]} r_j^{\bv[l]} r_k^{\bv[l]} )^2&= \pi(r_j^{\bv} r_i^{\bv} r_j^{\bv}  r_i^{\bv}r_k^{\bv}r_i^{\bv}r_j^{\bv} r_i^{\bv} r_j^{\bv}  r_i^{\bv}r_k^{\bv}r_i^{\bv})\\
	&=\pi(r_j^{\bv} r_i^{\bv} r_j^{\bv}  r_i^{\bv}r_k^{\bv}r_j^{\bv}r_k^{\bv}r_i^{\bv}) = \pi(r_j^{\bv} r_i^{\bv} r_j^{\bv}  r_i^{\bv}r_j^{\bv}r_i^{\bv}) = \pi(r_j^{\bv}  r_i^{\bv})^3 = \mathrm{id}. \end{align*}
	
	\underline{Case $l = j$:}
	Then $b_{ij}^{\bv} = b_{jk}^{\bv}$ and $b_{ki}^{\bv} = 0$.
	The inductive hypothesis gives us $\pi(r_i^{\bv}r_j^{\bv})^3 = \mathrm{id}$, $\pi(r_j^{\bv} r_k^{\bv})^3 = \mathrm{id}$, and $\pi(r_i^{\bv}r_k^{\bv})^2 = \mathrm{id}$.
	Additionally, either $b_{ij}^{\bv}c_j^{\bv}$ or $b_{kj}^{\bv}c_j^{\bv}$ must be positive but not both.
	Thus $r_j^{\bv[l]} r_i^{\bv[l]} r_j^{\bv[l]} r_k^{\bv[l]} = r_i^{\bv}r_k^{\bv}$ or $ r_j^{\bv} r_i^{\bv} r_k^{\bv}r_j^{\bv}$.
	
	If $r_j^{\bv[l]} r_i^{\bv[l]} r_j^{\bv[l]} r_k^{\bv[l]} = r_i^{\bv}r_k^{\bv}$, then
	$$\pi(r_j^{\bv[l]} r_i^{\bv[l]} r_j^{\bv[l]} r_k^{\bv[l]})^2 = \pi( r_i^{\bv}r_k^{\bv})^2 = \mathrm{id}.$$
	
	If $r_j^{\bv[l]} r_i^{\bv[l]} r_j^{\bv[l]} r_k^{\bv[l]} = r_j^{\bv} r_i^{\bv} r_k^{\bv}r_j^{\bv}$, then
	$$\pi(r_j^{\bv[l]} r_i^{\bv[l]} r_j^{\bv[l]} r_k^{\bv[l]})^2 = \pi(r_j^{\bv} r_i^{\bv} r_k^{\bv}r_i^{\bv} r_k^{\bv}r_j^{\bv}) = \pi(r_j^{\bv} r_j^{\bv}) = \mathrm{id}.$$
	
	\underline{Case $l = k$:}
	Then $b_{jk}^{\bv} = b_{ki}^{\bv}$ and $b_{ij}^{\bv} = 0$.
	The inductive hypothesis gives us $\pi(r_j^{\bv}r_k^{\bv})^3 = \mathrm{id}$, $\pi(r_k^{\bv} r_i^{\bv})^3 = \mathrm{id}$, and $\pi(r_i^{\bv}r_j^{\bv})^2 = \mathrm{id}$.
	Additionally, either $b_{jk}^{\bv}c_k^{\bv}$ or $b_{ik}^{\bv}c_k^{\bv}$ must be positive but not both.
	Thus $r_j^{\bv[l]} r_i^{\bv[l]} r_j^{\bv[l]} r_k^{\bv[l]} = r_k^{\bv}r_j^{\bv}r_k^{\bv}r_i^{\bv}r_k^{\bv}r_j^{\bv}$ or $ r_j^{\bv}r_k^{\bv} r_i^{\bv}r_k^{\bv} r_j^{\bv}r_k^{\bv}$.
	
	If $r_j^{\bv[l]} r_i^{\bv[l]} r_j^{\bv[l]} r_k^{\bv[l]} = r_k^{\bv}r_j^{\bv}r_k^{\bv}r_i^{\bv}r_k^{\bv}r_j^{\bv}$, then
	$$\pi(r_j^{\bv[l]} r_i^{\bv[l]} r_j^{\bv[l]} r_k^{\bv[l]} )^2 = \pi(r_k^{\bv}r_j^{\bv}r_k^{\bv}r_i^{\bv}r_k^{\bv}r_j^{\bv}r_k^{\bv}r_j^{\bv}r_k^{\bv}r_i^{\bv}r_k^{\bv}r_j^{\bv})$$
	$$= \pi(r_k^{\bv}r_j^{\bv}r_k^{\bv}r_i^{\bv}r_j^{\bv}r_i^{\bv}r_k^{\bv}r_j^{\bv}) = \pi(r_k^{\bv}r_j^{\bv}r_k^{\bv}r_j^{\bv}r_k^{\bv}r_j^{\bv}) = \pi(r_k^{\bv}r_j^{\bv})^3 = \mathrm{id}.$$
	
	If $r_j^{\bv[l]} r_i^{\bv[l]} r_j^{\bv[l]} r_k^{\bv[l]} = r_j^{\bv}r_k^{\bv} r_i^{\bv}r_k^{\bv} r_j^{\bv}r_k^{\bv}$, then
	$$\pi(r_j^{\bv[l]} r_i^{\bv[l]} r_j^{\bv[l]} r_k^{\bv[l]} )^2 = \pi(r_j^{\bv}r_k^{\bv} r_i^{\bv}r_k^{\bv} r_j^{\bv}r_k^{\bv}r_j^{\bv}r_k^{\bv} r_i^{\bv}r_k^{\bv} r_j^{\bv}r_k^{\bv})$$
	$$ = \pi(r_j^{\bv}r_k^{\bv} r_i^{\bv}r_j^{\bv} r_i^{\bv}r_k^{\bv} r_j^{\bv}r_k^{\bv}) = \pi(r_j^{\bv}r_k^{\bv}r_j^{\bv}r_k^{\bv} r_j^{\bv}r_k^{\bv}) = \pi(r_j^{\bv} r_k^{\bv})^3 = \mathrm{id}.$$
	
	This last case completes the inductive step.
	Therefore, there exists a linear ordering $\prec$ such that if $|b_{ij}^\bv|=0$, then $\pi(r_i^\bv r_j^\bv)^2 =\mathrm{id}$; if $|b_{ij}^\bv| =1$, then $\pi(r_i^\bv r_j^\bv)^3 =\mathrm{id}$; if $b_{ki}^{\bv} = b_{ij}^{\bv} = b_{jk}^{\bv} = \pm 1$, then $\pi(r_j^{\bv} r_i^{\bv} r_j^{\bv} r_k^{\bv})^2 = \mathrm{id}$ for $i,j,k \in \mathcal I$.
\end{proof}

We give a counterexample to demonstrate the necessity of Lemma \ref{order-choice} when proving Lemma \ref{induction}.

\begin{Exa}[Counterexample to Lemma \ref{induction} for arbitrary linear orders when quivers possess cycles] \label{exa-counter}
	Observe the case where we have the following quiver:
	\begin{center}
	\includegraphics[scale=0.5]{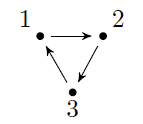}
	\end{center}
	Then  we have that 
	$$\pi(r_1^{[2]} r_3^{[2]})^2 = \mathrm{id}$$
	only when
	$$1 \prec 3 \prec 2, \quad 2 \prec 1 \prec 3, \ \text{ or } \  3 \prec 2 \prec 1.$$
	If
	$$1 \prec 2 \prec 3, \quad 2 \prec 3 \prec 1, \ \text{ or } \ 3 \prec 1 \prec 2,$$
	then
	$$\pi(r_1^{[2]} r_3^{[2]})^2 \neq \mathrm{id},$$
	creating a contradiction to Lemma \ref{induction} for arbitrary linear orders.
\end{Exa}

In what follows, notice that Lemma \ref{induction} combined with Lemmas \ref{swap-indices-i,j} and \ref{swap-indices-k} satisfies the assumption of Lemma \ref{swap-identity}.

\begin{Lem} \label{swap-identity}
If $\bv = [i,j]_{\bp}$ is an elementary swap for some $i,j \in \mathcal{I}$ and for a mutation sequence $\bp$ with $\pi(r_k^{\bp [i,j,i,j,i]}) = \pi(r_{(i, j) k}^{\bp })$ for all $k \in \mathcal{I}$, then $\pi(r_k^{\bp [i,j,i,j,i] \bu}) = \pi(r_{(i,j)k}^{\bp \bw})$, where $\bu=[u_1, u_2, \dots , u_m]$ is an arbitrary mutation sequence of length $m$ and $\bw = [(i,j) u_1, (i,j) u_2, \dots , (i,j) u_m]$.

\end{Lem}

\begin{proof}
	
	We argue by induction on the length $m$ of a mutation sequence $\bu$.
	If $m = 0$, then $\pi(r_k^{\bp [i,j,i,j,i]}) = \pi(r_{(i, j)k}^{\bp})$ by Lemmas \ref{swap-indices-k} and \ref{induction}, completing the base case.
	Assume now that $\pi(r_k^{\bp [i,j,i,j,i] \bu}) = \pi(r_{(i, j)k}^{\bp \bw})$ for all $\bu$ with length $m - 1$, where $\bw = [(i, j) u_1, (i, j) u_{2}, \dots, (i, j) u_{m-1}]$.
	If we mutate at the vertex $k \in \mathcal{I}$, then 
	$$b_{kl}^{\bp [i,j,i,j,i] \bu} c_l^{\bp [i,j,i,j,i] \bu} = b_{(i, j)k (i, j) l}^{\bp \bw}c_{(i, j) l}^{\bp \bw}.$$
	We may then reduce to two cases:
	$$r_k^{\bp [i,j,i,j,i] \bu[l]} = r_k^{\bp [i,j,i,j,i] \bu},$$
	$$r_k^{\bp [i,j,i,j,i] \bu[l]} = r_l^{\bp [i,j,i,j,i] \bu}r_k^{\bp [i,j,i,j,i] \bu}r_l^{\bp [i,j,i,j,i] \bu}.$$
	
	\underline{Case (1):}
	From our equality $b_{kl}^{\bp [i,j,i,j,i] \bu} c_l^{\bp [i,j,i,j,i] \bu} = b_{(i, j)k (i, j) l}^{\bp \bw}c_{(i, j) l}^{\bp \bw}$, we know that 
	$$r_{(i, j)k}^{\bp \bw [(i,j)l]} = r_{(i, j)k}^{\bp \bw}.$$
	We have that 
	$$\pi(r_{k}^{\bp [i,j,i,j,i] \bu[l]}) = \pi(r_{k}^{\bp [i,j,i,j,i] \bu }) = \pi(r_{(i,j)k}^{\bp \bw}) = \pi(r_{(i,j)k}^{\bp \bw[(i,j)l]}),$$
	completing the inductive step.
	
	\underline{Case (2):}
	From our equality $b_{kl}^{\bp [i,j,i,j,i] \bu} c_l^{\bp [i,j,i,j,i] \bu} = b_{(i, j)k (i, j) l}^{\bp \bw}c_{(i, j) l}^{\bp \bw}$, we know that 
	$$r_{(i, j)k}^{\bp \bw [(i,j)l]} = r_{(i,j)l}^{\bp \bw}r_{(i, j)k}^{\bp \bw}r_{(i,j)l}^{\bp \bw}.$$
	Repeating the same process as before, we arrive at 
	$$\pi(r_k^{\bp [i,j,i,j,i] \bu[l]}) = \pi(r_{(i, j)k}^{\bp \bw[(i,j)l]}),$$
	completing the inductive step.
	
	Therefore, for any mutation sequence $\bu$, we have $\pi(r_k^{\bp [i,j,i,j,i] \bu}) = \pi(r_{(i,j)k}^{\bp \bw})$.
\end{proof}

\begin{Cor} \label{single-swap-result}
	If $\bv = [i,j]_{\bp}$ is an elementary swap for some $i,j \in \mathcal{I}$ and for a mutation sequence $\bp$, then $\pi(r_k^{\bv}) = \pi(r_{(i,j)k})$ for $k \in \mathcal I$.
	
\end{Cor}

Corollary \ref{single-swap-result} will be used to prove Lemma \ref{repeated-swap} below.

\begin{Lem} \label{repeated-swap}
	Suppose that $\bv$ is a product of elementary swaps:
	$$\bv = [i_1,j_1]_{\bp_1} [i_2,j_2]_{\bp_2} \cdots [i_s,j_s]_{\bp_s}.$$
	If $\sigma = (i_1,j_1) (i_{2},j_{2}) \cdots (i_s,j_s) \in \mathfrak S_n,$
	then
	$$\pi(r_k^{\bv}) = \pi(r_{\sigma(k)})$$
	for all $k \in \mathcal{I}.$
\end{Lem}

\begin{proof}
	
	We argue by induction on $s$.
	For $s = 1$, Lemmas \ref{swap-indices-i,j}, \ref{swap-indices-k} and Corollary \ref{single-swap-result} prove the base case.
	We now assume that 
	$$\pi(r_k^{\bv}) = \pi(r_{\sigma(k)})$$
	for all mutation sequences $\bv$ composed of $s$ elementary swaps.
	
	Suppose that we have a mutation sequence
	$$\bv = [i_1,j_1]_{\bp_1} [i_2,j_2]_{\bp_2} \cdots [i_s,j_s]_{\bp_s}[i_{s+1},j_{s+1}]_{\bp_{s+1}}.$$
	Then $\bv = \bw [i_{s+1},j_{s+1}]_{\bp_{s+1}}$ for 
$\bw= [i_1,j_1]_{\bp_1} [i_2,j_2]_{\bp_2} \cdots [i_s,j_s]_{\bp_s}$.
	Hence, up to a spur insertion, 
	$$\bv = [i_{s+1},j_{s+1}]_{\bw \bp_{s+1}} \bu,$$
	where $\bu=\bw^{(i_{s+1},j_{s+1})}$.
	By Lemma \ref{swap-identity}, we have that
	$$\pi(r_k^{[i_{s+1},j_{s+1}]_{\bw \bp_{s+1}} \bu}) = \pi(r_{(i_{s+1},j_{s+1}) k}^{\bw}).$$
	From the inductive hypothesis
	$$\pi(r_{(i_{s+1},j_{s+1}) k}^{\bw}) = \pi(r_{\sigma((i_{s+1},j_{s+1})k)}).$$
	As $\sigma(i_{s+1},j_{s+1})$ is the permutation corresponding to $\bv$, we have completed the inductive step.
\end{proof}

\begin{Lem} \label{4-cycles-good} 
	Suppose that $b_{ij}^{\bu \bp} = 0$ for any two mutation sequences $\bu$ and $\bp$.
	Let $\bv = \bp[i,j,i,j]\bp^{-1}$, where $\bp^{-1}$ is the inverse of $\bp$.
	Then there exists a linear ordering $\prec$ such that
	$$\pi(r_k^{\bu \bv \bw}) = \pi(r_k^{\bu \bw})$$
	for any mutation sequence $\bw$ and any $k \in \mathcal{I}$. 
	
\end{Lem}

\begin{proof}
	
	Let us argue first by induction on the length of $\bw$ and $\bp$.
	If $\bw = [\,]$ and $\bp = [\,]$, then $b_{ij}^\bu = 0$ implies that
	$$r_k^{\bu [i,j,i,j]} = r_k^{\bu} \ \text{ or }\ (r_j^{\bu} r_i^{\bu})^2 r_k^{\bu} (r_i^{\bu} r_j^{\bu})^2.$$
	Lemma \ref{induction} gives us a linear ordering $\prec$ such that
	$$\pi(r_k^{\bu \bv \bw}) = \pi(r_k^{\bu \bw}),$$
	completing the base case.
	
	Suppose that $\pi(r_k^{\bu \bv \bw}) = \pi(r_k^{\bu \bw})$ whenever $\bw = [\,]$ and $\bp$ is a mutation sequence of length $m$.
	If $\bp'$ is a mutation sequence of length $m+1$, then $\bp' = [l]\bp$ for some $l \in \mathcal{I}$ and a mutation sequence $\bp$ of length $m$.
	Then $\bv' =  \bp'[i,j,i,j]\bp'^{-1} = [l]\bv[l]$ and
	$$r_k^{\bu \bv'} = r_k^{\bu[l]\bv [l]}.$$
	We know that $b_{kl}^{\bu} c_l^{\bu} = b_{kl}^{\bu [l] \bv [l] [l]} c_l^{\bu [l] \bv [l][l]}= b_{kl}^{\bu [l] \bv } c_l^{\bu [l] \bv }$.
	Thus $r_k^{\bu[l]\bv [l]} = r_k^{\bu[l]\bv}$ implies that
	$$\pi(r_k^{\bu[l]\bv [l]}) = \pi(r_k^{\bu[l]}) = \pi(r_k^{\bu}).$$
	Similarly, $r_k^{\bu[l]\bv [l]} = r_l^{\bu[l]\bv}r_k^{\bu[l]\bv}r_l^{\bu[l]\bv}$ implies that
	$$\pi(r_k^{\bu[l]\bv [l]}) = \pi(r_l^{\bu[l]\bv}r_k^{\bu[l]\bv}r_l^{\bu[l]\bv}) = \pi(r_l^{\bu[l]}r_k^{\bu[l]}r_l^{\bu[l]}) = \pi(r_k^{\bu}).$$
	Thus $\pi(r_k^{\bu \bv}) = \pi(r_k^{\bu})$ for all $\bv = \bp[i,j,i,j]\bp^{-1}$ by induction.
	
	Now, suppose that $\pi(r_k^{\bu \bv \bw}) = \pi(r_k^{\bu \bw})$ whenever $\bw$ is a mutation sequence of length $m$ and $\bp$ is any mutation sequence.
	If $\bw'$ is a mutation sequence of length $m+1$, then $\bw' = [l]\bw$ for some $l \in \mathcal{I}$ and a mutation sequence $\bw$ of length $m$.
	Then
	$$r_k^{\bu \bv \bw'} = r_k^{\bu \bv [l] \bw} = r_k^{\bu [l] [l] \bv [l] \bw}.$$
	Let $\bp' = [l]\bp$ and $\bv' = \bp'[i,j,i,j]\bp'^{-1}$.
	Then 
	$$\pi(r_k^{\bu [l] [l] \bv [l] \bw}) = \pi(r_k^{\bu [l] \bv' \bw}) = \pi(r_k^{\bu [l] \bw})$$
	by the inductive hypothesis.
	However, $[l] \bw = \bw'$.
	Thus
	$$\pi(r_k^{\bu \bv \bw}) = \pi(r_k^{\bu \bw}),$$
	completing the inductive step.
\end{proof}

Finally, Lemma \ref{repeated-swap} along with Proposition \ref{decomposition}, Lemmas \ref{order-choice}, and \ref{4-cycles-good} proves our main result of this paper.

\begin{Thm}[Proof of Theorem \ref{recomposition-1}] \label{recomposition}
	For any type $A_n$ quiver $Q$, there exists a pseudo-acyclic ordering $\prec$ on $\mathcal I$ for $Q$. That is, there exist a linear ordering $\prec$ and its associated GIM $A$ such that if $\bv$ is a mutation sequence with $C^\bv = I$ then
	$$\pi(r_k^{\bv}) = \pi(r_k)$$
	for all $k \in \mathcal{I}$.
\end{Thm}

\begin{proof}
	
	From Proposition \ref{decomposition}, we may decompose $\bv$ into a series of elementary swaps,
	$$\bw = [i_1,j_1]_{\bp_1} [i_2,j_2]_{\bp_2} \cdots [i_s,j_s]_{\bp_s},$$
	up to insertion or removal of spurs and removal of closed walks of the form $\bp[i,j,i,j]\bp^{-1}$.
	We already know that removal of spurs does not affect $r_k^\bv$ for any $k \in \mathcal{I}$, and Lemma \ref{4-cycles-good} proves that $\pi(r_k^{\bv}) = \pi(r_k^{\bw})$ after the removal of closed walks of the form $\bp[i,j,i,j]\bp^{-1}$.
	Lemma \ref{repeated-swap} shows that $\pi(r_k^{\bw}) = \pi(r_k)$.
	Therefore,
	$$\pi(r_k^{\bv}) = \pi(r_k),$$
	proving our result.
	
\end{proof}

\section{Proof of Lemma \ref{swap-indices-i,j}} \label{proof-swap-indices-i,j}

We restate the lemma for convenience.

\begin{Lem}
	If $\bv = [i,j]_{\bp}$ is an elementary swap for some $i,j \in \mathcal{I}$ and for a mutation sequence $\bp$, then $\pi(r_i^{\bp [i,j,i,j,i]}) = \pi(r_i^{\bp} r_j^{\bp})^3 \pi(r_j^{\bp})$ or $\pi(r_j^{\bp} r_i^{\bp})^3 \pi(r_j^{\bp})$ and $\pi(r_j^{\bp [i,j,i,j,i]}) = \pi(r_i^{\bp} r_j^{\bp})^3 \pi(r_i^{\bp})$ or $\pi(r_j^{\bp} r_i^{\bp})^3 \pi(r_i^{\bp})$.
\end{Lem}

\begin{proof}

To begin, we have four cases:
\begin{align*} (1)& \ b_{ij}^{\bp}c_{j}^{\bp} > 0 \, \text{ and }\, b_{ji}^{\bp}c_{i}^{\bp} < 0,&
(2)&\ b_{ij}^{\bp}c_{j}^{\bp} < 0 \, \text{ and } \, b_{ji}^{\bp}c_{i}^{\bp} > 0,\\
(3)&\ b_{ij}^{\bp}c_{j}^{\bp} < 0 \, \text{ and } \, b_{ji}^{\bp}c_{i}^{\bp} < 0,&
(4)&\ b_{ij}^{\bp}c_{j}^{\bp} > 0 \, \text{ and } \, b_{ji}^{\bp}c_{i}^{\bp} > 0.\end{align*}

\underline{Case (1):}
For $\bp [i]$:
$$b_{ij}^{\bp [i]}c_{j}^{\bp [i]} = -b_{ij}^{\bp}c_{j}^{\bp} < 0,$$
$$b_{ji}^{\bp [i]}c_{i}^{\bp [i]} = b_{ji}^{\bp}c_{i}^{\bp} < 0.$$
For $\bp [i,j]$:
$$b_{ij}^{\bp [i,j]}c_{j}^{\bp [i,j]} = b_{ij}^{\bp [i]}c_{j}^{\bp [i]} < 0,$$
$$b_{ji}^{\bp [i,j]}c_{i}^{\bp [i,j]} = -b_{ji}^{\bp [i]}c_{i}^{\bp [i]} > 0.$$
For $\bp [i,j,i]$:
$$b_{ij}^{\bp [i,j,i]}c_{j}^{\bp [i,j,i]} = -b_{ij}^{\bp [i,j]}(c_{j}^{\bp [i,j]} + \sgn(b_{ji}^{\bp [i,j]})b_{ji}^{\bp [i,j]}c_i^{\bp [i,j]}) > 0,$$
$$b_{ji}^{\bp [i,j,i]}c_{i}^{\bp [i,j,i]} = b_{ji}^{\bp [i,j]}c_{i}^{\bp [i,j]} > 0.$$
For $\bp [i,j,i,j]$:
$$b_{ij}^{\bp [i,j,i,j]}c_{j}^{\bp [i,j,i,j]} = b_{ij}^{\bp [i,j,i]}c_{j}^{\bp [i,j,i]} > 0,$$
$$b_{ji}^{\bp [i,j,i,j]}c_{i}^{\bp [i,j,i,j]} = b_{ij}^{\bp [j]}c_{j}^{\bp [j]} = b_{ij}^{\bp}c_{j}^{\bp} > 0.$$
Hence,
$$r_i^{\bp [i,j,i,j,i]} = r_i^{\bp [i,j,i,j]} = r_j^{\bp [i,j,i]}r_i^{\bp [i,j,i]}r_j^{\bp [i,j,i]} = r_i^{\bp [i,j]}r_j^{\bp [i,j]}r_i^{\bp [i,j]}r_j^{\bp [i,j]}r_i^{\bp [i,j]} = r_i^{\bp }r_j^{\bp }r_i^{\bp }r_j^{\bp }r_i^{\bp},$$
showing that $\pi(r_i^{\bp [i,j,i,j,i]}) = \pi(r_i^{\bp} r_j^{\bp})^3 \pi(r_j^{\bp})$.
We also have 
$$r_j^{\bp [i,j,i,j,i]} = r_i^{\bp [i,j,i,j]} r_j^{\bp [i,j,i,j]} r_i^{\bp [i,j,i,j]} = r_j^{\bp [i,j,i]}r_i^{\bp [i,j,i]}r_j^{\bp [i,j,i]}r_i^{\bp [i,j,i]}r_j^{\bp [i,j,i]}$$
$$ = r_i^{\bp [i,j]}r_j^{\bp [i,j]}r_i^{\bp [i,j]}r_j^{\bp [i,j]}r_i^{\bp [i,j]}r_j^{\bp [i,j]}r_i^{\bp [i,j]} = r_i^{\bp}r_j^{\bp}r_i^{\bp}r_j^{\bp}r_i^{\bp}r_j^{\bp}r_i^{\bp},$$
showing that $\pi(r_j^{\bp [i,j,i,j,i]}) = \pi(r_i^{\bp} r_j^{\bp})^3 \pi(r_i^{\bp})$.
This demonstrates that the result holds in Case (1).

\underline{Case (2):}
For $\bp [i]$:
$$b_{ij}^{\bp [i]}c_{j}^{\bp [i]} = -b_{ij}^{\bp}c_{j}^{\bp} > 0,$$
$$b_{ji}^{\bp [i]}c_{i}^{\bp [i]} = b_{ji}^{\bp}c_{i}^{\bp} > 0.$$
For $\bp [i,j]$:
$$b_{ij}^{\bp [i,j]}c_{j}^{\bp [i,j]} = b_{ij}^{\bp [i]}c_{j}^{\bp [i]} > 0,$$
$$b_{ji}^{\bp [i,j]}c_{i}^{\bp [i,j]} = b_{ij}^{\bp [j,i,j]} c_j^{\bp [j,i,j]} > 0.$$
The last inequality comes from the first case with $\bp [j,i,j,i,j]$ instead of $\bp [i,j,i,j,i]$.
For $\bp [i,j,i]$:
$$b_{ij}^{\bp [i,j,i]}c_{j}^{\bp [i,j,i]} = b_{ji}^{\bp [j,i]} c_i^{\bp [j,i]} < 0,$$
$$b_{ji}^{\bp [i,j,i]}c_{i}^{\bp [i,j,i]} = b_{ji}^{\bp [i,j]}c_{i}^{\bp [i,j]} > 0.$$
For $\bp [i,j,i,j]$:
$$b_{ij}^{\bp [i,j,i,j]}c_{j}^{\bp [i,j,i,j]} = b_{ij}^{\bp [i,j,i]}c_{j}^{\bp [i,j,i]} < 0,$$
$$b_{ji}^{\bp [i,j,i,j]}c_{i}^{\bp [i,j,i,j]} = -b_{ji}^{\bp [i,j,i]}c_{i}^{\bp [i,j,i]} < 0.$$
Hence,
$$r_i^{\bp [i,j,i,j,i]} = r_i^{\bp [i,j,i,j]} = r_i^{\bp [i,j,i]} = r_j^{\bp [i,j]}r_i^{\bp [i,j]}r_j^{\bp [i,j]} = r_i^{\bp [i]}r_j^{\bp [i]}r_i^{\bp [i]}r_j^{\bp [i]}r_i^{\bp [i]} = r_i^{\bp}r_j^{\bp}r_i^{\bp}r_j^{\bp}r_i^{\bp},$$
showing that $\pi(r_i^{\bp [i,j,i,j,i]}) = \pi(r_i^{\bp} r_j^{\bp})^3 \pi(r_j^{\bp})$.
We also have 
$$r_j^{\bp [i,j,i,j,i]} = r_j^{\bp [i,j,i]} = r_i^{\bp [i,j]}r_j^{\bp [i,j]}r_i^{\bp [i,j]} = r_j^{\bp [i]}r_i^{\bp [i]}r_j^{\bp [i]}r_i^{\bp [i]}r_j^{\bp [i]} = r_i^{\bp }r_j^{\bp }r_i^{\bp }r_j^{\bp }r_i^{\bp}r_j^{\bp}r_i^{\bp },$$
showing that $\pi(r_j^{\bp [i,j,i,j,i]}) = \pi(r_i^{\bp} r_j^{\bp})^3 \pi(r_i^{\bp})$.
This demonstrates that the result holds in Case (2).

\underline{Case (3):}
For $\bp [i]$:
$$b_{ij}^{\bp [i]}c_{j}^{\bp [i]} = -b_{ij}^{\bp}c_{j}^{\bp} > 0,$$
$$b_{ji}^{\bp [i]}c_{i}^{\bp [i]} = b_{ji}^{\bp}c_{i}^{\bp} < 0.$$
For $\bp [i,j]$:
$$b_{ij}^{\bp [i,j]}c_{j}^{\bp [i,j]} = b_{ij}^{\bp [i]}c_{j}^{\bp [i]} > 0,$$
$$b_{ji}^{\bp [i,j]}c_{i}^{\bp [i,j]} = -b_{ji}^{\bp [i]}(c_{i}^{\bp [i]}+ \sgn(b_{ij}^{\bp [i]})b_{ij}^{\bp [i]}c_{j}^{\bp [i]}) > 0.$$
For $\bp [i,j,i]$:
$$b_{ij}^{\bp [i,j,i]}c_{j}^{\bp [i,j,i]} = -b_{ij}^{\bp [i,j]}(c_{j}^{\bp [i,j]} + \sgn(b_{ji}^{\bp [i,j]})b_{ji}^{\bp [i,j]}c_i^{\bp [i,j]}),$$
$$b_{ji}^{\bp [i,j,i]}c_{i}^{\bp [i,j,i]} = b_{ji}^{\bp [i,j]}c_{i}^{\bp [i,j]} > 0.$$
We are not sure of the sign of the $b_{ij}^{\bp [i,j,i]}c_{j}^{\bp [i,j,i]}$.
However, we do know that, for $\bp [i,j,i,j,i]$,
$$b_{ij}^{\bp [i,j,i,j,i]}c_{j}^{\bp [i,j,i,j,i]} = b_{ji}^{\bp} c_{i}^{\bp} < 0,$$
$$b_{ji}^{\bp [i,j,i,j,i]}c_{i}^{\bp [i,j,i,j,i]} = b_{ij}^{\bp }c_{j}^{\bp  } < 0.$$
This allows us to solve for $\bp [i,j,i,j]$:
$$b_{ij}^{\bp [i,j,i,j]}c_{j}^{\bp [i,j,i,j]} = -b_{ij}^{\bp [i,j,i,j,i]}c_{j}^{\bp [i,j,i,j,i]} > 0,$$
$$b_{ji}^{\bp [i,j,i,j]}c_{i}^{\bp [i,j,i,j]} = b_{ji}^{\bp [i,j,i,j,i]}c_{i}^{\bp [i,j,i,j,i]} < 0.$$
We now may revisit $\bp [i,j,i]$:
$$b_{ij}^{\bp [i,j,i]}c_{j}^{\bp [i,j,i]} = b_{ij}^{\bp [i,j,i,j]}c_{j}^{\bp [i,j,i,j]} > 0,$$
$$b_{ji}^{\bp [i,j,i]}c_{i}^{\bp [i,j,i]} = b_{ji}^{\bp [i,j]}c_{i}^{\bp [i,j]} > 0.$$
Hence,
$$r_i^{\bp [i,j,i,j,i]} = r_i^{\bp [i,j,i,j]} = r_j^{\bp [i,j,i]}r_i^{\bp [i,j,i]}r_j^{\bp [i,j,i]} = r_i^{\bp [i,j]}r_j^{\bp [i,j]}r_i^{\bp [i,j]}r_j^{\bp [i,j]}r_i^{\bp [i,j]}$$
$$ = r_j^{\bp [i]}r_i^{\bp [i]}r_j^{\bp [i]}r_i^{\bp [i]}r_j^{\bp [i]}r_i^{\bp [i]}r_j^{\bp [i]} = r_j^{\bp}r_i^{\bp}r_j^{\bp}r_i^{\bp}r_j^{\bp}r_i^{\bp}r_j^{\bp},$$
We also have 
showing that $\pi(r_i^{\bp [i,j,i,j,i]}) = \pi(r_j^{\bp} r_i^{\bp})^3 \pi(r_j^{\bp})$.
$$r_j^{\bp [i,j,i,j,i]} = r_j^{\bp [i,j,i]} = r_i^{\bp [i,j]}r_j^{\bp [i,j]}r_i^{\bp [i,j]} = r_j^{\bp [i]}r_i^{\bp [i]}r_j^{\bp [i]}r_i^{\bp [i]}r_j^{\bp [i]} = r_i^{\bp }r_j^{\bp }r_i^{\bp }r_j^{\bp }r_i^{\bp}r_j^{\bp}r_i^{\bp },$$
showing that $\pi(r_j^{\bp [i,j,i,j,i]}) = \pi(r_i^{\bp} r_j^{\bp})^3 \pi(r_i^{\bp})$.
This demonstrates that the result holds in Case (3).

\underline{Case (4):}
The fourth case is the hardest.
For $\bp [i]$:
$$b_{ij}^{\bp [i]}c_{j}^{\bp [i]} = b_{ji}^{\bp [j,i,j,i]}c_{i}^{\bp [j,i,j,i]},$$
$$b_{ji}^{\bp [i]}c_{i}^{\bp [i]} = b_{ji}^{\bp}c_{i}^{\bp} > 0.$$
We immediately ran into an ambiguity.
Let us assume that $b_{ij}^{\bp [i]}c_{j}^{\bp [i]} < 0$.
For $\bp [i,j]$:
$$b_{ij}^{\bp [i,j]}c_{j}^{\bp [i,j]} = b_{ij}^{\bp [i]}c_{j}^{\bp [i]} < 0,$$
$$b_{ji}^{\bp [i,j]}c_{i}^{\bp [i,j]} = -b_{ji}^{\bp [i]}c_{i}^{\bp [i]} < 0.$$
For $\bp [i,j,i]$:
$$b_{ij}^{\bp [i,j,i]}c_{j}^{\bp [i,j,i]} = -b_{ij}^{\bp [i,j]}c_{j}^{\bp [i,j]} > 0,$$
$$b_{ji}^{\bp [i,j,i]}c_{i}^{\bp [i,j,i]} = b_{ji}^{\bp [i,j]}c_{i}^{\bp [i,j]} < 0.$$
For $\bp [i,j,i,j]$:
$$b_{ij}^{\bp [i,j,i,j]}c_{j}^{\bp [i,j,i,j]} = b_{ij}^{\bp [i,j,i]}c_{j}^{\bp [i,j,i]} > 0,$$
$$b_{ji}^{\bp [i,j,i,j]}c_{i}^{\bp [i,j,i,j]} = -b_{ji}^{\bp [i,j,i]}(c_{i}^{\bp [i,j,i]} + \sgn(b_{ij}^{\bp [i,j,i]})b_{ij}^{\bp [i,j,i]}c_{j}^{\bp [i,j,i]} ) > 0.$$
Hence,
$$r_i^{\bp [i,j,i,j,i]} = r_i^{\bp [i,j,i,j]} = r_j^{\bp [i,j,i]}r_i^{\bp [i,j,i]} r_j^{\bp [i,j,i]} = r_j^{\bp [i,j]}r_i^{\bp [i,j]}r_j^{\bp [i,j]} = r_j^{\bp [i]}r_i^{\bp [i]}r_j^{\bp [i]} = r_i^{\bp}r_j^{\bp}r_i^{\bp}r_j^{\bp}r_i^{\bp},$$
showing that $\pi(r_i^{\bp [i,j,i,j,i]}) = \pi(r_i^{\bp} r_j^{\bp})^3 \pi(r_j^{\bp})$.
We also have,
$$r_j^{\bp [i,j,i,j,i]} = r_i^{\bp [i,j,i,j]}r_j^{\bp [i,j,i,j]}r_i^{\bp [i,j,i,j]} = r_j^{\bp [i,j,i]}r_i^{\bp [i,j,i]} r_j^{\bp [i,j,i]} r_i^{\bp [i,j,i]} r_j^{\bp [i,j,i]}$$
$$= r_j^{\bp [i,j]}r_i^{\bp [i,j]}r_j^{\bp [i,j]}r_i^{\bp [i,j]}r_j^{\bp [i,j]} = r_j^{\bp [i]}r_i^{\bp [i]}r_j^{\bp [i]}r_i^{\bp [i]}r_j^{\bp [i]} = r_i^{\bp }r_j^{\bp }r_i^{\bp }r_j^{\bp }r_i^{\bp}r_j^{\bp}r_i^{\bp },$$
showing that $\pi(r_j^{\bp [i,j,i,j,i]}) = \pi(r_i^{\bp} r_j^{\bp})^3 \pi(r_i^{\bp})$.

Somewhat flipping the script, we begin at the top and assume that $b_{ij}^{\bp [i,j,i,j]} c_j^{\bp [i,j,i,j]} < 0$, as we can run into an identical ambiguity there.
For $\bp [i,j,i,j]$:
$$b_{ij}^{\bp [i,j,i,j]}c_{j}^{\bp [i,j,i,j]} < 0,$$
$$b_{ji}^{\bp [i,j,i,j]}c_{i}^{\bp [i,j,i,j]} = b_{ji}^{\bp [i,j,i,j,i]}c_{i}^{\bp [i,j,i,j,i]} > 0.$$
For $\bp [i,j,i]$:
$$b_{ij}^{\bp [i,j,i]}c_{j}^{\bp [i,j,i]} = b_{ij}^{\bp [i,j,i,j]}c_{j}^{\bp [i,j,i,j]} < 0,$$
$$b_{ji}^{\bp [i,j,i]}c_{i}^{\bp [i,j,i]} = -b_{ji}^{\bp [i,j,i,j]}c_{i}^{\bp [i,j,i,j]} < 0.$$
For $\bp [i,j]$:
$$b_{ij}^{\bp [i,j]}c_{j}^{\bp [i,j]} = -b_{ij}^{\bp [i,j,i]}c_{j}^{\bp [i,j,i]} > 0,$$
$$b_{ji}^{\bp [i,j]}c_{i}^{\bp [i,j]} = b_{ji}^{\bp [i,j,i]}c_{i}^{\bp [i,j,i]} < 0.$$
For $\bp [i]$:
$$b_{ij}^{\bp [i]}c_{j}^{\bp [i]} = b_{ij}^{\bp [i,j]}c_{j}^{\bp [i,j]} > 0,$$
$$b_{ji}^{\bp [i]}c_{i}^{\bp [i]} = -b_{ji}^{\bp [i,j]}(c_{i}^{\bp [i,j]} + \sgn(b_{ij}^{\bp [i,j]})b_{ij}^{\bp [i,j]}c_{j}^{\bp [i,j]})  > 0.$$
Hence,
$$r_i^{\bp [i,j,i,j,i]} = r_i^{\bp [i,j,i,j]} = r_j^{\bp [i,j,i]} = r_i^{\bp [i,j]} = r_j^{\bp [i]}r_i^{\bp [i]}r_j^{\bp [i]} = r_i^{\bp}r_j^{\bp}r_i^{\bp}r_j^{\bp}r_i^{\bp},$$
showing that $\pi(r_i^{\bp [i,j,i,j,i]}) = \pi(r_i^{\bp} r_j^{\bp})^3 \pi(r_j^{\bp})$.
We also have,
$$r_j^{\bp [i,j,i,j,i]} = r_i^{\bp [i,j,i,j]}r_j^{\bp [i,j,i,j]}r_i^{\bp [i,j,i,j]} =  r_i^{\bp [i,j,i]} r_j^{\bp [i,j,i]} r_i^{\bp [i,j,i]}   =  r_i^{\bp [i,j]}r_j^{\bp [i,j]}r_i^{\bp [i,j]} $$
$$ = r_j^{\bp [i]}r_i^{\bp [i]}r_j^{\bp [i]}r_i^{\bp [i]}r_j^{\bp [i]} = r_i^{\bp }r_j^{\bp }r_i^{\bp }r_j^{\bp }r_i^{\bp}r_j^{\bp}r_i^{\bp } .$$
showing that $\pi(r_j^{\bp [i,j,i,j,i]}) = \pi(r_i^{\bp} r_j^{\bp})^3 \pi(r_i^{\bp})$.
Thus the result holds for case (4) whenever $b_{ij}^{\bp [i]}c_{j}^{\bp [i]} < 0$ or $b_{ij}^{\bp [i,j,i,j]} c_j^{\bp [i,j,i,j]} < 0$.

Now, we use the fact that $b_{ij}^{\bp} = \pm 1$, as we are working with quivers of type $A_n$.
If $b_{ij}^\bp c_j^{\bp}$ and $b_{ji}^{\bp}c_i^{\bp}$ are both greater than zero, the two $c$-vectors must have opposite sign.
Additionally, we have that $c_i^{\bp[i,j,i,j,i]} = c_j^{\bp}$ and $c_j^{\bp[i,j,i,j,i]} = c_i^{\bp}$.
Since $b_{ij}^{\bp} = \pm 1$,
$$c_j^{\bp[i]} = c_j^{\bp} + \sgn(b_{ji}^{\bp})b_{ji}^{\bp}c_i^{\bp} = c_j^{\bp} + c_i^{\bp}$$
$$\text{and } c_j^{\bp[i,j,i,j]} = c_j^{\bp[i,j,i,j,i]} + \sgn(b_{ji}^{\bp[i,j,i,j,i]})b_{ji}^{\bp[i,j,i,j,i]}c_i^{\bp[i,j,i,j,i]} = c_i^{\bp} + \sgn(b_{ij}^{\bp})b_{ij}^{\bp}c_j^{\bp} = c_i^{\bp} + c_j^{\bp}.$$
Thus $c_j^{\bp[i]}$ and $c_j^{\bp[i,j,i,j]}$ share the same sign.
However, as $b_{ij}^{\bp[i]} = b_{ji}^{\bp[i,j,i,j]} = -b_{ij}^{\bp[i,j,i,j]}$, the two expressions  $b_{ij}^{\bp [i,j,i,j]} c_j^{\bp [i,j,i,j]}$ and $b_{ij}^{\bp [i]}c_{j}^{\bp [i]}$ have opposite sign, proving that one of the two is negative.
Hence, the result holds for Case (4).
\end{proof}

\section{Proof of Lemma \ref{base-case}} \label{proof-base-case}

We restate the lemma for convenience. 

\begin{Lem} 
	If $b_{ij} = 0$ for some $i,j \in \mathcal{I}$, then 
	$$\pi(r_ir_j)^2 = \mathrm{id} = \pi(r_jr_i)^2.$$
	If $|b_{ij}| = 1$, then 
	$$\pi(r_ir_j)^3 = \mathrm{id} = \pi(r_jr_i)^3.$$
	
	Additionally, if $k \neq i,j \in \mathcal{I}$ such that $b_{ki} = b_{ij} = b_{jk} = 1$ and the linear order satisfies
	$$i \prec k \prec j, \ j \prec i \prec k, \ \text{ or } \  k \prec j \prec i,$$
	then
	$$\pi(r_j r_i r_j r_k)^2 = \mathrm{id} = \pi(r_k r_j r_i r_j)^2.$$
	A similar identity holds for $b_{ki} = b_{ij} = b_{jk} = -1$ with 
	$$i \prec j \prec k, \ j \prec k \prec i, \ \text{ or } \ k \prec i \prec j.$$
	
\end{Lem}

\begin{proof}
	
	We start with $\pi(r_i) = I + A_i$, where $I$ is the $n\times n$ identity matrix and $A_i$ is the matrix whose $(i,k)$-element is $-a_{ki}$ and 0 elsewhere for all $k \in \mathcal{I}$, acting on elements of $\Gamma = \sum_{i=1}^n\mathbb{Z}\alpha_i$ on the left.
	Thus 
	$$\pi(r_ir_j) = I + A_iA_j + A_i + A_j = I + A_j + (A_iA_j + A_i).$$
	Let $B_{i} = A_iA_j + A_i$.
	Then the $(i,k)$-element of $B_{i}$ is $a_{ji}a_{kj} - a_{ki}$ and 0 elsewhere.
	Thus
	$$\pi(r_ir_j)^2 = I + 2A_j + 2B_{i} + A_j^2 + A_jB_i + B_iA_j + B_i^2$$
	$$ = I + (2A_j + A_jB_i + A_j^2) + (2B_i + B_iA_j + B_i^2).$$
	Let $C_i = 2B_i + B_iA_j + B_i^2$ and $D_j = 2A_j + A_jB_i + A_j^2$.
	Then the $(i,k)$-element of $C_i$ is $ a_{ji}^2a_{ij}a_{kj} -a_{ji}a_{kj}  - a_{ji}a_{ij}a_{ki}$ and 0 elsewhere.
	The $(j,k)$-element of $D_j$ is $a_{ij}a_{ki} - a_{ij}a_{ji}a_{kj}$ and 0 elsewhere ($2A_j = -A_j^2$).
	This also shows that $\pi(r_ir_j)^2 = I$ if $a_{ij} = 0$, meaning that $\pi(r_i)$ and $\pi(r_j)$ commute.
	Thus $b_{ij} = 0$ implies that $\pi(r_ir_j)^2 = \mathrm{id}$, proving the first identity.
	
	To continue on with the second identity, we have 
	$$\pi(r_ir_j)^3 = I + A_j + B_i + C_i + D_j+ A_jC_i + A_jD_j + B_iC_i + B_iD_j $$
	$$= I + (A_j + D_j + A_jC_i + A_jD_j) + (B_i + C_i + B_iC_i + B_iD_j).$$
	For $\pi(r_ir_j)^3 = I$, we require that $2A_j + A_jC_i + A_jD_j = 0 = 2B_i + B_iC_i + B_iD_j$.
	This necessitates the following equations for all $k \leq n$:
	$$2a_{ij}a_{ji}a_{kj} -a_{kj} - a_{ij}a_{ki} + a_{ji}a_{ij}^2a_{ki} -a_{ji}^2a_{ij}^2a_{kj} = 0$$
	$$2a_{ji}a_{kj} - a_{ki} - 3a_{ji}^2a_{ij}a_{kj} + 2a_{ji}a_{ij}a_{ki} + a_{ji}^3a_{ij}^2a_{kj} - a_{ji}^2a_{ij}^2a_{ki} = 0.$$
	
	Factoring produces:
	$$(2a_{ij}a_{ji} - 1 -a_{ji}^2a_{ij}^2)a_{kj}  + (a_{ji}a_{ij}^2 - a_{ij})a_{ki}  = 0$$
	$$(2a_{ji} - 3a_{ji}^2a_{ij}  + a_{ji}^3a_{ij}^2)a_{kj} + (2a_{ji}a_{ij}- a_{ji}^2a_{ij}^2 - 1)a_{ki}  = 0.$$
	If $|b_{ij}| = 1$, then $a_{ij} = a_{ji} = \pm 1$.
	Hence the above equations become:
	$$(2 - 1 -1)a_{kj}  + (a_{ij} -a_{ij})a_{ki}  = 0$$
	$$(2a_{ji} - 3a_{ji}  + a_{ji})a_{kj} + (2 - 1 - 1)a_{ki}  = 0.$$
	This demonstrates that $\pi(r_i r_j)^3 = \mathrm{id}$ for any linear ordering, as long as $|b_{ij}| = 1$, proving the second identity.
	
	On to the third identity, assume that $k \neq i,j \in \mathcal{I}$.
	Then
	$$\pi(r_jr_i) = I + A_i + (A_jA_i + A_j),$$
	and
	$$\pi(r_jr_k) = I + A_k + (A_jA_k + A_j).$$
	Let $B_{j,i} = A_jA_i + A_j$ and $B_{j,k} = A_jA_k + A_j$.
	Then
	$$\pi(r_jr_ir_jr_k) = (I + A_i + B_{j,i})(I + A_k + B_{j,k})$$
	$$ = I + A_k + B_{j,k} + A_i + A_iA_k + A_iB_{j,k} + B_{j,i} + B_{j,i}A_k + B_{j,i}B_{j,k}$$
	$$ = I + A_k + (A_i + A_iA_k + A_iB_{j,k}) + (B_{j,k} + B_{j,i} + B_{j,i}A_k + B_{j,i}B_{j,k}).$$
	
	Let $E_i = A_i + A_iA_k + A_iB_{j,k}$ and $F_j = B_{j,k} + B_{j,i} + B_{j,i}A_k + B_{j,i}B_{j,k}$.
	Then the $(i,l)$th element of $E_i$ is $-a_{li} + a_{ki}a_{lk} - a_{ji}a_{kj}a_{lk} + a_{ji}a_{lj}$ for any vertex $l \in \mathcal{I}$.
	The $(j,l)$th element of $F_j$ is 
	$$a_{kj}a_{lk} - a_{lj} + a_{ij}a_{li} - a_{lj}  - a_{ij}a_{ki}a_{lk} + a_{kj}a_{lk}  + a_{ij}a_{ji}a_{kj}a_{lk} - a_{jj}a_{kj}a_{lk} - a_{ij}a_{ji}a_{lj} + a_{jj}a_{lj}$$
	$$= a_{ij}a_{li} - a_{ij}a_{ki}a_{lk} + a_{ij}a_{ji}a_{kj}a_{lk} - a_{ij}a_{ji}a_{lj}.$$
	We can then write:
	$$\pi(r_jr_ir_jr_k)^2 = (I + A_k + E_i + F_j)^2$$
	$$= I + A_k + E_i + F_j + A_k + A_k^2 + A_kE_i + A_kF_j + E_i + E_iA_k +E_i^2 + E_iF_j + F_j + F_jA_k + F_jE_i + F_j^2$$
	$$= I + (2A_k + A_k^2 + A_kE_i + A_kF_j) + (2E_i + E_i^2 + E_iA_k  + E_iF_j) + (2F_j + F_jA_k + F_jE_i + F_j^2)$$
	$$= I  + A_k(E_i + F_j) + (2E_i + E_i^2 + E_iA_k  + E_iF_j) + (2F_j + F_j^2 + F_jA_k + F_jE_i).$$
	
	This implies three equations if $\pi(r_jr_ir_jr_k)^2 = I$.
	The first is 
	$$-a_{ik}(-a_{li} + a_{ki}a_{lk} - a_{ji}a_{kj}a_{lk} + a_{ji}a_{lj}) - a_{jk}(a_{ij}a_{li} - a_{ij}a_{ki}a_{lk} + a_{ij}a_{ji}a_{kj}a_{lk} - a_{ij}a_{ji}a_{lj}) = 0$$
	$$a_{ik}a_{li} - a_{ik}a_{ki}a_{lk} + a_{ik}a_{ji}a_{kj}a_{lk} - a_{ik}a_{ji}a_{lj} - a_{jk}a_{ij}a_{li} + a_{jk}a_{ij}a_{ki}a_{lk} - a_{jk} a_{ij}a_{ji}a_{kj}a_{lk} + a_{jk}a_{ij}a_{ji}a_{lj} = 0$$
	$$(a_{ik}- a_{jk}a_{ij})a_{li} + (a_{jk}a_{ij}a_{ji} - a_{ik}a_{ji})a_{lj} + (a_{ik}a_{ji}a_{kj}- a_{ik}a_{ki} + a_{jk}a_{ij}a_{ki} - a_{jk} a_{ij}a_{ji}a_{kj})a_{lk} = 0.$$
	
	The second is composed of the following:
	$$(2 -a_{ii} + a_{ki}a_{ik} - a_{ji}a_{kj}a_{ik} + a_{ji}a_{ij})(-a_{li} + a_{ki}a_{lk} - a_{ji}a_{kj}a_{lk} + a_{ji}a_{lj})$$
	$$= (a_{ki}a_{ik} - a_{ji}a_{kj}a_{ik} + a_{ji}a_{ij})(-a_{li} + a_{ki}a_{lk} - a_{ji}a_{kj}a_{lk} + a_{ji}a_{lj})$$
	$$ = (a_{ji}a_{kj}a_{ik}-a_{ki}a_{ik} - a_{ji}a_{ij})a_{li} + (a_{ki}a_{ik}a_{ji} - a_{ji}^2a_{kj}a_{ik} +a_{ji}^2a_{ij})a_{lj}$$
	$$+ (a_{ki}^2a_{ik}- 2a_{ji}a_{kj}a_{ik}a_{ki} + a_{ji}^2a_{kj}^2a_{ik} + a_{ji}a_{ij}a_{ki}- a_{ji}^2a_{ij}a_{kj})a_{lk},$$
	the expression
	$$(a_{ki} - a_{ji}a_{kj})(-a_{lk}) = (a_{ji}a_{kj} - a_{ki})a_{lk},$$
	and
	$$(a_{ki}a_{jk} - a_{ji}a_{kj}a_{jk} + a_{ji})(a_{ij}a_{li} - a_{ij}a_{ki}a_{lk} + a_{ij}a_{ji}a_{kj}a_{lk} - a_{ij}a_{ji}a_{lj})$$
	$$= (a_{ki}a_{jk}a_{ij} - a_{ji}a_{kj}a_{jk}a_{ij} + a_{ji}a_{ij})a_{li} + ( a_{ji}^2a_{kj}a_{jk}a_{ij} - a_{ki}a_{jk}a_{ij}a_{ji} - a_{ji}^2a_{ij})a_{lj}$$
	$$+ (2a_{ji}a_{kj}a_{jk}a_{ij}a_{ki} -a_{ki}^2a_{jk}a_{ij}  - a_{ji}a_{ij}a_{ki} - a_{ji}^2a_{kj}^2a_{jk}a_{ij} + a_{ji}^2a_{ij}a_{kj})a_{lk}.$$
	Thus
	$$(a_{ji}a_{kj}a_{ik} -a_{ki}a_{ik} + a_{ki}a_{jk}a_{ij} - a_{ji}a_{kj}a_{jk}a_{ij})a_{li} + (a_{ki}a_{ik}a_{ji} - a_{ji}^2a_{kj}a_{ik} +  a_{ji}^2a_{kj}a_{jk}a_{ij} - a_{ki}a_{jk}a_{ij}a_{ji})a_{lj}$$
	$$+ (a_{ki}^2a_{ik}- 2a_{ji}a_{kj}a_{ik}a_{ki} + a_{ji}^2a_{kj}^2a_{ik}+ a_{ji}a_{kj} - a_{ki} + 2a_{ji}a_{kj}a_{jk}a_{ij}a_{ki} -a_{ki}^2a_{jk}a_{ij} - a_{ji}^2a_{kj}^2a_{jk}a_{ij})a_{lk} = 0.$$
	
	The third is composed of the following:
	$$(2 - a_{ij}a_{ki}a_{jk} + a_{ij}a_{ji}a_{kj}a_{jk} - a_{ij}a_{ji})(a_{ij}a_{li} - a_{ij}a_{ki}a_{lk} + a_{ij}a_{ji}a_{kj}a_{lk} - a_{ij}a_{ji}a_{lj})$$
	$$= (2a_{ij} - a_{ij}^2a_{ki}a_{jk} + a_{ij}^2a_{ji}a_{kj}a_{jk} - a_{ij}^2a_{ji})a_{li} + (a_{ij}^2a_{ki}a_{jk}a_{ji} - 2a_{ij}a_{ji} - a_{ij}^2a_{ji}^2a_{kj}a_{jk} + a_{ij}^2a_{ji}^2)a_{lj}$$
	$$+ (a_{ij}^2a_{ki}^2a_{jk} - 2a_{ij}a_{ki} - 2a_{ij}^2a_{ji}a_{kj}a_{jk}a_{ki} + a_{ij}^2a_{ji}a_{ki} + 2a_{ij}a_{ji}a_{kj} + a_{ij}^2a_{ji}^2a_{kj}^2a_{jk} - a_{ij}^2a_{ji}^2a_{kj})a_{lk},$$
	the expression
	$$(a_{ij}a_{ji}a_{kj}-a_{ij}a_{ki})(-a_{lk}) = (a_{ij}a_{ki} - a_{ij}a_{ji}a_{kj})a_{lk},$$
	and
	$$(2a_{ij} - a_{ij}a_{ki}a_{ik} + a_{ij}a_{ji}a_{kj}a_{ik} - a_{ij}^2a_{ji})(-a_{li} + a_{ki}a_{lk} - a_{ji}a_{kj}a_{lk} + a_{ji}a_{lj})$$
	$$= (a_{ij}a_{ki}a_{ik} - 2a_{ij} - a_{ij}a_{ji}a_{kj}a_{ik} + a_{ij}^2a_{ji})a_{li} + (2a_{ij}a_{ji} - a_{ij}a_{ki}a_{ik}a_{ji} + a_{ij}a_{ji}^2a_{kj}a_{ik} - a_{ij}^2a_{ji}^2)a_{lj}$$
	$$+(2a_{ij}a_{ki} - a_{ij}a_{ki}^2a_{ik} + 2a_{ij}a_{ji}a_{kj}a_{ik}a_{ki} - a_{ij}^2a_{ji}a_{ki} - 2a_{ij}a_{ji}a_{kj} - a_{ij}a_{ji}^2a_{kj}^2a_{ik} + a_{ij}^2a_{ji}^2a_{kj})a_{lk}.$$
	Thus
	\begin{align*}& (a_{ij}^2a_{ji}a_{kj}a_{jk}- a_{ij}^2a_{ki}a_{jk} + a_{ij}a_{ki}a_{ik} - a_{ij}a_{ji}a_{kj}a_{ik})a_{li}\\&+ (a_{ij}^2a_{ki}a_{jk}a_{ji} - a_{ij}^2a_{ji}^2a_{kj}a_{jk}- a_{ij}a_{ki}a_{ik}a_{ji} + a_{ij}a_{ji}^2a_{kj}a_{ik})a_{lj}\\
	&+  (a_{ij}^2a_{ki}^2a_{jk} + a_{ij}a_{ki} - 2a_{ij}^2a_{ji}a_{kj}a_{jk}a_{ki}  - a_{ij}a_{ji}a_{kj}) a_{lk}  \\&+( a_{ij}^2a_{ji}^2a_{kj}^2a_{jk}  - a_{ij}a_{ki}^2a_{ik} + 2a_{ij}a_{ji}a_{kj}a_{ik}a_{ki} - a_{ij}a_{ji}^2a_{kj}^2a_{ik}  )a_{lk} = 0.\end{align*}
	
	Now, let's assume $b_{ki} = b_{ij} = b_{jk} = \pm 1$, which implies that
	$$|a_{ij}| = |a_{ji}| = |a_{ik}| = |a_{ki}| = |a_{jk}| = |a_{kj}| = 1.$$
	Note that the $3 \times 3$ submatrix defined by these indices is a symmetric GIM.
	Then the three above equations reduce to:
	$$(a_{ik}- a_{jk}a_{ij})a_{li} + (a_{jk} - a_{ik}a_{ji})a_{lj} + 2(a_{ik}a_{ji}a_{kj} - 1)a_{lk} = 0;$$
	$$2(a_{ji}a_{kj}a_{ik} - 1)a_{li} + 2(a_{ji} - a_{ki}a_{jk})a_{lj} + 3(a_{ik} - a_{jk}a_{ij})a_{lk} = 0;$$
	$$2(a_{ji}- a_{kj}a_{ik})a_{li} + 2(a_{ki}a_{jk}a_{ji} - 1)a_{lj} + 3(a_{jk}- a_{ij}a_{ik})a_{lk} = 0.$$
	Note that the above equations hold if $a_{ki}$, $a_{ij}$, and $a_{jk}$ are either all positive or only one of them is.
	
	The six possible linear orderings are
	\begin{align*} (1)& \, i \prec j \prec k, &
	(2)& \, i \prec k \prec j, &
	(3)& \, j \prec i \prec k, \\ 
	(4)& \, j \prec k \prec i, &
	(5)& \, k \prec i \prec j,&
	(6)& \,  k \prec j \prec i. \end{align*}
	For linear orderings (2), (3), and (6), the above equations hold if  $b_{ki} = b_{ij} = b_{jk} = 1$.
	For linear orderings (1), (4), and (5), the above equations hold if  $b_{ki} = b_{ij} = b_{jk} = -1$.
	Therefore $\pi(r_j r_i r_j r_k)^2 = \mathrm{id}$ under the appropriate linear orderings, completing our lemma.
	
\end{proof}

\section{Discussion of Lemma \ref{swap-indices-k}} \label{discussion-swap-indices-k}

We restate the lemma for convenience. 

\begin{Lem}
	Suppose that $\bv = [i,j]_{\bp}$ is an elementary swap for some $i,j \in \mathcal{I}$ and for a mutation sequence $\bp$.
	Let $k \neq i,j \in \mathcal{I}$.
	Then the following hold:
	\begin{enumerate}
		\item[(A)] If $b_{ki}^{\bp} = 0 = b_{kj}^{\bp}$, then 
		$$\pi(r_k^{\bp[i,j,i,j,i]}) = \pi(r_k^{\bp}).$$
		\item[(B)] If $b_{ki}^{\bp} = 0$ and $b_{kj}^{\bp} \neq 0$, then
		\begin{align*} \pi(r_k^{\bp [i,j,i,j,i]}) = &\ \pi(r_k^{\bp}), \quad \pi(r_i^\bp r_j^\bp)^3\pi(r_k^{\bp}) \pi(r_j^\bp r_i^\bp)^3, \quad  \pi(r_j^\bp r_i^\bp)^3\pi(r_k^{\bp}) \pi(r_i^\bp r_j^\bp)^3, \\ &\  \pi(r_i^{\bp}r_k^{\bp})^2 \pi(r_k^{\bp}), \
		\text{ or } \  \pi(r_i^\bp r_j^\bp)^3\pi(r_i^\bp r_k^{\bp})^2 \pi(r_k^\bp)\pi(r_j^\bp r_i^\bp)^3.\end{align*}
		\item[(C)] If $b_{ki}^{\bp} \neq 0$ and $b_{kj}^{\bp} = 0$, then
		\begin{align*} \pi(r_k^{\bp [i,j,i,j,i]}) =& \ \pi(r_k^{\bp}), \quad \pi(r_i^\bp r_j^\bp)^3\pi(r_k^{\bp}) \pi(r_j^\bp r_i^\bp)^3, \quad \pi(r_j^\bp r_i^\bp)^3\pi(r_k^{\bp}) \pi(r_i^\bp r_j^\bp)^3, \\ & \ \pi(r_j^{\bp}r_k^{\bp})^2 \pi(r_k^{\bp}), \ \text{ or } \ \pi(r_i^\bp r_j^\bp)^3\pi(r_j^\bp r_k^{\bp})^2 \pi(r_k^\bp)\pi(r_j^\bp r_i^\bp)^3. \end{align*}
		\item[(D)] If $b_{ki}^{\bp} \neq 0$ and $b_{kj}^{\bp} \neq 0$, then
		$$\pi(r_k^{\bp[i,j,i,j,i]}) = \pi(r_k^{\bp}) \ \text{ or } \ \pi(r_j^\bp r_i^\bp r_j^\bp r_k^{\bp})^2 \pi(r_k^\bp).$$
	\end{enumerate}
\end{Lem}

Suppose $\bv = [i,j]_{\bp}$ is an elementary swap for some $i,j \in \mathcal{I}$ and for a mutation sequence $\bp$.
We wish to find some analogue of Lemma \ref{swap-indices-i,j} for $\pi(r_k^{\bp [i,j,i,j,i]})$ in terms of $\pi(r_k^\bp), \pi(r_i^\bp),$ and $\pi(r_j^\bp)$ for any index $k \in \mathcal{I}$ such that $k \neq i,j$.
Then we have eighteen cases, each with five sub-cases corresponding to the cases and sub-cases in Lemma \ref{swap-indices-i,j}.
\begin{align*}(1)& \, b_{ki}^{\bp} =0, \, b_{kj}^{\bp} = 0, \, b_{ij}^{\bp} > 0, \, \mathrm{ and }\, b_{ji}^{\bp} < 0, & 
(2)& \, b_{ki}^{\bp} =0, \, b_{kj}^{\bp} = 0, \, b_{ij}^{\bp} < 0, \, \mathrm{ and }\, b_{ji}^{\bp} > 0,\\
(3)& \, b_{ki}^{\bp} =0, \, b_{kj}^{\bp} > 0, \, b_{ij}^{\bp} > 0, \, \mathrm{ and }\, b_{ji}^{\bp} < 0, &
(4)& \, b_{ki}^{\bp} =0, \, b_{kj}^{\bp} > 0, \, b_{ij}^{\bp} < 0, \, \mathrm{ and }\, b_{ji}^{\bp} > 0,\\ 
(5)& \, b_{ki}^{\bp} =0, \, b_{kj}^{\bp} < 0, \, b_{ij}^{\bp} > 0, \, \mathrm{ and }\, b_{ji}^{\bp} < 0, &
(6)& \, b_{ki}^{\bp} =0, \, b_{kj}^{\bp} < 0, \, b_{ij}^{\bp} < 0, \, \mathrm{ and }\, b_{ji}^{\bp} > 0,\\ 
(7)& \, b_{ki}^{\bp} >0, \, b_{kj}^{\bp} = 0, \, b_{ij}^{\bp} > 0, \, \mathrm{ and }\, b_{ji}^{\bp} < 0,&
(8)& \, b_{ki}^{\bp} >0, \, b_{kj}^{\bp} = 0, \, b_{ij}^{\bp} < 0, \, \mathrm{ and }\, b_{ji}^{\bp} > 0,\\ 
(9)& \, b_{ki}^{\bp} <0, \, b_{kj}^{\bp} = 0, \, b_{ij}^{\bp} > 0, \, \mathrm{ and }\, b_{ji}^{\bp} < 0, &
(10)& \, b_{ki}^{\bp} <0, \, b_{kj}^{\bp} = 0, \, b_{ij}^{\bp} < 0, \, \mathrm{ and }\, b_{ji}^{\bp} > 0,\\ 
(11)& \, b_{ki}^{\bp} >0, \, b_{kj}^{\bp} > 0, \, b_{ij}^{\bp} > 0, \, \mathrm{ and }\, b_{ji}^{\bp} < 0, &
(12)& \, b_{ki}^{\bp} >0, \, b_{kj}^{\bp} > 0, \, b_{ij}^{\bp} < 0, \, \mathrm{ and }\, b_{ji}^{\bp} > 0,\\ 
(13)& \, b_{ki}^{\bp} <0, \, b_{kj}^{\bp} > 0, \, b_{ij}^{\bp} > 0, \, \mathrm{ and }\, b_{ji}^{\bp} < 0, &
(14)& \, b_{ki}^{\bp} <0, \, b_{kj}^{\bp} > 0, \, b_{ij}^{\bp} < 0, \, \mathrm{ and }\, b_{ji}^{\bp} > 0,\\ 
(15)& \, b_{ki}^{\bp} >0, \, b_{kj}^{\bp} < 0, \, b_{ij}^{\bp} > 0, \, \mathrm{ and }\, b_{ji}^{\bp} < 0, &
(16)& \, b_{ki}^{\bp} >0, \, b_{kj}^{\bp} < 0, \, b_{ij}^{\bp} < 0, \, \mathrm{ and }\, b_{ji}^{\bp} > 0,\\ 
(17)& \, b_{ki}^{\bp} <0, \, b_{kj}^{\bp} < 0, \, b_{ij}^{\bp} > 0, \, \mathrm{ and }\, b_{ji}^{\bp} < 0, &
(18)& \, b_{ki}^{\bp} <0, \, b_{kj}^{\bp} < 0, \, b_{ij}^{\bp} < 0, \, \mathrm{ and }\, b_{ji}^{\bp} > 0. \end{align*}
Note that cases (1) and (2) are the easiest, having $\pi(r_k^{\bp [i,j,i,j,i]}) = \pi(r_k^{\bp })$.
Additionally, the cases (11), (12), (13), (16), (17), and (18) cannot occur in a type $A_n$ quiver, as all triangles must be cycles. 
As such, we may discard these cases.
They may return when observing arbitrary cluster algebras related to bordered surfaces, but we do not need them for our current purposes.
The following cases remain:
$$(3) ,\ (4)  ,\ (5)  , \ (6) ,\ (7) ,\ (8) ,\ (9) ,\ (10) ,\ (14) ,\ (15). $$

In type $A_n$ quivers, each case will correspond with exactly one of six subquivers, up to permutation of the indices $i$ and $j$.
The first quiver belongs to the cases (1) and (2). 

\begin{center}
	\includegraphics[scale=0.5]{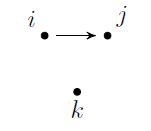}
	\end{center}

\begin{Lem}[Case (1)] \label{lem-case-1}
	Suppose that $\bv = [i,j]_{\bp}$ is an elementary swap for some $i,j \in \mathcal{I}$ and for a mutation sequence $\bp$.
	If $k \neq i,j \in \mathcal{I}$ satisfies
	$$b_{ki}^{\bp} =0, \, b_{kj}^{\bp} = 0, \, b_{ij}^{\bp} > 0, \, \mathrm{ and }\, b_{ji}^{\bp} < 0,$$
	then $\pi(r_k^{\bp [i,j,i,j,i]}) = \pi(r_k^{\bp})$.
\end{Lem}

\begin{Lem}[Case (2)] \label{lem-case-2}
	Suppose that $\bv = [i,j]_{\bp}$ is an elementary swap for some $i,j \in \mathcal{I}$ and for a mutation sequence $\bp$.
	If $k \neq i,j \in \mathcal{I}$ satisfies
	$$b_{ki}^{\bp} =0, \, b_{kj}^{\bp} = 0, \, b_{ij}^{\bp} < 0, \, \mathrm{ and }\, b_{ji}^{\bp} > 0,$$ then $\pi(r_k^{\bp [i,j,i,j,i]}) = \pi(r_k^{\bp})$.	
\end{Lem}

The second quiver belongs to the cases (5) and (10).

\begin{center}
	\includegraphics[scale=0.5]{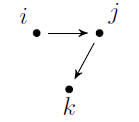}
	\end{center}

\begin{Lem}[Case (5)] \label{lem-case-5}
	Suppose that $\bv = [i,j]_{\bp}$ is an elementary swap for some $i,j \in \mathcal{I}$ and for a mutation sequence $\bp$.
	If $k \neq i,j \in \mathcal{I}$ satisfies
	$$b_{ki}^{\bp} =0, \, b_{kj}^{\bp} < 0, \, b_{ij}^{\bp} > 0, \, \mathrm{ and }\, b_{ji}^{\bp} < 0,$$
	then $\pi(r_k^{\bp [i,j,i,j,i]}) = \pi(r_k^{\bp})$, $\pi(r_i^\bp r_j^\bp)^3\pi(r_k^{\bp}) \pi(r_j^\bp r_i^\bp)^3$, or $\pi(r_i^{\bp}r_k^{\bp})^2 \pi(r_k^{\bp})$.
\end{Lem}

\begin{Lem}[Case (10)] \label{lem-case-10}
	Suppose that $\bv = [i,j]_{\bp}$ is an elementary swap for some $i,j \in \mathcal{I}$ and for a mutation sequence $\bp$.
	If $k \neq i,j \in \mathcal{I}$ satisfies
	$$b_{ki}^{\bp} <0, \, b_{kj}^{\bp} = 0, \, b_{ij}^{\bp} < 0, \, \mathrm{ and }\, b_{ji}^{\bp} > 0,$$ then $\pi(r_k^{\bp [i,j,i,j,i]}) = \pi(r_k^{\bp})$, $\pi(r_j^\bp r_i^\bp)^3\pi(r_k^{\bp}) \pi(r_i^\bp r_j^\bp)^3$, or $\pi(r_i^\bp r_j^\bp)^3\pi(r_j^\bp r_k^{\bp})^2 \pi(r_k^\bp)\pi(r_j^\bp r_i^\bp)^3$.	
\end{Lem}

The third quiver belongs to the cases (3) and (8).

\begin{center}
	\includegraphics[scale=0.5]{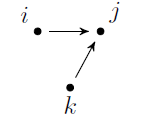}
	\end{center}

\begin{Lem}[Case (3)] \label{lem-case-3}
	Suppose that $\bv = [i,j]_{\bp}$ is an elementary swap for some $i,j \in \mathcal{I}$ and for a mutation sequence $\bp$.
	If $k \neq i,j \in \mathcal{I}$ satisfies
	$$b_{ki}^{\bp} =0, \, b_{kj}^{\bp} > 0, \, b_{ij}^{\bp} > 0, \, \mathrm{ and }\, b_{ji}^{\bp} < 0,$$ then $\pi(r_k^{\bp [i,j,i,j,i]}) = \pi(r_k^{\bp})$, $\pi(r_i^\bp r_j^\bp)^3\pi(r_k^{\bp}) \pi(r_j^\bp r_i^\bp)^3$, or $\pi(r_i^\bp r_j^\bp)^3\pi(r_i^\bp r_k^{\bp})^2 \pi(r_k^\bp)\pi(r_j^\bp r_i^\bp)^3$.
\end{Lem}

\begin{Lem}[Case (8)] \label{lem-case-8}
	Suppose that $\bv = [i,j]_{\bp}$ is an elementary swap for some $i,j \in \mathcal{I}$ and for a mutation sequence $\bp$.
	If $k \neq i,j \in \mathcal{I}$ satisfies
	$$b_{ki}^{\bp} >0, \, b_{kj}^{\bp} = 0, \, b_{ij}^{\bp} < 0, \, \mathrm{ and }\, b_{ji}^{\bp} > 0,$$ then $\pi(r_k^{\bp [i,j,i,j,i]}) = \pi(r_k^{\bp})$, $\pi(r_i^\bp r_j^\bp)^3\pi(r_k^{\bp}) \pi(r_j^\bp r_i^\bp)^3$, or $\pi(r_j^\bp r_k^{\bp})^2 \pi(r_k^\bp)$.	
\end{Lem}

The fourth quiver belongs to the cases (9) and (6).
\begin{center}
	\includegraphics[scale=0.5]{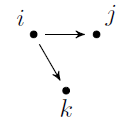}
	\end{center}

\begin{Lem}[Case (6)] \label{lem-case-6}
	Suppose that $\bv = [i,j]_{\bp}$ is an elementary swap for some $i,j \in \mathcal{I}$ and for a mutation sequence $\bp$.
	If $k \neq i,j \in \mathcal{I}$ satisfies
	$$b_{ki}^{\bp} = 0, \, b_{kj}^{\bp} < 0, \, b_{ij}^{\bp} < 0, \, \mathrm{ and }\, b_{ji}^{\bp} > 0,$$ then $\pi(r_k^{\bp [i,j,i,j,i]}) = \pi(r_k^{\bp})$ or $\pi(r_i^\bp r_j^\bp)^3\pi(r_k^{\bp})\pi(r_j^\bp r_i^\bp)^3$.
\end{Lem}

\begin{Lem}[Case (9)] \label{lem-case-9}
	Suppose that $\bv = [i,j]_{\bp}$ is an elementary swap for some $i,j \in \mathcal{I}$ and for a mutation sequence $\bp$.
	If $k \neq i,j \in \mathcal{I}$ satisfies
	$$b_{ki}^{\bp} < 0, \, b_{kj}^{\bp} = 0, \, b_{ij}^{\bp} > 0, \, \mathrm{ and }\, b_{ji}^{\bp} < 0,$$ then $\pi(r_k^{\bp [i,j,i,j,i]}) = \pi(r_k^{\bp})$, $\pi(r_j^\bp r_i^\bp)^3\pi(r_k^{\bp})\pi(r_i^\bp r_j^\bp)^3$, or $\pi(r_j^{\bp} r_k^{\bp})^2 \pi(r_k^{\bp})$.
\end{Lem}

The fifth belongs to the cases (7) and (4).

\begin{center}
	\includegraphics[scale=0.5]{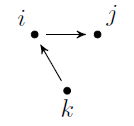}
	\end{center}

\begin{Lem}[Case (4)] \label{lem-case-4}
	Suppose that $\bv = [i,j]_{\bp}$ is an elementary swap for some $i,j \in \mathcal{I}$ and for a mutation sequence $\bp$.
	If $k \neq i,j \in \mathcal{I}$ satisfies
	$$b_{ki}^{\bp} = 0, \, b_{kj}^{\bp} > 0, \, b_{ij}^{\bp} < 0, \, \mathrm{ and }\, b_{ji}^{\bp} > 0,$$ then $\pi(r_k^{\bp [i,j,i,j,i]}) = \pi(r_k^{\bp})$ or $\pi(r_j^\bp r_i^\bp)^3\pi(r_k^{\bp})\pi(r_i^\bp r_j^\bp)^3$.
\end{Lem}

\begin{Lem}[Case (7)] \label{lem-case-7}
	Suppose that $\bv = [i,j]_{\bp}$ is an elementary swap for some $i,j \in \mathcal{I}$ and for a mutation sequence $\bp$.
	If $k \neq i,j \in \mathcal{I}$ satisfies
	$$b_{ki}^{\bp} > 0, \, b_{kj}^{\bp} = 0, \, b_{ij}^{\bp} > 0, \, \mathrm{ and }\, b_{ji}^{\bp} < 0,$$ then $\pi(r_k^{\bp [i,j,i,j,i]}) = \pi(r_k^{\bp})$, $\pi(r_i^{\bp}r_j^{\bp})^3\pi(r_j^{\bp}r_k^{\bp})^2 \pi(r_k^{\bp})\pi(r_j^{\bp}r_i^{\bp})^3$, or $\pi(r_j^\bp r_i^\bp)^3\pi(r_k^{\bp})\pi(r_i^\bp r_j^\bp)^3$.
\end{Lem}

Finally, the sixth quiver belongs to the cases (15) and (14).

\begin{center}
	\includegraphics[scale=0.5]{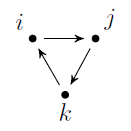}
	\end{center}

\begin{Lem}[Case (14)] \label{lem-case-14}
	Suppose that $\bv = [i,j]_{\bp}$ is an elementary swap for some $i,j \in \mathcal{I}$ and for a mutation sequence $\bp$.
	If $k \neq i,j \in \mathcal{I}$ satisfies
	$$b_{ki}^{\bp} <0, \, b_{kj}^{\bp} > 0, \, b_{ij}^{\bp} < 0, \, \mathrm{ and }\, b_{ji}^{\bp} > 0,$$ then $\pi(r_k^{\bp [i,j,i,j,i]}) = \pi(r_k^{\bp})$ or $\pi(r_k^{\bp})\pi(r_k^{\bp}r_j^{\bp}r_i^{\bp}r_j^{\bp})^2$.	
\end{Lem}

\begin{Lem}[Case (15)] \label{lem-case-15}
	Suppose that $\bv = [i,j]_{\bp}$ is an elementary swap for some $i,j \in \mathcal{I}$ and for a mutation sequence $\bp$.
	If $k \neq i,j \in \mathcal{I}$ satisfies
	$$b_{ki}^{\bp} >0, \, b_{kj}^{\bp} < 0, \, b_{ij}^{\bp} > 0, \, \mathrm{ and }\, b_{ji}^{\bp} < 0,$$ then $\pi(r_k^{\bp [i,j,i,j,i]}) = \pi(r_k^{\bp})$ or $\pi(r_k^{\bp})\pi(r_k^{\bp}r_j^{\bp}r_i^{\bp}r_j^{\bp})^2$.	
\end{Lem}

Seeing what quivers belong to each case allows us to calculate exactly what values the entries $b_{ki}^{\bp}$ and $b_{kj}^{\bp}$ take throughout the mutation $[i,j,i,j,i]$.
We can now prove each lemma in the following section.
The combination of Lemmas \ref{lem-case-1} through \ref{lem-case-15} proves Lemma \ref{swap-indices-k}.

\section{Proof of Lemma \ref{swap-indices-k}} \label{proof-swap-indices-k}

To begin the proof, we need to consider the signs of $c_i^\bp$ and $c_j^\bp$ according to the possible outcomes outlined in Lemma \ref{swap-indices-i,j}.
In what follows, we write $(b_{ij}^{\bp},...,b_{ij}^{\bp[i,j,i,j]})$ for $(b_{ij}^{\bp},b_{ij}^{\bp[i]},b_{ij}^{\bp[i,j]},b_{ij}^{\bp[i,j,i]},b_{ij}^{\bp[i,j,i,j]})$ and $(+,-,+,-,+)$ records the signs of the values. We will use similar shorthand notations throughout this section. 

Note the values of $b_{ij}^{\bp}$ and $b_{ji}^{\bp}$ throughout the mutation $[i,j,i,j,i]$. If $b_{ij}^{\bp} > 0$ and $b_{ji}^{\bp} < 0$, we have
$$
(b_{ij}^{\bp},...,b_{ij}^{\bp[i,j,i,j]})=(+,-,+,-,+) \ \text{ and }\ (b_{ji}^{\bp},...,b_{ji}^{\bp[i,j,i,j]})=(-,+,-,+,-);
$$ if $b_{ij}^{\bp} < 0$ and $b_{ji}^{\bp} > 0$, we have 
$$
(b_{ij}^{\bp},...,b_{ij}^{\bp[i,j,i,j]})=(-,+,-,+,-) \ \text{ and }\ (b_{ji}^{\bp},...,b_{ji}^{\bp[i,j,i,j]})=(+,-,+,-,+).
$$
We replicate them here for ease of use and continue on to the signs of $c_i^{\bp}$ and $c_j^{\bp}$.

\noindent \underline{Case $(b_{ij}^{\bp}c_j^{\bp},b_{ji}^{\bp}c_i^{\bp})=(+,-)$}
$$
(b_{ij}^{\bp [i]}c_{j}^{\bp [i]},...,b_{ij}^{\bp [i,j,i,j]}c_{j}^{\bp [i,j,i,j]} )=(-,-,+,+)\text{ and } (b_{ji}^{\bp [i]}c_{i}^{\bp [i]},...,b_{ji}^{\bp [i,j,i,j]}c_{i}^{\bp [i,j,i,j]})=(-,+,+,+)
.$$

If $(b_{ij}^{\bp},b_{ji}^{\bp})=(+,-)$ then
$
(c_j^{\bp},...,c_j^{\bp[i,j,i,j]})=(+,+,-,-,+)$ and $(c_i^{\bp},...,c_i^{\bp[i,j,i,j]})=(+,-,-,+,-).
$

If $(b_{ij}^{\bp},b_{ji}^{\bp})=(-,+)$ then
$
(c_j^{\bp},...,c_j^{\bp[i,j,i,j]})=(-,-,+,+,-)$ and $(c_i^{\bp},...,c_i^{\bp[i,j,i,j]})=(-,+,+,-,+).
$

\noindent\underline{Case $(b_{ij}^{\bp}c_j^{\bp},b_{ji}^{\bp}c_i^{\bp})=(-,+)$}
$$
(...,b_{ij}^{\bp [i,j,i,j]}c_{j}^{\bp [i,j,i,j]} )=(+,+,-,-)\text{ and }(...,b_{ji}^{\bp [i,j,i,j]}c_{i}^{\bp [i,j,i,j]})=(+,+,+,-).
$$


If $(b_{ij}^{\bp},b_{ji}^{\bp})=(+,-)$ then
$
(c_j^{\bp},...,c_j^{\bp[i,j,i,j]})=(-,-,+,+,-)$ and $(c_i^{\bp},...,c_i^{\bp[i,j,i,j]})=(-,+,-,+,+).
$


If $(b_{ij}^{\bp},b_{ji}^{\bp})=(-,+)$ then
$
(c_j^{\bp},...,c_j^{\bp[i,j,i,j]})=(+,+,-,-,+)$ and $(c_i^{\bp},...,c_i^{\bp[i,j,i,j]})=(+,-,+,-,-).
$


\noindent\underline{Case $(b_{ij}^{\bp}c_j^{\bp},b_{ji}^{\bp}c_i^{\bp})=(-,-)$}
$$
(...,b_{ij}^{\bp [i,j,i,j]}c_{j}^{\bp [i,j,i,j]} )=(+,+,+,+)\text{ and }(...,b_{ji}^{\bp [i,j,i,j]}c_{i}^{\bp [i,j,i,j]})=(-,+,+,-).
$$


If $(b_{ij}^{\bp},b_{ji}^{\bp})=(+,-)$ then
$
(c_j^{\bp},...,c_j^{\bp[i,j,i,j]})=(-,-,+,-,+)$ and $(c_i^{\bp},...,c_i^{\bp[i,j,i,j]})=(+,-,-,+,+).
$


If $(b_{ij}^{\bp},b_{ji}^{\bp})=(-,+)$ then
$
(c_j^{\bp},...,c_j^{\bp[i,j,i,j]})=(+,+,-,+,-)$ and $(c_i^{\bp},...,c_i^{\bp[i,j,i,j]})=(-,+,+,-,-).
$


\noindent\underline{Case $(b_{ij}^{\bp}c_j^{\bp},b_{ji}^{\bp}c_i^{\bp},b_{ij}^{\bp[i]}c_j^{\bp[i]})=(+,+,-)$}
$$
(...,b_{ij}^{\bp [i,j,i,j]}c_{j}^{\bp [i,j,i,j]} )=(-,-,+,+)\text{ and }(...,b_{ji}^{\bp [i,j,i,j]}c_{i}^{\bp [i,j,i,j]})=(+,-,-,+).
$$


If $(b_{ij}^{\bp},b_{ji}^{\bp})=(+,-)$ then
$
(c_j^{\bp},...,c_j^{\bp[i,j,i,j]})=(+,+,-,-,+)$ and $(c_i^{\bp},...,c_i^{\bp[i,j,i,j]})=(-,+,+,-,-).
$


If $(b_{ij}^{\bp},b_{ji}^{\bp})=(-,+)$ then
$
(c_j^{\bp},...,c_j^{\bp[i,j,i,j]})=(-,-,+,+,-)$ and $(c_i^{\bp},...,c_i^{\bp[i,j,i,j]})=(+,-,-,+,+).
$


\noindent\underline{Case $(b_{ij}^{\bp}c_j^{\bp},b_{ji}^{\bp}c_i^{\bp},b_{ij}^{\bp[i,j,i,j]}c_j^{\bp[i,j,i,j]})=(+,+,-)$}
$$
(...,b_{ij}^{\bp [i,j,i,j]}c_{j}^{\bp [i,j,i,j]} )=(+,+,-,-)\text{ and }(...,b_{ji}^{\bp [i,j,i,j]}c_{i}^{\bp [i,j,i,j]})=(+,-,-,+).
$$


If $(b_{ij}^{\bp},b_{ji}^{\bp})=(+,-)$ then
$
(c_j^{\bp},...,c_j^{\bp[i,j,i,j]})=(+,-,+,+,-)$ and $(c_i^{\bp},...,c_i^{\bp[i,j,i,j]})=(-,+,+,-,-).
$


If $(b_{ij}^{\bp},b_{ji}^{\bp})=(-,+)$ then
$
(c_j^{\bp},...,c_j^{\bp[i,j,i,j]})=(-,+,-,-,+)$ and $(c_i^{\bp},...,c_i^{\bp[i,j,i,j]})=(+,-,-,+,+).
$


\noindent Now that the book-keeping is taken care of, we can compute $r_k^{\bp[i,j,i,j,i]}$ in each case as required:

\begin{proof}[\boxed{Proof\ of\ Lemma\ \ref{lem-case-5}}]
	We need to keep track of $b_{ki}^\bp, b_{kj}^\bp, b_{ij}^\bp, b_{ji}^\bp, c_i^\bp,$ and $c_j^\bp$ throughout the entire mutation $[i,j,i,j,i]$.
	Let us list all the possible values of the first four variables.
	{\tiny $$
		(...,b_{ki}^{\bp[i,j,i,j]})=(0,0,0,0,+);\, (...,b_{kj}^{\bp[i,j,i,j]})=(-,-,+,+,-);\, (...,b_{ij}^{\bp[i,j,i,j]})=(+,-,+,-,+);\, (...,b_{ji}^{\bp[i,j,i,j]})=(-,+,-,+,-)
		$$}
	
	
	\noindent \underline{Case $(b_{ij}^{\bp}c_j^{\bp},b_{ji}^{\bp}c_i^{\bp})=(+,-)$} The values we are concerned with are as follows:
	$$
	(b_{ki}^{\bp}c_i^{\bp}, \ b_{kj}^{\bp[i]}c_j^{\bp[i]}, \ b_{ki}^{\bp[i,j]}c_i^{\bp[i,j]}, \ b_{kj}^{\bp[i,j,i]}c_j^{\bp[i,j,i]}, \ b_{ki}^{\bp[i,j,i,j]}c_i^{\bp[i,j,i,j]}) = (0,-,0,-,-).
	$$
	
	Hence
	$r_k^{\bp[i,j,i,j,i]} = r_k^{\bp},$
	and  
	$\pi(r_k^{\bp[i,j,i,j,i]} ) = \pi(r_k^{\bp}).$
	
	\noindent \underline{Case $(b_{ij}^{\bp}c_j^{\bp},b_{ji}^{\bp}c_i^{\bp})=(-,+)$} The values we are concerned with are as follows:
	$$
	(b_{ki}^{\bp}c_i^{\bp}, \ b_{kj}^{\bp[i]}c_j^{\bp[i]}, \ b_{ki}^{\bp[i,j]}c_i^{\bp[i,j]}, \ b_{kj}^{\bp[i,j,i]}c_j^{\bp[i,j,i]}, \ b_{ki}^{\bp[i,j,i,j]}c_i^{\bp[i,j,i,j]}) = (0,+,0,+,+).
	$$
	
	Hence
	$$r_k^{\bp[i,j,i,j,i]} = r_i^{\bp[i,j,i,j]}r_k^{\bp[i,j,i,j]}r_i^{\bp[i,j,i,j]} = r_i^{\bp[i,j,i,j]}r_j^{\bp[i,j,i]}r_k^{\bp[i,j,i]}r_j^{\bp[i,j,i]}r_i^{\bp[i,j,i,j]}$$
	$$= r_i^{\bp[i,j,i,j]}r_j^{\bp[i,j,i]}r_k^{\bp[i,j]}r_j^{\bp[i,j,i]}r_i^{\bp[i,j,i,j]} = r_i^{\bp[i,j,i,j]}r_j^{\bp[i,j,i]}r_j^{\bp[i]}r_k^{\bp[i]}r_j^{\bp[i]}r_j^{\bp[i,j,i]}r_i^{\bp[i,j,i,j]}$$
	$$= r_i^{\bp[i,j,i,j]}r_j^{\bp[i,j,i]}r_i^{\bp}r_j^{\bp}r_i^{\bp}r_k^{\bp}r_i^{\bp}r_j^{\bp}r_i^{\bp}r_j^{\bp[i,j,i]}r_i^{\bp[i,j,i,j]}.$$
	
	By Lemma \ref{swap-indices-i,j}, we have that  
	$$\pi(r_i^{\bp[i,j,i,j]}) = \pi(r_i^{\bp[i,j,i,j,i]}) = \pi(r_i^\bp r_j^\bp)^3\pi(r_j^\bp) = \pi(r_j^\bp)\pi(r_j^\bp r_i^\bp)^3.$$
	Additionally, 
	$$r_j^{\bp[i,j,i]} = r_i^{\bp[i,j]}r_j^{\bp[i]}r_i^{\bp[i,j]} = r_j^{\bp[i]}r_i^{\bp[i]}r_j^{\bp[i]}r_i^{\bp[i]}r_j^{\bp[i]} = r_i^{\bp}r_j^{\bp}r_i^{\bp}r_j^{\bp}r_i^{\bp}r_i^{\bp}r_j^{\bp}r_i^{\bp}$$
	$$= (r_i^{\bp}r_j^{\bp})^3r_i^{\bp} =  r_i^{\bp}(r_j^{\bp}r_i^{\bp})^3.$$
	
	Thus
	$$\pi(r_k^{\bp[i,j,i,j,i]}) = \pi(r_j^{\bp}) \pi(r_j^\bp r_i^\bp)^3\pi(r_i^{\bp}r_j^{\bp})^3 \pi(r_i^{\bp}r_i^{\bp}r_j^{\bp}r_i^{\bp}r_k^{\bp}r_i^{\bp}r_j^{\bp}r_i^{\bp} r_i^{\bp}) \pi(r_j^{\bp}r_i^{\bp})^3\pi(r_i^{\bp}r_j^{\bp})^3\pi(r_j^{\bp})$$
	$$ = \pi(r_i^{\bp}r_k^{\bp})^2\pi(r_k^{\bp}).$$
	
	\noindent \underline{Case $(b_{ij}^{\bp}c_j^{\bp},b_{ji}^{\bp}c_i^{\bp})=(-,-)$} The values we are concerned with are as follows:
	$$
	(b_{ki}^{\bp}c_i^{\bp}, \ b_{kj}^{\bp[i]}c_j^{\bp[i]}, \ b_{ki}^{\bp[i,j]}c_i^{\bp[i,j]}, \ b_{kj}^{\bp[i,j,i]}c_j^{\bp[i,j,i]}, \ b_{ki}^{\bp[i,j,i,j]}c_i^{\bp[i,j,i,j]}) = (0,+,0,-,+).
	$$
	
	Hence
	$$r_k^{\bp[i,j,i,j,i]} = r_i^{\bp[i,j,i,j]}r_k^{\bp[i,j,i,j]}r_i^{\bp[i,j,i,j]} = r_i^{\bp[i,j,i,j]}r_k^{\bp[i,j]}r_i^{\bp[i,j,i,j]} = r_i^{\bp[i,j,i,j]}r_j^{\bp[i]}r_k^{\bp}r_j^{\bp[i]}r_i^{\bp[i,j,i,j]}$$
	$$ = r_i^{\bp[i,j,i,j]}r_j^{\bp}r_k^{\bp}r_j^{\bp}r_i^{\bp[i,j,i,j]}.$$
	
	By Lemma \ref{swap-indices-i,j}, we have that  
	$$\pi(r_i^{\bp[i,j,i,j]}) = \pi(r_i^{\bp[i,j,i,j,i]}) = \pi(r_i^\bp r_j^\bp)^3\pi(r_j^\bp) = \pi(r_j^\bp)\pi(r_j^\bp r_i^\bp)^3.$$
	
	Thus
	$$\pi(r_k^{\bp[i,j,i,j,i]}) = \pi(r_i^\bp r_j^\bp)^3\pi(r_k^\bp)\pi(r_j^\bp r_i^\bp)^3.$$
	
	\noindent \underline{Case $(b_{ij}^{\bp}c_j^{\bp},b_{ji}^{\bp}c_i^{\bp},b_{ij}^{\bp[i]}c_j^{\bp[i]})=(+,+,-)$} The values we are concerned with are as follows:
	$$
	(b_{ki}^{\bp}c_i^{\bp}, \ b_{kj}^{\bp[i]}c_j^{\bp[i]}, \ b_{ki}^{\bp[i,j]}c_i^{\bp[i,j]}, \ b_{kj}^{\bp[i,j,i]}c_j^{\bp[i,j,i]}, \ b_{ki}^{\bp[i,j,i,j]}c_i^{\bp[i,j,i,j]}) = (0,-,0,-,-).
	$$
	
	Hence
	$r_k^{\bp[i,j,i,j,i]} = r_k^{\bp},$
	and  
	$\pi(r_k^{\bp[i,j,i,j,i]} ) = \pi(r_k^{\bp}).$
	
	\noindent \underline{Case $(b_{ij}^{\bp}c_j^{\bp},b_{ji}^{\bp}c_i^{\bp},b_{ij}^{\bp[i,j,i,j]}c_j^{\bp[i,j,i,j]})=(+,+,-)$} The values we are concerned with are as follows:
	$$
	(b_{ki}^{\bp}c_i^{\bp}, \ b_{kj}^{\bp[i]}c_j^{\bp[i]}, \ b_{ki}^{\bp[i,j]}c_i^{\bp[i,j]}, \ b_{kj}^{\bp[i,j,i]}c_j^{\bp[i,j,i]}, \ b_{ki}^{\bp[i,j,i,j]}c_i^{\bp[i,j,i,j]}) = (0,+,0,+,-).
	$$
	
	Hence
	$$r_k^{\bp[i,j,i,j,i]} = r_k^{\bp[i,j,i,j]} = r_j^{\bp[i,j,i]}r_k^{\bp[i,j,i]}r_j^{\bp[i,j,i]} = r_j^{\bp[i,j,i]}r_k^{\bp[i,j]}r_j^{\bp[i,j,i]}$$
	$$= r_j^{\bp[i,j,i]}r_j^{\bp[i]}r_k^{\bp[i]}r_j^{\bp[i]}r_j^{\bp[i,j,i]} = r_j^{\bp[i,j,i]}r_i^{\bp}r_j^{\bp}r_i^{\bp}r_k^{\bp}r_i^{\bp}r_j^{\bp}r_i^{\bp}r_j^{\bp[i,j,i]}.$$
	
	As
	$r_j^{\bp[i,j,i]} = r_j^{\bp[i]} = r_i^{\bp}r_j^{\bp}r_i^{\bp},$
	we have
	$\pi(r_k^{\bp[i,j,i,j,i]} ) = \pi(r_k^{\bp}).$
\end{proof}

\begin{proof}[\boxed{Proof\ of\ Lemma\ \ref{lem-case-10}}]
	We need to keep track of $b_{ki}^\bp, b_{kj}^\bp, b_{ij}^\bp, b_{ji}^\bp, c_i^\bp,$ and $c_j^\bp$ throughout the entire mutation $[i,j,i,j,i]$.
	Let us first list all the possible values of the first four variables.
	{\tiny $$
		(...,b_{ki}^{\bp[i,j,i,j]})=(-,+,0,0,0);\, (...,b_{kj}^{\bp[i,j,i,j]})=(0,-,+,+,-);\, (...,b_{ij}^{\bp[i,j,i,j]})=(-,+,-,+,-);\, (...,b_{ji}^{\bp[i,j,i,j]})=(+,-,+,-,+)
		$$}
	
	\noindent \underline{Case $(b_{ij}^{\bp}c_j^{\bp},b_{ji}^{\bp}c_i^{\bp})=(+,-)$} The values we are concerned with are as follows:
	$$
	(b_{ki}^{\bp}c_i^{\bp}, \ b_{kj}^{\bp[i]}c_j^{\bp[i]}, \ b_{ki}^{\bp[i,j]}c_i^{\bp[i,j]}, \ b_{kj}^{\bp[i,j,i]}c_j^{\bp[i,j,i]}, \ b_{ki}^{\bp[i,j,i,j]}c_i^{\bp[i,j,i,j]}) = (+,+,0,+,0).
	$$
	
	Hence
	$$r_k^{\bp[i,j,i,j,i]} = r_k^{\bp[i,j,i,j]} = r_j^{\bp[i,j,i]}r_k^{\bp[i,j,i]}r_j^{\bp[i,j,i]} = r_j^{\bp[i,j,i]}r_k^{\bp[i,j]}r_j^{\bp[i,j,i]}$$
	$$ = r_j^{\bp[i,j,i]}r_j^{\bp[i]}r_k^{\bp[i]}r_j^{\bp[i]}r_j^{\bp[i,j,i]} = r_j^{\bp[i,j,i]}r_j^{\bp}r_i^{\bp}r_k^{\bp}r_i^{\bp}r_j^{\bp}r_j^{\bp[i,j,i]}.$$
	
	Additionally,
	$$r_j^{\bp[i,j,i]} = r_i^{\bp[i,j]}r_j^{\bp[i,j]}r_i^{\bp[i,j]} = r_i^{\bp}r_j^{\bp}r_i^{\bp}.$$
	
	Thus
	$$\pi(r_k^{\bp[i,j,i,j,i]}) = \pi(r_i^{\bp}r_j^{\bp}r_i^{\bp}r_j^{\bp}r_i^{\bp}r_k^{\bp}r_i^{\bp}r_j^{\bp}r_i^{\bp}r_j^{\bp}r_i^{\bp}) = \pi(r_i^{\bp}r_j^{\bp})^3 \pi(r_j^{\bp}r_k^{\bp}r_j^{\bp})\pi(r_j^{\bp}r_i^{\bp})^3$$
	$$= \pi(r_i^{\bp}r_j^{\bp})^3 \pi(r_j^{\bp}r_k^{\bp})^2\pi(r_k^{\bp})\pi(r_j^{\bp}r_i^{\bp})^3.$$
	
	\noindent \underline{Case $(b_{ij}^{\bp}c_j^{\bp},b_{ji}^{\bp}c_i^{\bp})=(-,+)$} The values we are concerned with are as follows:
	$$
	(b_{ki}^{\bp}c_i^{\bp}, \ b_{kj}^{\bp[i]}c_j^{\bp[i]}, \ b_{ki}^{\bp[i,j]}c_i^{\bp[i,j]}, \ b_{kj}^{\bp[i,j,i]}c_j^{\bp[i,j,i]}, \ b_{ki}^{\bp[i,j,i,j]}c_i^{\bp[i,j,i,j]}) = (-,-,0,-,0).
	$$
	
	Hence
	$r_k^{\bp[i,j,i,j,i]} = r_k^{\bp},$
	and  
	$\pi(r_k^{\bp[i,j,i,j,i]} ) = \pi(r_k^{\bp}).$
	
	\noindent \underline{Case $(b_{ij}^{\bp}c_j^{\bp},b_{ji}^{\bp}c_i^{\bp})=(-,-)$} The values we are concerned with are as follows:
	$$
	(b_{ki}^{\bp}c_i^{\bp}, \ b_{kj}^{\bp[i]}c_j^{\bp[i]}, \ b_{ki}^{\bp[i,j]}c_i^{\bp[i,j]}, \ b_{kj}^{\bp[i,j,i]}c_j^{\bp[i,j,i]}, \ b_{ki}^{\bp[i,j,i,j]}c_i^{\bp[i,j,i,j]}) = (+,-,0,+,0).
	$$
	
	Hence
	$$r_k^{\bp[i,j,i,j,i]} = r_k^{\bp[i,j,i,j]} = r_j^{\bp[i,j,i]}r_k^{\bp[i,j,i]}r_j^{\bp[i,j,i]} = r_j^{\bp[i,j,i]}r_k^{\bp[i]}r_j^{\bp[i,j,i]} = r_j^{\bp[i,j,i]}r_i^{\bp}r_k^{\bp}r_i^{\bp}r_j^{\bp[i,j,i]}.$$
	
	Additionally,
	$$r_j^{\bp[i,j,i]} = r_i^{\bp[i,j]}r_j^{\bp[i,j]}r_i^{\bp[i,j]} = r_j^{\bp}r_i^{\bp}r_j^{\bp}r_i^{\bp}r_j^{\bp}.$$
	
	Thus
	$\pi(r_k^{\bp[i,j,i,j,i]}) = \pi(r_j^{\bp}r_i^{\bp})^3 \pi(r_k^{\bp}) \pi(r_i^{\bp}r_j^{\bp})^3.$
	
	\noindent \underline{Case $(b_{ij}^{\bp}c_j^{\bp},b_{ji}^{\bp}c_i^{\bp},b_{ij}^{\bp[i]}c_j^{\bp[i]})=(+,+,-)$} The values we are concerned with are as follows:
	$$
	(b_{ki}^{\bp}c_i^{\bp}, \ b_{kj}^{\bp[i]}c_j^{\bp[i]}, \ b_{ki}^{\bp[i,j]}c_i^{\bp[i,j]}, \ b_{kj}^{\bp[i,j,i]}c_j^{\bp[i,j,i]}, \ b_{ki}^{\bp[i,j,i,j]}c_i^{\bp[i,j,i,j]}) = (-,+,0,+,0).
	$$
	
	Hence
	$$r_k^{\bp[i,j,i,j,i]} = r_k^{\bp[i,j,i,j]} = r_j^{\bp[i,j,i]}r_k^{\bp[i,j,i]}r_j^{\bp[i,j,i]} = r_j^{\bp[i,j,i]}r_k^{\bp[i,j]}r_j^{\bp[i,j,i]} = r_j^{\bp[i,j,i]}r_j^{\bp[i]}r_k^{\bp}r_j^{\bp[i]}r_j^{\bp[i,j,i]}$$
	$$ = r_j^{\bp[i,j,i]}r_i^{\bp}r_j^{\bp}r_i^{\bp}r_k^{\bp}r_i^{\bp}r_j^{\bp}r_i^{\bp}r_j^{\bp[i,j,i]}.$$
	
	Additionally, 
	$$r_j^{\bp[i,j,i]} = r_j^{\bp[i]} = r_i^{\bp}r_j^{\bp}r_i^{\bp}.$$
	Thus
	$\pi(r_k^{\bp[i,j,i,j,i]}) = \pi(r_k^{\bp}).$
	
	\noindent \underline{Case $(b_{ij}^{\bp}c_j^{\bp},b_{ji}^{\bp}c_i^{\bp},b_{ij}^{\bp[i,j,i,j]}c_j^{\bp[i,j,i,j]})=(+,+,-)$} The values we are concerned with are as follows:
	$$
	(b_{ki}^{\bp}c_i^{\bp}, \ b_{kj}^{\bp[i]}c_j^{\bp[i]}, \ b_{ki}^{\bp[i,j]}c_i^{\bp[i,j]}, \ b_{kj}^{\bp[i,j,i]}c_j^{\bp[i,j,i]}, \ b_{ki}^{\bp[i,j,i,j]}c_i^{\bp[i,j,i,j]}) = (-,-,0,-,0).
	$$
	
	Hence
	$r_k^{\bp[i,j,i,j,i]} = r_k^{\bp},$
	and  
	$\pi(r_k^{\bp[i,j,i,j,i]} ) = \pi(r_k^{\bp}).$
\end{proof}

\begin{proof}[\boxed{Proof\ of\ Lemma\ \ref{lem-case-3}}]
	
	Let's break down what variables we have to keep track.
	We need to keep track of $b_{ki}^\bp, b_{kj}^\bp, b_{ij}^\bp, b_{ji}^\bp, c_i^\bp,$ and $c_j^\bp$ throughout the entire mutation $[i,j,i,j,i]$.
	First list all the possible values of the first four variables.
	{\tiny $$
	(...,b_{ki}^{\bp[i,j,i,j]})=(0,0,+,-,-);\, (...,b_{kj}^{\bp[i,j,i,j]})=(+,+,-,0,0);\, (...,b_{ij}^{\bp[i,j,i,j]})=(+,-,+,-,+);\, (...,b_{ji}^{\bp[i,j,i,j]})=(-,+,-,+,-)
	$$}
	{\small Note that $b_{kj}^{\bp[i,j,i]} = 0$ follows from our restriction to the $A_n$ case.}
	
	\noindent \underline{Case $(b_{ij}^{\bp}c_j^{\bp},b_{ji}^{\bp}c_i^{\bp})=(+,-)$}
	The values we are concerned with are as follows:
	$$
	(b_{ki}^{\bp}c_i^{\bp}, \ b_{kj}^{\bp[i]}c_j^{\bp[i]}, \ b_{ki}^{\bp[i,j]}c_i^{\bp[i,j]}, \ b_{kj}^{\bp[i,j,i]}c_j^{\bp[i,j,i]}, \ b_{ki}^{\bp[i,j,i,j]}c_i^{\bp[i,j,i,j]}) = (0,+,-,0,+).
	$$
	
	
	Hence
	$$r_k^{\bp[i,j,i,j,i]} = r_i^{\bp[i,j,i,j]}r_k^{\bp[i,j,i,j]}r_i^{\bp[i,j,i,j]} = r_i^{\bp[i,j,i,j]}r_k^{\bp[i,j,i]}r_i^{\bp[i,j,i,j]} = r_i^{\bp[i,j,i,j]}r_k^{\bp[i,j]}r_i^{\bp[i,j,i,j]}$$
	$$ = r_i^{\bp[i,j,i,j]}r_j^{\bp[i]}r_k^{\bp[i]}r_j^{\bp[i]}r_i^{\bp[i,j,i,j]}  = r_i^{\bp[i,j,i,j]}r_j^{\bp}r_k^{\bp}r_j^{\bp}r_i^{\bp[i,j,i,j]}.$$
	
	By Lemma \ref{swap-indices-i,j}, we have that  
	$$\pi(r_i^{\bp[i,j,i,j]}) = \pi(r_i^{\bp[i,j,i,j,i]}) = \pi(r_i^\bp r_j^\bp)^3\pi(r_j^\bp) = \pi(r_j^\bp)\pi(r_j^\bp r_i^\bp)^3.$$
	Thus  
	$$\pi(r_k^{\bp[i,j,i,j,i]} ) = \pi(r_i^{\bp[i,j,i,j]}r_j^{\bp}r_k^{\bp}r_j^{\bp}r_i^{\bp[i,j,i,j]}) = \pi(r_i^\bp r_j^\bp)^3\pi(r_j^\bp)\pi(r_j^{\bp}r_k^{\bp}r_j^{\bp}) \pi(r_j^\bp)\pi(r_j^\bp r_i^\bp)^3$$
	$$ =  \pi(r_i^\bp r_j^\bp)^3\pi(r_k^{\bp}) \pi(r_j^\bp r_i^\bp)^3.$$ 
	
	\noindent \underline{Case $(b_{ij}^{\bp}c_j^{\bp},b_{ji}^{\bp}c_i^{\bp})=(-,+)$} The values we are concerned with are as follows:
	$$
	(b_{ki}^{\bp}c_i^{\bp}, \ b_{kj}^{\bp[i]}c_j^{\bp[i]}, \ b_{ki}^{\bp[i,j]}c_i^{\bp[i,j]}, \ b_{kj}^{\bp[i,j,i]}c_j^{\bp[i,j,i]}, \ b_{ki}^{\bp[i,j,i,j]}c_i^{\bp[i,j,i,j]}) = (0,-,-,0,-).
	$$
	
	
	Hence
	$r_k^{\bp[i,j,i,j,i]} = r_k^{\bp}$ and $\pi(r_k^{\bp[i,j,i,j,i]} ) = \pi(r_k^{\bp}).$
	
	\noindent \underline{Case $(b_{ij}^{\bp}c_j^{\bp},b_{ji}^{\bp}c_i^{\bp})=(-,-)$} The values we are concerned with are as follows:
	$$
	(b_{ki}^{\bp}c_i^{\bp}, \ b_{kj}^{\bp[i]}c_j^{\bp[i]}, \ b_{ki}^{\bp[i,j]}c_i^{\bp[i,j]}, \ b_{kj}^{\bp[i,j,i]}c_j^{\bp[i,j,i]}, \ b_{ki}^{\bp[i,j,i,j]}c_i^{\bp[i,j,i,j]}) = (0,-,-,0,-).
	$$
	
	
	As in the previous scenario,
	$r_k^{\bp[i,j,i,j,i]} = r_k^{\bp},$
	and  
	$\pi(r_k^{\bp[i,j,i,j,i]} ) = \pi(r_k^{\bp}).$
	
	\noindent \underline{Case $(b_{ij}^{\bp}c_j^{\bp},b_{ji}^{\bp}c_i^{\bp},b_{ij}^{\bp[i]}c_j^{\bp[i]})=(+,+,-)$} The values we are concerned with are as follows:
	$$
	(b_{ki}^{\bp}c_i^{\bp}, \ b_{kj}^{\bp[i]}c_j^{\bp[i]}, \ b_{ki}^{\bp[i,j]}c_i^{\bp[i,j]}, \ b_{kj}^{\bp[i,j,i]}c_j^{\bp[i,j,i]}, \ b_{ki}^{\bp[i,j,i,j]}c_i^{\bp[i,j,i,j]}) = (0,+,+,0,+).
	$$
	
	
	Hence
	$$r_k^{\bp[i,j,i,j,i]} = r_i^{\bp[i,j,i,j]}r_k^{\bp[i,j,i,j]}r_i^{\bp[i,j,i,j]} = r_i^{\bp[i,j,i,j]}r_k^{\bp[i,j,i]}r_i^{\bp[i,j,i,j]} = r_i^{\bp[i,j,i,j]}r_i^{\bp[i,j]}r_k^{\bp[i,j]}r_i^{\bp[i,j]}r_i^{\bp[i,j,i,j]}$$
	$$ = r_i^{\bp[i,j,i,j]}r_i^{\bp[i,j]}r_j^{\bp[i]}r_k^{\bp[i]}r_j^{\bp[i]}r_i^{\bp[i,j]}r_i^{\bp[i,j,i,j]} = r_i^{\bp[i,j,i,j]}r_i^{\bp[i,j]}r_j^{\bp[i]}r_k^{\bp}r_j^{\bp[i]}r_i^{\bp[i,j]}r_i^{\bp[i,j,i,j]}.$$
	
	By Lemma \ref{swap-indices-i,j}, we have that  
	$$\pi(r_i^{\bp[i,j,i,j]}) = \pi(r_i^{\bp[i,j,i,j,i]}) = \pi(r_i^\bp r_j^\bp)^3\pi(r_j^\bp) = \pi(r_j^\bp)\pi(r_j^\bp r_i^\bp)^3.$$
	Additionally, from the proof of Lemma \ref{swap-indices-i,j}, we have that 
	$$r_i^{\bp[i,j]} = r_i^{\bp}$$
	and 
	$$r_j^{\bp[i]} = r_i^{\bp}r_j^{\bp}r_i^{\bp}.$$
	
	Thus
	$$\pi(r_k^{\bp[i,j,i,j,i]} ) = \pi(r_i^\bp r_j^\bp)^3\pi(r_j^\bp)\pi(r_j^{\bp}r_i^{\bp}r_k^{\bp}r_i^{\bp}r_j^{\bp}) \pi(r_j^\bp)\pi(r_j^\bp r_i^\bp)^3$$
	$$=\pi(r_i^\bp r_j^\bp)^3\pi(r_i^\bp)\pi(r_k^{\bp}) \pi(r_i^\bp)\pi(r_j^\bp r_i^\bp)^3$$
	$$=\pi(r_i^\bp r_j^\bp)^3\pi(r_i^\bp r_k^{\bp})^2 \pi(r_k^\bp)\pi(r_j^\bp r_i^\bp)^3.$$
	
	\noindent \underline{Case $(b_{ij}^{\bp}c_j^{\bp},b_{ji}^{\bp}c_i^{\bp},b_{ij}^{\bp[i,j,i,j]}c_j^{\bp[i,j,i,j]})=(+,+,-)$} The values we are concerned with are as follows:
	$$
	(b_{ki}^{\bp}c_i^{\bp}, \ b_{kj}^{\bp[i]}c_j^{\bp[i]}, \ b_{ki}^{\bp[i,j]}c_i^{\bp[i,j]}, \ b_{kj}^{\bp[i,j,i]}c_j^{\bp[i,j,i]}, \ b_{ki}^{\bp[i,j,i,j]}c_i^{\bp[i,j,i,j]}) = (0,-,+,0,+).
	$$
	
	
	Hence
	$$r_k^{\bp[i,j,i,j,i]} = r_i^{\bp[i,j,i,j]}r_k^{\bp[i,j,i,j]}r_i^{\bp[i,j,i,j]} = r_i^{\bp[i,j,i,j]}r_k^{\bp[i,j,i]}r_i^{\bp[i,j,i,j]} = r_i^{\bp[i,j,i,j]}r_i^{\bp[i,j]}r_k^{\bp[i,j]}r_i^{\bp[i,j]}r_i^{\bp[i,j,i,j]}$$
	$$ = r_i^{\bp[i,j,i,j]}r_i^{\bp[i,j]}r_k^{\bp[i]}r_i^{\bp[i,j]}r_i^{\bp[i,j,i,j]} = r_i^{\bp[i,j,i,j]}r_i^{\bp[i,j]}r_k^{\bp}r_i^{\bp[i,j]}r_i^{\bp[i,j,i,j]}.$$
	
	By Lemma \ref{swap-indices-i,j}, we have that  
	$$\pi(r_i^{\bp[i,j,i,j]}) = \pi(r_i^{\bp[i,j,i,j,i]}) = \pi(r_i^\bp r_j^\bp)^3\pi(r_j^\bp) = \pi(r_j^\bp)\pi(r_j^\bp r_i^\bp)^3.$$
	Additionally, from the proof of Lemma \ref{swap-indices-i,j}, we have that 
	$$r_i^{\bp[i,j]} = r_j^{\bp[i]}r_i^{\bp[i]}r_j^{\bp[i]} = r_i^{\bp}r_j^{\bp}r_i^{\bp}r_j^{\bp}r_i^{\bp} = (r_i^{\bp}r_j^{\bp})^3r_j^{\bp} = r_j^{\bp}(r_j^{\bp}r_i^{\bp})^3.$$
	
	Thus
	$$\pi(r_k^{\bp[i,j,i,j,i]}) = \pi(r_i^\bp r_j^\bp)^3\pi(r_j^\bp) \pi(r_j^{\bp}) \pi(r_j^{\bp}r_i^{\bp})^3 \pi(r_k^{\bp}) \pi(r_i^{\bp}r_j^{\bp})^3 \pi(r_j^{\bp}) \pi(r_j^\bp)\pi(r_j^\bp r_i^\bp)^3$$
	$$= \pi(r_k^{\bp}).$$
\end{proof}

\begin{proof}[\boxed{Proof\ of\ Lemma\ \ref{lem-case-8}}]
	We need to keep track of $b_{ki}^\bp, b_{kj}^\bp, b_{ij}^\bp, b_{ji}^\bp, c_i^\bp,$ and $c_j^\bp$ throughout the entire mutation $[i,j,i,j,i]$.
	Let us first list all the possible values of the first four variables.
	{\tiny $$
		(...,b_{ki}^{\bp[i,j,i,j]})=(+,-,-,+,0);\, (...,b_{kj}^{\bp[i,j,i,j]})=(0,0,0,-,+);\, (...,b_{ij}^{\bp[i,j,i,j]})=(-,+,-,+,-);\, (...,b_{ji}^{\bp[i,j,i,j]})=(+,-,+,-,+)
		$$}
	Note that $ b_{ki}^{\bp[i,j,i,j]} = 0$ follows from our restriction to the $A_n$ case.
	
	\noindent \underline{Case $(b_{ij}^{\bp}c_j^{\bp},b_{ji}^{\bp}c_i^{\bp})=(+,-)$} The values we are concerned with are as follows:
	$$
	(b_{ki}^{\bp}c_i^{\bp}, \ b_{kj}^{\bp[i]}c_j^{\bp[i]}, \ b_{ki}^{\bp[i,j]}c_i^{\bp[i,j]}, \ b_{kj}^{\bp[i,j,i]}c_j^{\bp[i,j,i]}, \ b_{ki}^{\bp[i,j,i,j]}c_i^{\bp[i,j,i,j]}) = (-,0,-,-,0).
	$$
	
	Hence 
	$r_k^{\bp[i,j,i,j,i]} = r_k^{\bp},$
	and 
	$\pi(r_k^{\bp[i,j,i,j,i]}) = \pi(r_k^{\bp}).$ 
	
	\noindent \underline{Case $(b_{ij}^{\bp}c_j^{\bp},b_{ji}^{\bp}c_i^{\bp})=(-,+)$} The values we are concerned with are as follows:
	$$
	(b_{ki}^{\bp}c_i^{\bp}, \ b_{kj}^{\bp[i]}c_j^{\bp[i]}, \ b_{ki}^{\bp[i,j]}c_i^{\bp[i,j]}, \ b_{kj}^{\bp[i,j,i]}c_j^{\bp[i,j,i]}, \ b_{ki}^{\bp[i,j,i,j]}c_i^{\bp[i,j,i,j]}) = (+,0,-,+,0).
	$$
	
	Hence
	$$r_k^{\bp[i,j,i,j,i]} = r_k^{\bp[i,j,i,j]} = r_j^{\bp[i,j,i]}r_k^{\bp[i,j,i]}r_j^{\bp[i,j,i]} = r_j^{\bp[i,j,i]}r_k^{\bp[i]}r_j^{\bp[i,j,i]}$$
	$$ = r_j^{\bp[i,j,i]}r_i^{\bp}r_k^{\bp}r_i^{\bp}r_j^{\bp[i,j,i]}.$$
	Additionally,
	$$r_j^{\bp[i,j,i]} = r_i^{\bp[i,j]}r_j^{\bp[i,j]}r_i^{\bp[i,j]} = r_j^{\bp[i]}r_i^{\bp[i]}r_j^{\bp[i]}r_i^{\bp[i]}r_j^{\bp[i]} = r_i^{\bp}r_j^{\bp}r_i^{\bp}r_j^{\bp}r_i^{\bp}r_j^{\bp}r_i^{\bp}$$
	$$ = (r_i^{\bp}r_j^{\bp})^3 r_i^{\bp} =  r_i^{\bp}(r_j^{\bp}r_i^{\bp})^3.$$
	
	Thus
	$\pi(r_k^{\bp[i,j,i,j,i]}) = \pi(r_i^{\bp}r_j^{\bp})^3\pi(r_k^{\bp}) \pi(r_j^{\bp}r_i^{\bp})^3.$
	
	\noindent \underline{Case $(b_{ij}^{\bp}c_j^{\bp},b_{ji}^{\bp}c_i^{\bp})=(-,-)$} The values we are concerned with are as follows:
	$$
	(b_{ki}^{\bp}c_i^{\bp}, \ b_{kj}^{\bp[i]}c_j^{\bp[i]}, \ b_{ki}^{\bp[i,j]}c_i^{\bp[i,j]}, \ b_{kj}^{\bp[i,j,i]}c_j^{\bp[i,j,i]}, \ b_{ki}^{\bp[i,j,i,j]}c_i^{\bp[i,j,i,j]}) = (-,0,-,-,0).
	$$
	
	Hence 
	$r_k^{\bp[i,j,i,j,i]} = r_k^{\bp},$
	and 
	$\pi(r_k^{\bp[i,j,i,j,i]}) = \pi(r_k^{\bp}).$
	
	\noindent \underline{Case $(b_{ij}^{\bp}c_j^{\bp},b_{ji}^{\bp}c_i^{\bp},b_{ij}^{\bp[i]}c_j^{\bp[i]})=(+,+,-)$} The values we are concerned with are as follows:
	$$
	(b_{ki}^{\bp}c_i^{\bp}, \ b_{kj}^{\bp[i]}c_j^{\bp[i]}, \ b_{ki}^{\bp[i,j]}c_i^{\bp[i,j]}, \ b_{kj}^{\bp[i,j,i]}c_j^{\bp[i,j,i]}, \ b_{ki}^{\bp[i,j,i,j]}c_i^{\bp[i,j,i,j]}) = (+,0,+,-,0).
	$$
	
	Hence
	$$r_k^{\bp[i,j,i,j,i]} = r_k^{\bp[i,j,i,j]}  = r_k^{\bp[i,j,i]} = r_i^{\bp[i,j]}r_k^{\bp[i]}r_i^{\bp[i,j]}$$
	$$ = r_i^{\bp[i,j]}r_i^{\bp}r_k^{\bp}r_i^{\bp}r_i^{\bp[i,j]} = r_k^{\bp}.$$
	
	Thus
	$\pi(r_k^{\bp[i,j,i,j,i]}) = \pi(r_k^{\bp}).$
	
	\noindent \underline{Case $(b_{ij}^{\bp}c_j^{\bp},b_{ji}^{\bp}c_i^{\bp},b_{ij}^{\bp[i,j,i,j]}c_j^{\bp[i,j,i,j]})=(+,+,-)$} The values we are concerned with are as follows:
	$$
	(b_{ki}^{\bp}c_i^{\bp}, \ b_{kj}^{\bp[i]}c_j^{\bp[i]}, \ b_{ki}^{\bp[i,j]}c_i^{\bp[i,j]}, \ b_{kj}^{\bp[i,j,i]}c_j^{\bp[i,j,i]}, \ b_{ki}^{\bp[i,j,i,j]}c_i^{\bp[i,j,i,j]}) = (+,0,+,+,0).
	$$
	
	Hence
	$$r_k^{\bp[i,j,i,j,i]} = r_k^{\bp[i,j,i,j]}  = r_j^{\bp[i,j,i]}r_k^{\bp[i,j,i]}r_j^{\bp[i,j,i]} = r_j^{\bp[i,j,i]}r_i^{\bp[i,j]}r_k^{\bp[i,j]}r_i^{\bp[i,j]}r_j^{\bp[i,j,i]}$$
	$$ = r_j^{\bp[i,j,i]}r_i^{\bp[i,j]}r_k^{\bp[i]}r_i^{\bp[i,j]}r_j^{\bp[i,j,i]} = r_j^{\bp[i,j,i]}r_i^{\bp[i,j]}r_i^{\bp}r_k^{\bp} r_i^{\bp}r_i^{\bp[i,j]}r_j^{\bp[i,j,i]}.$$
	
	Additionally, 
	$r_j^{\bp[i,j,i]} = r_j^{\bp[i]} = r_i^{\bp}r_j^{\bp}r_i^{\bp}$
	and
	$r_i^{\bp[i,j]} = r_j^{\bp[i]}r_i^{\bp[i]}r_j^{\bp[i]} = r_i^{\bp}r_j^{\bp}r_i^{\bp}r_j^{\bp}r_i^{\bp}.$
	
	Thus
	$$ r_j^{\bp[i,j,i]}r_i^{\bp[i,j]}r_i^{\bp} = r_j^{\bp} =  r_i^{\bp}r_i^{\bp[i,j]}r_j^{\bp[i,j,i]},$$
	forcing
	$$\pi(r_k^{\bp[i,j,i,j,i]}) = \pi(r_j^{\bp}r_k^{\bp}r_j^{\bp}) = \pi(r_j^{\bp}r_k^{\bp})^2\pi(r_k^{\bp}).$$
\end{proof}

\begin{proof}[\boxed{Proof\ of\ Lemma\ \ref{lem-case-6}}]
	We need to keep track of $b_{ki}^\bp, b_{kj}^\bp, b_{ij}^\bp, b_{ji}^\bp, c_i^\bp,$ and $c_j^\bp$ throughout the entire mutation $[i,j,i,j,i]$.
	Let us first list all the possible values of the first four variables.
	{\tiny $$
		(...,b_{ki}^{\bp[i,j,i,j]})=(0,0,-,+,+);\, (...,b_{kj}^{\bp[i,j,i,j]})=(-,-,+,0,0);\, (...,b_{ij}^{\bp[i,j,i,j]})=(-,+,-,+,-);\, (...,b_{ji}^{\bp[i,j,i,j]})=(+,-,+,-,+).
		$$}
	
	
	\noindent \underline{Case $(b_{ij}^{\bp}c_j^{\bp},b_{ji}^{\bp}c_i^{\bp})=(+,-)$} The values we are concerned with are as follows:
	$$
	(b_{ki}^{\bp}c_i^{\bp}, \ b_{kj}^{\bp[i]}c_j^{\bp[i]}, \ b_{ki}^{\bp[i,j]}c_i^{\bp[i,j]}, \ b_{kj}^{\bp[i,j,i]}c_j^{\bp[i,j,i]}, \ b_{ki}^{\bp[i,j,i,j]}c_i^{\bp[i,j,i,j]}) = (0,+,-,0,+).
	$$
	
	Hence 
	$$r_k^{\bp[i,j,i,j,i]} = r_i^{\bp[i,j,i,j]}r_k^{\bp[i,j,i,j]}r_i^{\bp[i,j,i,j]} = r_j^{\bp[i,j,i]}r_i^{\bp[i,j,i]}r_j^{\bp[i,j,i]}r_k^{\bp[i,j]}r_j^{\bp[i,j,i]}r_i^{\bp[i,j]}r_j^{\bp[i,j,i]}$$
	$$= r_i^{\bp[i,j]}r_j^{\bp[i]}r_i^{\bp[i,j]}r_j^{\bp[i]}r_i^{\bp[i,j]}r_k^{\bp[i,j]}r_i^{\bp[i,j]}r_j^{\bp[i]}r_i^{\bp[i,j]}r_j^{\bp[i]}r_i^{\bp[i,j]} = r_i^{\bp}r_j^{\bp}r_i^{\bp}r_j^{\bp}r_i^{\bp}r_j^{\bp}r_k^{\bp}r_j^{\bp}r_i^{\bp}r_j^{\bp}r_i^{\bp}r_j^{\bp}r_i^{\bp},$$
	and 
	$\pi(r_k^{\bp[i,j,i,j,i]}) = \pi(r_i^\bp r_j^\bp)^3\pi(r_k^{\bp})\pi(r_j^\bp r_i^\bp)^3.$ 
	
	\noindent \underline{Case $(b_{ij}^{\bp}c_j^{\bp},b_{ji}^{\bp}c_i^{\bp})=(-,+)$} The values we are concerned with are as follows:
	$$
	(b_{ki}^{\bp}c_i^{\bp}, \ b_{kj}^{\bp[i]}c_j^{\bp[i]}, \ b_{ki}^{\bp[i,j]}c_i^{\bp[i,j]}, \ b_{kj}^{\bp[i,j,i]}c_j^{\bp[i,j,i]}, \ b_{ki}^{\bp[i,j,i,j]}c_i^{\bp[i,j,i,j]}) = (0,-,-,0,-).
	$$
	
	Hence
	$r_k^{\bp[i,j,i,j,i]} = r_k^{\bp}.$
	Thus
	$\pi(r_k^{\bp[i,j,i,j,i]}) = \pi(r_k^{\bp}).$

	\noindent \underline{Case $(b_{ij}^{\bp}c_j^{\bp},b_{ji}^{\bp}c_i^{\bp})=(-,-)$} The values we are concerned with are as follows:
	$$
	(b_{ki}^{\bp}c_i^{\bp}, \ b_{kj}^{\bp[i]}c_j^{\bp[i]}, \ b_{ki}^{\bp[i,j]}c_i^{\bp[i,j]}, \ b_{kj}^{\bp[i,j,i]}c_j^{\bp[i,j,i]}, \ b_{ki}^{\bp[i,j,i,j]}c_i^{\bp[i,j,i,j]}) = (0,-,-,0,-).
	$$
	
	Hence 
	$r_k^{\bp[i,j,i,j,i]} = r_k^{\bp},$
	and 
	$\pi(r_k^{\bp[i,j,i,j,i]}) = \pi(r_k^\bp).$ 
	
	\noindent \underline{Case $(b_{ij}^{\bp}c_j^{\bp},b_{ji}^{\bp}c_i^{\bp},b_{ij}^{\bp[i]}c_j^{\bp[i]})=(+,+,-)$} The values we are concerned with are as follows:
	$$
	(b_{ki}^{\bp}c_i^{\bp}, \ b_{kj}^{\bp[i]}c_j^{\bp[i]}, \ b_{ki}^{\bp[i,j]}c_i^{\bp[i,j]}, \ b_{kj}^{\bp[i,j,i]}c_j^{\bp[i,j,i]}, \ b_{ki}^{\bp[i,j,i,j]}c_i^{\bp[i,j,i,j]}) = (0,+,+,0,+).
	$$
	
	Hence
	$$r_k^{\bp[i,j,i,j,i]} = r_i^{\bp[i,j,i,j]}r_k^{\bp[i,j,i]}r_i^{\bp[i,j,i,j]} = r_j^{\bp[i,j,i]}r_i^{\bp[i,j,i]}r_j^{\bp[i,j,i]}r_k^{\bp[i,j,i]}r_j^{\bp[i,j,i]}r_i^{\bp[i,j,i]}r_j^{\bp[i,j,i]}$$
	$$= r_j^{\bp[i]}r_i^{\bp[i,j]}r_j^{\bp[i]}r_i^{\bp[i,j]}r_k^{\bp[i,j]}r_i^{\bp[i,j]}r_j^{\bp[i]}r_i^{\bp[i,j]}r_j^{\bp[i]} = r_j^{\bp[i]}r_i^{\bp}r_j^{\bp[i]}r_i^{\bp}r_j^{\bp[i]}r_k^{\bp[i]}r_j^{\bp[i]}r_i^{\bp}r_j^{\bp[i]}r_i^{\bp}r_j^{\bp[i]}$$
	$$ = r_i^{\bp}r_j^{\bp}r_i^{\bp}r_j^{\bp}r_i^{\bp}r_j^{\bp}r_k^{\bp}r_j^{\bp}r_i^{\bp}r_j^{\bp}r_i^{\bp}r_j^{\bp}r_i^{\bp}.$$
	
	Thus
	$\pi(r_k^{\bp[i,j,i,j,i]}) = \pi(r_i^\bp r_j^\bp)^3\pi(r_k^{\bp})\pi(r_j^\bp r_i^\bp)^3.$

	\noindent \underline{Case $(b_{ij}^{\bp}c_j^{\bp},b_{ji}^{\bp}c_i^{\bp},b_{ij}^{\bp[i,j,i,j]}c_j^{\bp[i,j,i,j]})=(+,+,-)$} The values we are concerned with are as follows:
	$$
	(b_{ki}^{\bp}c_i^{\bp}, \ b_{kj}^{\bp[i]}c_j^{\bp[i]}, \ b_{ki}^{\bp[i,j]}c_i^{\bp[i,j]}, \ b_{kj}^{\bp[i,j,i]}c_j^{\bp[i,j,i]}, \ b_{ki}^{\bp[i,j,i,j]}c_i^{\bp[i,j,i,j]}) = (0,-,+,0,+).
	$$
	
	Hence
	$$r_k^{\bp[i,j,i,j,i]} = r_i^{\bp[i,j,i,j]}r_k^{\bp[i,j,i]}r_i^{\bp[i,j,i,j]} = r_i^{\bp[i,j]}r_k^{\bp[i,j,i]}r_i^{\bp[i,j]} = r_k^{\bp[i,j]} = r_k^\bp.$$
	
	Therefore
	$\pi(r_k^{\bp[i,j,i,j,i]}) = \pi(r_k^{\bp}).$
	
\end{proof}

\begin{proof}[\boxed{Proof\ of\ Lemma\ \ref{lem-case-9}}]
	We need to keep track of $b_{ki}^\bp, b_{kj}^\bp, b_{ij}^\bp, b_{ji}^\bp, c_i^\bp,$ and $c_j^\bp$ throughout the entire mutation $[i,j,i,j,i]$.
	Let us first list all the possible values of the first four variables.
	{\tiny $$
		(...,b_{ki}^{\bp[i,j,i,j]})=(-,+,+,-,0);\, (...,b_{kj}^{\bp[i,j,i,j]})=(0,0,0,+,-);\, (...,b_{ij}^{\bp[i,j,i,j]})=(+,-,+,-,+);\, (...,b_{ji}^{\bp[i,j,i,j]})=(-,+,-,+,-)
		$$}
	
	\noindent \underline{Case $(b_{ij}^{\bp}c_j^{\bp},b_{ji}^{\bp}c_i^{\bp})=(+,-)$} The values we are concerned with are as follows:
	$$
	(b_{ki}^{\bp}c_i^{\bp}, \ b_{kj}^{\bp[i]}c_j^{\bp[i]}, \ b_{ki}^{\bp[i,j]}c_i^{\bp[i,j]}, \ b_{kj}^{\bp[i,j,i]}c_j^{\bp[i,j,i]}, \ b_{ki}^{\bp[i,j,i,j]}c_i^{\bp[i,j,i,j]}) = (-,0,-,-,0).
	$$
	
	Hence 
	$r_k^{\bp[i,j,i,j,i]} = r_k^{\bp},$
	and 
	$\pi(r_k^{\bp[i,j,i,j,i]}) = \pi(r_k^{\bp}).$
	
	\noindent \underline{Case $(b_{ij}^{\bp}c_j^{\bp},b_{ji}^{\bp}c_i^{\bp})=(-,+)$} The values we are concerned with are as follows:
	$$
	(b_{ki}^{\bp}c_i^{\bp}, \ b_{kj}^{\bp[i]}c_j^{\bp[i]}, \ b_{ki}^{\bp[i,j]}c_i^{\bp[i,j]}, \ b_{kj}^{\bp[i,j,i]}c_j^{\bp[i,j,i]}, \ b_{ki}^{\bp[i,j,i,j]}c_i^{\bp[i,j,i,j]}) = (+,0,-,+,0).
	$$
	
	Hence
	$$r_k^{\bp[i,j,i,j,i]} = r_j^{\bp[i,j,i]}r_k^{\bp[i,j,i]}r_j^{\bp[i,j,i]} = r_i^{\bp[i,j]}r_j^{\bp[i]}r_i^{\bp[i,j]}r_k^{\bp[i]}r_i^{\bp[i,j]}r_j^{\bp[i]}r_i^{\bp[i,j]}$$
	$$ = r_j^{\bp[i]}r_i^{\bp[i]}r_j^{\bp[i]}r_i^{\bp[i]}r_j^{\bp[i]}r_k^{\bp[i]}r_j^{\bp[i]}r_i^{\bp[i]}r_j^{\bp[i]}r_i^{\bp[i]}r_j^{\bp[i]} = r_i^{\bp}r_j^{\bp}r_i^{\bp}r_j^{\bp}r_i^{\bp}r_j^{\bp}r_k^{\bp}r_j^{\bp}r_i^{\bp}r_j^{\bp}r_i^{\bp}r_j^{\bp}r_i^{\bp}.$$
	
	Thus
	$\pi(r_k^{\bp[i,j,i,j,i]}) = \pi(r_i^\bp r_j^\bp)^3\pi(r_k^{\bp})\pi(r_j^\bp r_i^\bp)^3.$

	\noindent \underline{Case $(b_{ij}^{\bp}c_j^{\bp},b_{ji}^{\bp}c_i^{\bp})=(-,-)$} The values we are concerned with are as follows:
	$$
	(b_{ki}^{\bp}c_i^{\bp}, \ b_{kj}^{\bp[i]}c_j^{\bp[i]}, \ b_{ki}^{\bp[i,j]}c_i^{\bp[i,j]}, \ b_{kj}^{\bp[i,j,i]}c_j^{\bp[i,j,i]}, \ b_{ki}^{\bp[i,j,i,j]}c_i^{\bp[i,j,i,j]}) = (-,0,-,-,0).
	$$
	
	Hence 
	$r_k^{\bp[i,j,i,j,i]} = r_k^{\bp},$
	and 
	$\pi(r_k^{\bp[i,j,i,j,i]}) = \pi(r_k^\bp).$
	
	\noindent \underline{Case $(b_{ij}^{\bp}c_j^{\bp},b_{ji}^{\bp}c_i^{\bp},b_{ij}^{\bp[i]}c_j^{\bp[i]})=(+,+,-)$} The values we are concerned with are as follows:
	$$
	(b_{ki}^{\bp}c_i^{\bp}, \ b_{kj}^{\bp[i]}c_j^{\bp[i]}, \ b_{ki}^{\bp[i,j]}c_i^{\bp[i,j]}, \ b_{kj}^{\bp[i,j,i]}c_j^{\bp[i,j,i]}, \ b_{ki}^{\bp[i,j,i,j]}c_i^{\bp[i,j,i,j]}) = (+,0,+,-,0).
	$$
	
	Hence
	$$r_k^{\bp[i,j,i,j,i]} = r_k^{\bp[i,j,i]} = r_i^{\bp[i,j]}r_k^{\bp[i,j]}r_i^{\bp[i,j]} = r_k^\bp.$$
	Thus
	$\pi(r_k^{\bp[i,j,i,j,i]}) = \pi(r_k^{\bp}).$

	\noindent \underline{Case $(b_{ij}^{\bp}c_j^{\bp},b_{ji}^{\bp}c_i^{\bp},b_{ij}^{\bp[i,j,i,j]}c_j^{\bp[i,j,i,j]})=(+,+,-)$} The values we are concerned with are as follows:
	$$
	(b_{ki}^{\bp}c_i^{\bp}, \ b_{kj}^{\bp[i]}c_j^{\bp[i]}, \ b_{ki}^{\bp[i,j]}c_i^{\bp[i,j]}, \ b_{kj}^{\bp[i,j,i]}c_j^{\bp[i,j,i]}, \ b_{ki}^{\bp[i,j,i,j]}c_i^{\bp[i,j,i,j]}) = (+,0,+,+,0).
	$$
	
	Hence
	$$r_k^{\bp[i,j,i,j,i]} = r_k^{\bp[i,j,i,j]} = r_j^{\bp[i,j,i]}r_k^{\bp[i,j,i]}r_j^{\bp[i,j,i]} = r_j^{\bp[i,j]}r_i^{\bp[i,j]}r_k^{\bp[i,j]}r_i^{\bp[i,j]}r_j^{\bp[i,j]} $$
	$$ = r_j^{\bp[i]}r_j^{\bp[i]}r_i^{\bp[i]}r_j^{\bp[i]}r_k^{\bp[i]}r_j^{\bp[i]}r_i^{\bp[i]}r_j^{\bp[i]}r_j^{\bp[i]} = r_i^{\bp[i]}r_j^{\bp[i]}r_k^{\bp[i]}r_j^{\bp[i]}r_i^{\bp[i]}$$
	$$ = r_i^{\bp}r_i^{\bp}r_j^{\bp}r_i^{\bp}r_i^{\bp}r_k^{\bp}r_i^{\bp}r_i^{\bp}r_j^{\bp}r_i^{\bp}r_i^{\bp} = r_j^{\bp}r_k^{\bp}r_j^{\bp}.$$	
	Therefore
	$\pi(r_k^{\bp[i,j,i,j,i]}) = \pi(r_j^\bp r_k^\bp)^2\pi(r_k^{\bp}).$	
\end{proof}

\begin{proof}[\boxed{Proof\ of\ Lemma\ \ref{lem-case-4}}]
	We need to keep track of $b_{ki}^\bp, b_{kj}^\bp, b_{ij}^\bp, b_{ji}^\bp, c_i^\bp,$ and $c_j^\bp$ throughout the entire mutation $[i,j,i,j,i]$.
	Let us first list all the possible values of the first four variables.
	{\tiny $$
		(...,b_{ki}^{\bp[i,j,i,j]})=(0,0,0,0,-);\, (...,b_{kj}^{\bp[i,j,i,j]})=(+,+,-,-,+);\, (...,b_{ij}^{\bp[i,j,i,j]})=(-,+,-,+,-);\, (...,b_{ji}^{\bp[i,j,i,j]})=(+,-,+,-,+)
		$$}
	
	
	\noindent \underline{Case $(b_{ij}^{\bp}c_j^{\bp},b_{ji}^{\bp}c_i^{\bp})=(+,-)$} The values we are concerned with are as follows:
	$$
	(b_{ki}^{\bp}c_i^{\bp}, \ b_{kj}^{\bp[i]}c_j^{\bp[i]}, \ b_{ki}^{\bp[i,j]}c_i^{\bp[i,j]}, \ b_{kj}^{\bp[i,j,i]}c_j^{\bp[i,j,i]}, \ b_{ki}^{\bp[i,j,i,j]}c_i^{\bp[i,j,i,j]}) = (0,-,0,-,-).
	$$
	
	Hence 
	$r_k^{\bp[i,j,i,j,i]} = r_k^{\bp},$
	and 
	$\pi(r_k^{\bp[i,j,i,j,i]}) = \pi(r_k^{\bp}).$
	
	\noindent \underline{Case $(b_{ij}^{\bp}c_j^{\bp},b_{ji}^{\bp}c_i^{\bp})=(-,+)$} The values we are concerned with are as follows:
	$$
	(b_{ki}^{\bp}c_i^{\bp}, \ b_{kj}^{\bp[i]}c_j^{\bp[i]}, \ b_{ki}^{\bp[i,j]}c_i^{\bp[i,j]}, \ b_{kj}^{\bp[i,j,i]}c_j^{\bp[i,j,i]}, \ b_{ki}^{\bp[i,j,i,j]}c_i^{\bp[i,j,i,j]}) = (0,+,0,+,+).
	$$
	
	Hence
	$$r_k^{\bp[i,j,i,j,i]} = r_i^{\bp[i,j,i,j]}r_k^{\bp[i,j,i,j]}r_i^{\bp[i,j,i,j]} = r_i^{\bp[i,j]}r_j^{\bp[i,j,i]}r_k^{\bp[i,j]}r_j^{\bp[i,j,i]}r_i^{\bp[i,j]}$$ 
	$$= r_j^{\bp[i]}r_i^{\bp[i,j]}r_k^{\bp[i,j]}r_i^{\bp[i,j]}r_j^{\bp[i]} = r_k^{\bp}.$$
	
	Thus
	$\pi(r_k^{\bp[i,j,i,j,i]}) = \pi(r_k^{\bp}).$

	\noindent \underline{Case $(b_{ij}^{\bp}c_j^{\bp},b_{ji}^{\bp}c_i^{\bp})=(-,-)$} The values we are concerned with are as follows:
	$$
	(b_{ki}^{\bp}c_i^{\bp}, \ b_{kj}^{\bp[i]}c_j^{\bp[i]}, \ b_{ki}^{\bp[i,j]}c_i^{\bp[i,j]}, \ b_{kj}^{\bp[i,j,i]}c_j^{\bp[i,j,i]}, \ b_{ki}^{\bp[i,j,i,j]}c_i^{\bp[i,j,i,j]}) = (0,+,0,-,+).
	$$
	
	Hence 
	$$r_k^{\bp[i,j,i,j,i]} = r_i^{\bp[i,j,i,j]}r_k^{\bp[i,j,i,j]}r_i^{\bp[i,j,i,j]} = r_j^{\bp[i,j,i]}r_i^{\bp[i,j,i]}r_j^{\bp[i,j,i]}r_k^{\bp[i,j,i]}r_j^{\bp[i,j,i]}r_i^{\bp[i,j,i]}r_j^{\bp[i,j,i]}$$
	$$= r_i^{\bp[i,j]}r_j^{\bp[i]}r_i^{\bp[i,j]}r_j^{\bp[i]}r_i^{\bp[i,j]}r_k^{\bp[i,j]}r_i^{\bp[i,j]}r_j^{\bp[i]}r_i^{\bp[i,j]}r_j^{\bp[i]}r_i^{\bp[i,j]} = r_j^{\bp[i]}r_i^{\bp}r_j^{\bp[i]}r_i^{\bp}r_j^{\bp[i]}r_i^{\bp}r_k^{\bp}r_i^{\bp}r_j^{\bp[i]}r_i^{\bp}r_j^{\bp[i]}r_i^{\bp}r_j^{\bp[i]}$$
	$$= r_j^{\bp}r_i^{\bp}r_j^{\bp}r_i^{\bp}r_j^{\bp}r_i^{\bp}r_k^{\bp}r_i^{\bp}r_j^{\bp}r_i^{\bp}r_j^{\bp}r_i^{\bp}r_j^{\bp},$$
	and 
	$$\pi(r_k^{\bp[i,j,i,j,i]}) = \pi(r_j^\bp r_i^\bp)^3\pi(r_k^\bp)\pi(r_i^\bp r_j^\bp)^3.$$ 
	
	\noindent \underline{Case $(b_{ij}^{\bp}c_j^{\bp},b_{ji}^{\bp}c_i^{\bp},b_{ij}^{\bp[i]}c_j^{\bp[i]})=(+,+,-)$} The values we are concerned with are as follows:
	$$
	(b_{ki}^{\bp}c_i^{\bp}, \ b_{kj}^{\bp[i]}c_j^{\bp[i]}, \ b_{ki}^{\bp[i,j]}c_i^{\bp[i,j]}, \ b_{kj}^{\bp[i,j,i]}c_j^{\bp[i,j,i]}, \ b_{ki}^{\bp[i,j,i,j]}c_i^{\bp[i,j,i,j]}) = (0,-,0,-,-).
	$$
	
	Hence
	$r_k^{\bp[i,j,i,j,i]} = r_k^{\bp}.$
	Thus
	$\pi(r_k^{\bp[i,j,i,j,i]}) = \pi(r_k^{\bp}).$

	\noindent \underline{Case $(b_{ij}^{\bp}c_j^{\bp},b_{ji}^{\bp}c_i^{\bp},b_{ij}^{\bp[i,j,i,j]}c_j^{\bp[i,j,i,j]})=(+,+,-)$} The values we are concerned with are as follows:
	$$
	(b_{ki}^{\bp}c_i^{\bp}, \ b_{kj}^{\bp[i]}c_j^{\bp[i]}, \ b_{ki}^{\bp[i,j]}c_i^{\bp[i,j]}, \ b_{kj}^{\bp[i,j,i]}c_j^{\bp[i,j,i]}, \ b_{ki}^{\bp[i,j,i,j]}c_i^{\bp[i,j,i,j]}) = (0,+,0,+,-).
	$$
	
	Hence
	$$r_k^{\bp[i,j,i,j,i]} = r_k^{\bp[i,j,i,j]} = r_j^{\bp[i,j,i]}r_k^{\bp[i,j,i]}r_j^{\bp[i,j,i]} = r_j^{\bp[i]}r_k^{\bp[i,j]}r_j^{\bp[i]} = r_k^\bp.$$
	
	Therefore
	$\pi(r_k^{\bp[i,j,i,j,i]}) = \pi(r_k^{\bp}).$
	
\end{proof}

\begin{proof}[\boxed{Proof\ of\ Lemma\ \ref{lem-case-7}}]
	We need to keep track of $b_{ki}^\bp, b_{kj}^\bp, b_{ij}^\bp, b_{ji}^\bp, c_i^\bp,$ and $c_j^\bp$ throughout the entire mutation $[i,j,i,j,i]$.
	Let us first list all the possible values of the first four variables.
	{\tiny $$
		(...,b_{ki}^{\bp[i,j,i,j]})=(+,-,0,0,0);\, (...,b_{kj}^{\bp[i,j,i,j]})=(0,+,-,-,+);\, (...,b_{ij}^{\bp[i,j,i,j]})=(+,-,+,-,+);\, (...,b_{ji}^{\bp[i,j,i,j]})=(-,+,-,+,-)
		$$}
	
	\noindent \underline{Case $(b_{ij}^{\bp}c_j^{\bp},b_{ji}^{\bp}c_i^{\bp})=(+,-)$} The values we are concerned with are as follows:
	$$
	(b_{ki}^{\bp}c_i^{\bp}, \ b_{kj}^{\bp[i]}c_j^{\bp[i]}, \ b_{ki}^{\bp[i,j]}c_i^{\bp[i,j]}, \ b_{kj}^{\bp[i,j,i]}c_j^{\bp[i,j,i]}, \ b_{ki}^{\bp[i,j,i,j]}c_i^{\bp[i,j,i,j]}) = (+,+,0,+,0).
	$$
	
	Hence 
	$$r_k^{\bp[i,j,i,j,i]} = r_k^{\bp[i,j,i,j]} = r_j^{\bp[i,j,i]}r_k^{\bp[i,j,i]}r_j^{\bp[i,j,i]} = r_j^{\bp[i,j,i]}r_k^{\bp[i,j]}r_j^{\bp[i,j,i]}$$
	$$= r_j^{\bp[i,j,i]}r_j^{\bp[i]}r_k^{\bp[i]}r_j^{\bp[i]}r_j^{\bp[i,j,i]} = r_j^{\bp[i,j,i]}r_j^{\bp}r_i^{\bp}r_k^{\bp}r_i^{\bp}r_j^{\bp}r_j^{\bp[i,j,i]}.$$
	
	Additionally, 
	$$r_j^{\bp[i,j,i]} = r_i^{\bp[i,j]}r_j^{\bp}r_i^{\bp[i,j]} = r_i^{\bp}r_j^{\bp}r_i^{\bp}.$$
	
	Thus
	$$\pi(r_k^{\bp[i,j,i,j,i]}) = \pi(r_i^{\bp}r_j^{\bp})^3\pi(r_j^{\bp}r_k^{\bp}r_j^{\bp})\pi(r_j^{\bp}r_i^{\bp})^3 = \pi(r_i^{\bp}r_j^{\bp})^3\pi(r_j^{\bp}r_k^{\bp})^2 \pi(r_k^{\bp})\pi(r_j^{\bp}r_i^{\bp})^3.$$ 
	
	\noindent \underline{Case $(b_{ij}^{\bp}c_j^{\bp},b_{ji}^{\bp}c_i^{\bp})=(-,+)$} The values we are concerned with are as follows:
	$$
	(b_{ki}^{\bp}c_i^{\bp}, \ b_{kj}^{\bp[i]}c_j^{\bp[i]}, \ b_{ki}^{\bp[i,j]}c_i^{\bp[i,j]}, \ b_{kj}^{\bp[i,j,i]}c_j^{\bp[i,j,i]}, \ b_{ki}^{\bp[i,j,i,j]}c_i^{\bp[i,j,i,j]}) = (-,-,0,-,0).
	$$
	
	Hence
	$r_k^{\bp[i,j,i,j,i]} = r_k^\bp.$
	Therefore
	$\pi(r_k^{\bp[i,j,i,j,i]}) = \pi(r_k^{\bp}).$
	
	\noindent \underline{Case $(b_{ij}^{\bp}c_j^{\bp},b_{ji}^{\bp}c_i^{\bp})=(-,-)$} The values we are concerned with are as follows:
	$$
	(b_{ki}^{\bp}c_i^{\bp}, \ b_{kj}^{\bp[i]}c_j^{\bp[i]}, \ b_{ki}^{\bp[i,j]}c_i^{\bp[i,j]}, \ b_{kj}^{\bp[i,j,i]}c_j^{\bp[i,j,i]}, \ b_{ki}^{\bp[i,j,i,j]}c_i^{\bp[i,j,i,j]}) = (+,-,0,+,0).
	$$
	
	Hence 
	$$r_k^{\bp[i,j,i,j,i]} = r_k^{\bp[i,j,i,j]} = r_j^{\bp[i,j,i]}r_k^{\bp[i,j,i]}r_j^{\bp[i,j,i]} = r_j^{\bp[i,j,i]}r_k^{\bp[i]}r_j^{\bp[i,j,i]} = r_j^{\bp[i,j,i]}r_i^{\bp}r_k^{\bp}r_i^{\bp}r_j^{\bp[i,j,i]}.$$
	
	Additionally,
	$$r_j^{\bp[i,j,i]} = r_i^{\bp[i,j]}r_j^{\bp[i]}r_i^{\bp[i,j]} = r_j^{\bp}r_i^{\bp}r_j^{\bp}r_i^{\bp}r_j^{\bp}.$$
	
	Thus
	$\pi(r_k^{\bp[i,j,i,j,i]}) = \pi(r_j^\bp r_i^\bp)^3\pi(r_k^\bp)\pi(r_i^\bp r_j^\bp)^3.$
	
	\noindent \underline{Case $(b_{ij}^{\bp}c_j^{\bp},b_{ji}^{\bp}c_i^{\bp},b_{ij}^{\bp[i]}c_j^{\bp[i]})=(+,+,-)$} The values we are concerned with are as follows:
	$$
	(b_{ki}^{\bp}c_i^{\bp}, \ b_{kj}^{\bp[i]}c_j^{\bp[i]}, \ b_{ki}^{\bp[i,j]}c_i^{\bp[i,j]}, \ b_{kj}^{\bp[i,j,i]}c_j^{\bp[i,j,i]}, \ b_{ki}^{\bp[i,j,i,j]}c_i^{\bp[i,j,i,j]}) = (-,+,0,+,0).
	$$
	
	Hence
	$$r_k^{\bp[i,j,i,j,i]} = r_k^{\bp[i,j,i,j]} = r_j^{\bp[i,j,i]}r_k^{\bp[i,j,i]}r_j^{\bp[i,j,i]} = r_j^{\bp[i,j,i]}r_k^{\bp[i,j]}r_j^{\bp[i,j,i]}$$
	$$= r_j^{\bp[i,j,i]}r_j^{\bp[i]}r_k^{\bp[i]}r_j^{\bp[i]}r_j^{\bp[i,j,i]} = r_j^{\bp[i,j,i]}r_j^{\bp[i]}r_k^{\bp}r_j^{\bp[i]}r_j^{\bp[i,j,i]}.$$
	
	Additionally,
	$$r_j^{\bp[i,j,i]} = r_j^{\bp[i,j]} = r_j^{\bp[i]}.$$ 
	
	Thus
	$\pi(r_k^{\bp[i,j,i,j,i]}) = \pi(r_j^{\bp[i]}r_j^{\bp[i]}r_k^{\bp}r_j^{\bp[i]}r_j^{\bp[i]}) = \pi(r_k^{\bp}).$

	\noindent \underline{Case $(b_{ij}^{\bp}c_j^{\bp},b_{ji}^{\bp}c_i^{\bp},b_{ij}^{\bp[i,j,i,j]}c_j^{\bp[i,j,i,j]})=(+,+,-)$} The values we are concerned with are as follows:
	$$
	(b_{ki}^{\bp}c_i^{\bp}, \ b_{kj}^{\bp[i]}c_j^{\bp[i]}, \ b_{ki}^{\bp[i,j]}c_i^{\bp[i,j]}, \ b_{kj}^{\bp[i,j,i]}c_j^{\bp[i,j,i]}, \ b_{ki}^{\bp[i,j,i,j]}c_i^{\bp[i,j,i,j]}) = (-,+,0,-,0).
	$$
	
	Hence
	$r_k^{\bp[i,j,i,j,i]} = r_k^\bp.$
	Therefore
	$\pi(r_k^{\bp[i,j,i,j,i]}) = \pi(r_k^{\bp}).$
\end{proof}

\begin{proof}[\boxed{Proof\ of\ Lemma\ \ref{lem-case-14}}]
	We need to keep track of $b_{ki}^\bp, b_{kj}^\bp, b_{ij}^\bp, b_{ji}^\bp, c_i^\bp,$ and $c_j^\bp$ throughout the entire mutation $[i,j,i,j,i]$.
	Let us first list all the possible values of the first four variables.
	{\tiny $$
		(...,b_{ki}^{\bp[i,j,i,j]})=(-,+,+,-,-);\, (...,b_{kj}^{\bp[i,j,i,j]})=(+,0,0,0,0);\, (...,b_{ij}^{\bp[i,j,i,j]})=(-,+,-,+,-);\, (...,b_{ji}^{\bp[i,j,i,j]})=(+,-,+,-,+)
		$$}
	
	\noindent \underline{Case $(b_{ij}^{\bp}c_j^{\bp},b_{ji}^{\bp}c_i^{\bp})=(+,-)$} The values we are concerned with are as follows:
	$$
	(b_{ki}^{\bp}c_i^{\bp}, \ b_{kj}^{\bp[i]}c_j^{\bp[i]}, \ b_{ki}^{\bp[i,j]}c_i^{\bp[i,j]}, \ b_{kj}^{\bp[i,j,i]}c_j^{\bp[i,j,i]}, \ b_{ki}^{\bp[i,j,i,j]}c_i^{\bp[i,j,i,j]}) = (+,0,+,0,-).
	$$
	
	Hence 
	$r_k^{\bp[i,j,i,j,i]} = r_i^{\bp[i,j]}r_k^{\bp[i,j]}r_i^{\bp[i,j]} = r_i^{\bp}r_k^{\bp[i]}r_i^{\bp} = r_k^{\bp},$
	and 
	$\pi(r_k^{\bp[i,j,i,j,i]}) = \pi(r_k^{\bp}).$
	
	\noindent \underline{Case $(b_{ij}^{\bp}c_j^{\bp},b_{ji}^{\bp}c_i^{\bp})=(-,+)$} The values we are concerned with are as follows:
	$$
	(b_{ki}^{\bp}c_i^{\bp}, \ b_{kj}^{\bp[i]}c_j^{\bp[i]}, \ b_{ki}^{\bp[i,j]}c_i^{\bp[i,j]}, \ b_{kj}^{\bp[i,j,i]}c_j^{\bp[i,j,i]}, \ b_{ki}^{\bp[i,j,i,j]}c_i^{\bp[i,j,i,j]}) = (-,0,+,0,+).
	$$
	
	Hence
	$$r_k^{\bp[i,j,i,j,i]} = r_i^{\bp[i,j,i,j]}r_k^{\bp[i,j,i,j]}r_i^{\bp[i,j,i,j]} = r_i^{\bp[i,j,i]}r_k^{\bp[i,j,i]}r_i^{\bp[i,j,i]} = r_k^{\bp[i,j]} = r_k^{\bp}.$$
	
	Thus
	$\pi(r_k^{\bp[i,j,i,j,i]}) = \pi(r_k^{\bp}).$

	\noindent \underline{Case $(b_{ij}^{\bp}c_j^{\bp},b_{ji}^{\bp}c_i^{\bp})=(-,-)$} The values we are concerned with are as follows:
	$$
	(b_{ki}^{\bp}c_i^{\bp}, \ b_{kj}^{\bp[i]}c_j^{\bp[i]}, \ b_{ki}^{\bp[i,j]}c_i^{\bp[i,j]}, \ b_{kj}^{\bp[i,j,i]}c_j^{\bp[i,j,i]}, \ b_{ki}^{\bp[i,j,i,j]}c_i^{\bp[i,j,i,j]}) = (+,0,+,0,+).
	$$
	
	Hence 
	$$r_k^{\bp[i,j,i,j,i]} = r_i^{\bp[i,j,i,j]}r_k^{\bp[i,j,i]}r_i^{\bp[i,j,i,j]} = r_j^{\bp[i,j,i]}r_i^{\bp[i,j,i]}r_j^{\bp[i,j,i]}r_k^{\bp[i,j,i]}r_j^{\bp[i,j,i]}r_i^{\bp[i,j,i]}r_j^{\bp[i,j,i]}$$
	$$ r_i^{\bp[i,j]}r_j^{\bp[i]}r_i^{\bp[i,j]}r_j^{\bp[i]}r_k^{\bp[i]}r_j^{\bp[i]}r_i^{\bp[i,j]}r_j^{\bp[i]}r_i^{\bp[i,j]} = r_j^{\bp[i]}r_i^{\bp[i]}r_j^{\bp[i]}r_i^{\bp[i]}r_k^{\bp[i]}r_i^{\bp[i]}r_j^{\bp[i]}r_i^{\bp[i]}r_j^{\bp[i]}$$
	$$= r_j^{\bp}r_i^{\bp}r_j^{\bp}r_k^{\bp}r_j^{\bp}r_i^{\bp}r_j^{\bp},$$
	and, by premultiplying by $1$,
	$$= r_k^{\bp}r_k^{\bp}r_j^{\bp}r_i^{\bp}r_j^{\bp}r_k^{\bp}r_j^{\bp}r_i^{\bp}r_j^{\bp}.$$
	Therefore,
	$$\pi(r_k^{\bp[i,j,i,j,i]}) = \pi(r_k^{\bp})\pi(r_k^{\bp}r_j^{\bp}r_i^{\bp}r_j^{\bp})^2.$$

	\noindent \underline{Case $(b_{ij}^{\bp}c_j^{\bp},b_{ji}^{\bp}c_i^{\bp},b_{ij}^{\bp[i]}c_j^{\bp[i]})=(+,+,-)$} The values we are concerned with are as follows:
	$$
	(b_{ki}^{\bp}c_i^{\bp}, \ b_{kj}^{\bp[i]}c_j^{\bp[i]}, \ b_{ki}^{\bp[i,j]}c_i^{\bp[i,j]}, \ b_{kj}^{\bp[i,j,i]}c_j^{\bp[i,j,i]}, \ b_{ki}^{\bp[i,j,i,j]}c_i^{\bp[i,j,i,j]}) = (-,0,-,0,-).
	$$
	
	Hence
	$r_k^{\bp[i,j,i,j,i]} = r_k^{\bp[i,j,i,j]} = r_k^{\bp}.$
	Thus
	$\pi(r_k^{\bp[i,j,i,j,i]}) = \pi(r_k^{\bp}).$

	\noindent \underline{Case $(b_{ij}^{\bp}c_j^{\bp},b_{ji}^{\bp}c_i^{\bp},b_{ij}^{\bp[i,j,i,j]}c_j^{\bp[i,j,i,j]})=(+,+,-)$} The values we are concerned with are as follows:
	$$
	(b_{ki}^{\bp}c_i^{\bp}, \ b_{kj}^{\bp[i]}c_j^{\bp[i]}, \ b_{ki}^{\bp[i,j]}c_i^{\bp[i,j]}, \ b_{kj}^{\bp[i,j,i]}c_j^{\bp[i,j,i]}, \ b_{ki}^{\bp[i,j,i,j]}c_i^{\bp[i,j,i,j]}) = (-,0,-,0,-).
	$$
	
	Hence
	$r_k^{\bp[i,j,i,j,i]} = r_k^{\bp}.$
	Therefore
	$\pi(r_k^{\bp[i,j,i,j,i]}) = \pi(r_k^{\bp}).$
\end{proof}

\begin{proof}[\boxed{Proof\ of\ Lemma\ \ref{lem-case-15}}]
	We need to keep track of $b_{ki}^\bp, b_{kj}^\bp, b_{ij}^\bp, b_{ji}^\bp, c_i^\bp,$ and $c_j^\bp$ throughout the entire mutation $[i,j,i,j,i]$.
	Let us first list all the possible values of the first four variables.
	{\tiny $$
		(...,b_{ki}^{\bp[i,j,i,j]})=(+,-,-,+,+);\, (...,b_{kj}^{\bp[i,j,i,j]})=(-,0,0,0,0);\, (...,b_{ij}^{\bp[i,j,i,j]})=(+,-,+,-,+);\, (...,b_{ji}^{\bp[i,j,i,j]})=(-,+,-,+,-)
		$$}
	
	\noindent \underline{Case $(b_{ij}^{\bp}c_j^{\bp},b_{ji}^{\bp}c_i^{\bp})=(+,-)$} The values we are concerned with are as follows:
	$$
	(b_{ki}^{\bp}c_i^{\bp}, \ b_{kj}^{\bp[i]}c_j^{\bp[i]}, \ b_{ki}^{\bp[i,j]}c_i^{\bp[i,j]}, \ b_{kj}^{\bp[i,j,i]}c_j^{\bp[i,j,i]}, \ b_{ki}^{\bp[i,j,i,j]}c_i^{\bp[i,j,i,j]}) = (+,0,+,0,-).
	$$
	
	Hence 
	$r_k^{\bp[i,j,i,j,i]} = r_i^{\bp[i,j]}r_k^{\bp[i,j]}r_i^{\bp[i,j]} = r_i^{\bp}r_k^{\bp[i]}r_i^{\bp} = r_k^{\bp},$
	and 
	$\pi(r_k^{\bp[i,j,i,j,i]}) = \pi(r_k^{\bp}).$
	
	\noindent \underline{Case $(b_{ij}^{\bp}c_j^{\bp},b_{ji}^{\bp}c_i^{\bp})=(-,+)$} The values we are concerned with are as follows:
	$$
	(b_{ki}^{\bp}c_i^{\bp}, \ b_{kj}^{\bp[i]}c_j^{\bp[i]}, \ b_{ki}^{\bp[i,j]}c_i^{\bp[i,j]}, \ b_{kj}^{\bp[i,j,i]}c_j^{\bp[i,j,i]}, \ b_{ki}^{\bp[i,j,i,j]}c_i^{\bp[i,j,i,j]}) = (-,0,+,0,+).
	$$
	
	Hence
	$$r_k^{\bp[i,j,i,j,i]} = r_i^{\bp[i,j,i,j]}r_k^{\bp[i,j,i]}r_i^{\bp[i,j,i,j]} = r_i^{\bp[i,j]}r_k^{\bp[i,j,i]}r_i^{\bp[i,j]} = r_k^{\bp[i,j]} = r_k^{\bp}.$$
	
	Thus
	$\pi(r_k^{\bp[i,j,i,j,i]}) = \pi(r_k^{\bp}).$
	
	\noindent \underline{Case $(b_{ij}^{\bp}c_j^{\bp},b_{ji}^{\bp}c_i^{\bp})=(-,-)$} The values we are concerned with are as follows:
	$$
	(b_{ki}^{\bp}c_i^{\bp}, \ b_{kj}^{\bp[i]}c_j^{\bp[i]}, \ b_{ki}^{\bp[i,j]}c_i^{\bp[i,j]}, \ b_{kj}^{\bp[i,j,i]}c_j^{\bp[i,j,i]}, \ b_{ki}^{\bp[i,j,i,j]}c_i^{\bp[i,j,i,j]}) = (+,0,+,0,+).
	$$
	
	Hence 
	$$r_k^{\bp[i,j,i,j,i]} = r_i^{\bp[i,j,i,j]}r_k^{\bp[i,j,i]}r_i^{\bp[i,j,i,j]} = r_j^{\bp[i,j,i]}r_i^{\bp[i,j,i]}r_j^{\bp[i,j,i]}r_k^{\bp[i,j,i]}r_j^{\bp[i,j,i]}r_i^{\bp[i,j,i]}r_j^{\bp[i,j,i]}$$
	$$ r_i^{\bp[i,j]}r_j^{\bp[i]}r_i^{\bp[i,j]}r_j^{\bp[i]}r_k^{\bp[i]}r_j^{\bp[i]}r_i^{\bp[i,j]}r_j^{\bp[i]}r_i^{\bp[i,j]} = r_j^{\bp[i]}r_i^{\bp[i]}r_j^{\bp[i]}r_i^{\bp[i]}r_k^{\bp[i]}r_i^{\bp[i]}r_j^{\bp[i]}r_i^{\bp[i]}r_j^{\bp[i]}$$
	$$= r_j^{\bp}r_i^{\bp}r_j^{\bp}r_k^{\bp}r_j^{\bp}r_i^{\bp}r_j^{\bp},$$
	and 
	$$\pi(r_k^{\bp[i,j,i,j,i]}) = \pi(r_k^{\bp})\pi(r_k^{\bp}r_j^{\bp}r_i^{\bp}r_j^{\bp})^2.$$

	\noindent \underline{Case $(b_{ij}^{\bp}c_j^{\bp},b_{ji}^{\bp}c_i^{\bp},b_{ij}^{\bp[i]}c_j^{\bp[i]})=(+,+,-)$} The values we are concerned with are as follows:
	$$
	(b_{ki}^{\bp}c_i^{\bp}, \ b_{kj}^{\bp[i]}c_j^{\bp[i]}, \ b_{ki}^{\bp[i,j]}c_i^{\bp[i,j]}, \ b_{kj}^{\bp[i,j,i]}c_j^{\bp[i,j,i]}, \ b_{ki}^{\bp[i,j,i,j]}c_i^{\bp[i,j,i,j]}) = (-,0,-,0,-).
	$$
	
	Hence
	$r_k^{\bp[i,j,i,j,i]} = r_k^{\bp[i,j,i,j]} = r_k^{\bp}.$
	Thus
	$\pi(r_k^{\bp[i,j,i,j,i]}) = \pi(r_k^{\bp}).$

	\noindent \underline{Case $(b_{ij}^{\bp}c_j^{\bp},b_{ji}^{\bp}c_i^{\bp},b_{ij}^{\bp[i,j,i,j]}c_j^{\bp[i,j,i,j]})=(+,+,-)$} The values we are concerned with are as follows:
	$$
	(b_{ki}^{\bp}c_i^{\bp}, \ b_{kj}^{\bp[i]}c_j^{\bp[i]}, \ b_{ki}^{\bp[i,j]}c_i^{\bp[i,j]}, \ b_{kj}^{\bp[i,j,i]}c_j^{\bp[i,j,i]}, \ b_{ki}^{\bp[i,j,i,j]}c_i^{\bp[i,j,i,j]}) = (-,0,-,0,-).
	$$
	
	Hence
	$r_k^{\bp[i,j,i,j,i]} = r_k^{\bp}.$
	Therefore
	$\pi(r_k^{\bp[i,j,i,j,i]}) = \pi(r_k^{\bp}).$
\end{proof}

\end{document}